\documentclass[11pt,leqno]{amsart}
\usepackage{amsthm,amsfonts,amssymb,amsmath,oldgerm}
\numberwithin{equation}{section}
\usepackage{fullpage}
\usepackage{color}
\usepackage{calrsfs}
\usepackage{bbm}
\DeclareMathAlphabet{\pazocal}{OMS}{zplm}{m}{n}

\def\eps{\varepsilon }

\newcommand\R{\mathbb R}

\def\eps{\varepsilon}

\newcommand\br{\begin{remark}}
\newcommand\er{\end{remark}}
\newcommand\bp{\begin{pmatrix}}
\newcommand\ep{\end{pmatrix}}
\newcommand{\be}{\begin{equation}}
\newcommand{\ee}{\end{equation}}
\newcommand\ba{\begin{equation}\begin{aligned}}
\newcommand\ea{\end{aligned}\end{equation}}

\newcommand{\bap}{\begin{app}}
\newcommand{\eap}{\end{app}}
\newcommand{\begs}{\begin{exams}}
\newcommand{\eegs}{\end{exams}}
\newcommand{\beg}{\begin{example}}
\newcommand{\eeg}{\end{exaplem}}
\newcommand{\bpr}{\begin{proposition}}
\newcommand{\epr}{\end{proposition}}
\newcommand{\bt}{\begin{theorem}}
\newcommand{\et}{\end{theorem}}
\newcommand{\bc}{\begin{corollary}}
\newcommand{\ec}{\end{corollary}}
\newcommand{\bl}{\begin{lemma}}
\newcommand{\el}{\end{lemma}}
\newcommand{\bd}{\begin{definition}}
\newcommand{\ed}{\end{definition}}
\newcommand{\brs}{\begin{remarks}}
\newcommand{\ers}{\end{remarks}}

\newcommand{\RR}{{\mathbb R}}

\newcommand{\NN}{{\mathbb N}}

\newcommand{\pa}{{\partial}}
\newcommand{\rmi}{{\mathrm{i}}}
\newcommand{\rmd}{{\mathrm{d}}}
\newcommand{\rms}{{\mathrm{s}}}
\newcommand{\rmu}{{\mathrm{u}}}
\newcommand{\rmc}{{\mathrm{c}}}
\newcommand{\rmh}{{\mathrm{h}}}
\newcommand{\rmone}{{\mathbbm{1}}}

\newcommand{\const}{\text{\rm constant}}
\newcommand{\Id}{{\rm Id }}
\newcommand{\im}{{\rm im }}

\newcommand{\sgn}{\text{\rm sgn}}
\newtheorem{theorem}{Theorem}[section]
\newtheorem{proposition}[theorem]{Proposition}
\newtheorem{corollary}[theorem]{Corollary}
\newtheorem{lemma}[theorem]{Lemma}

\theoremstyle{remark}
\newtheorem{remark}[theorem]{Remark}
\theoremstyle{definition}
\newtheorem{definition}[theorem]{Definition}

\newtheorem{example}[theorem]{Example}

\newcommand\cB{{\mathcal B}}
\newcommand\cD{{\mathcal D}}

\newcommand\cJ{{\mathcal J}}
\newcommand\cV{{\mathcal V}}
\newcommand\cW{{\mathcal W}}

\newcommand\cR{{\mathcal R}}
\newcommand\cG{{\mathcal G}}
\newcommand\cK{{\mathcal K}}
\newcommand\cL{{\mathcal L}}

\newcommand\cE{{\mathcal E}}
\newcommand\cF{{\mathcal F}}

\newcommand\cY{{\mathcal Y}}
\newcommand\cM{{\mathcal M}}
\newcommand\cN{{\mathcal N}}
\newcommand\cT{{\mathcal T}}
\newcommand\cS{{\mathcal S}}
\newcommand\cZ{{\mathcal Z}}

\newcommand\bH{{\mathbb H}}

\newcommand\bV{{\mathbb V}}

\newcommand\bX{{\mathbb X}}
\newcommand\bY{{\mathbb Y}}

\newcommand\bZ{{\mathbb Z}}

\newcommand{\tT}{{\widetilde{T}}}

\newcommand{\wti}{\widetilde}

\newcommand{\tv}{\wti{v}}
\newcommand{\tu}{\wti{u}}
\newcommand{\tA}{\wti{A}}
\newcommand{\tP}{\wti{P}}
\newcommand{\tbV}{\wti{\bV}}
\newcommand{\bu}{\mathbf{u}}

\newcommand{\bw}{\mathbf{w}}

\newcommand{\obu}{\overline{\mathbf{u}}}
\newcommand{\obr}{\overline{\mathbf{r}}}
\newcommand{\obw}{\overline{\mathbf{w}}}

\newcommand{\obf}{\overline{\mathbf{f}}}

\newcommand{\tbH}{\widetilde{\bH}}
\newcommand{\tbh}{\widetilde{h}}
\newcommand{\tgamma}{\widetilde{\gamma}}
\newcommand{\tbeta}{\widetilde{\beta}}
\newcommand{\obeta}{\overline{\beta}}
\newcommand{\ogamma}{\overline{\gamma}}
\newcommand{\by}{\mathbf{y}}

\newcommand{\dom}{\text{\rm{dom}}}

\newcommand{\beq}{\begin{equation}}
\newcommand{\eeq}{\end{equation}}

\title{Center manifolds for a class of degenerate evolution equations and existence of small-amplitude kinetic shocks}

\author{ Alin Pogan}
\address{Miami University, Oxford, OH 45056}
\email{pogana@miamioh.edu}
\thanks{ A. P. research was partially supported under the
Summer Research Grant program, Miami University}
\author{Kevin Zumbrun}
\address{Indiana University, Bloomington, IN 47405}
\email{kzumbrun@indiana.edu}
\thanks{Research of K.Z. was partially supported
under NSF grant no. DMS-0300487}

\begin{document}

\begin{abstract}
We construct center manifolds for a class of degenerate evolution equations including the
steady Boltzmann equation and related kinetic models,
establishing in the process existence and behavior of small-amplitude kinetic shock and boundary layers.
Notably, for Boltzmann's equation, we show that elements of the center manifold decay in velocity at 
\emph{near-Maxwellian rate}, in accord with the formal Chapman-Enskog picture of near-equilibrium flow as
evolution along the manifold of Maxwellian states, or Grad moment approximation via Hermite polynomials in velocity.
Our analysis is from a classical dynamical systems point of view, 
with a number of interesting modifications to accommodate ill-posedness of the underlying evolution equation.
\end{abstract}

\maketitle

\vspace{0.3cm}
\begin{minipage}[h]{0.48\textwidth}
\begin{center}
Miami University\\
Department of Mathematics\\
301 S. Patterson Ave.\\
Oxford, OH 45056, USA
\end{center}
\end{minipage}
\begin{minipage}[h]{0.48\textwidth}
\begin{center}
Indiana University \\
Department of Mathematics\\
831 E. Third St.\\
Bloomington, IN 47405, USA
\end{center}
\end{minipage}

\vspace{0.3cm}

\tableofcontents

\section{Introduction }\label{s1}

In this paper, we study existence and properties of
near-equilibrium steady solutions, including in particular small-amplitude
shock and boundary layers, of kinetic-type relaxation systems
\begin{equation}\label{Relax-EQ}
A^0 \bu_t + A\bu_x = Q(\bu),
\end{equation}
on a general Hilbert space $\bH$, where $A^0$, $A$ are given (constant) bounded linear operator and $Q$ is a bounded bilinear map (cf. \cite{MZ,PZ1}).
More generally, we study existence and approximation of center manifolds for a
class of degenerate evolution equations arising as steady equations
\begin{equation}\label{nonlinear}
A\bu' = Q(\bu)
\end{equation}
for \eqref{Relax-EQ}, including in particular the steady Boltzmann equation and cousins
along with approximants such as BGK and discrete-velocity models \cite{MZ,PZ1}.
Specifically, we are interested in the case when the linear operator $A$ is self-adjoint,
bounded, and one-to-one, but \textit{not boundedly invertible}.

Following \cite{MZ,PZ1}, we make the following assumptions on linear operator $A$ and nonlinearity $Q$.

\smallskip

\noindent{\bf Hypothesis (H1)}
\begin{enumerate}
\item[(i)] The linear operator $A$ is
bounded, self-adjoint, and one-to-one
on the Hilbert space $\bH$;
\item[(ii)] There exists $\bV$ a proper, closed subspace of $\bH$ with $\dim\bV^\perp<\infty$ and $B:\bH\times\bH\to\bV$ is a bilinear, symmetric, continuous map such that $Q(\bu)=B(\bu,\bu)$.
\end{enumerate}

\smallskip

\noindent{\bf Hypothesis (H2)} There exists an equilibrium $\obu\in\ker Q$ satisfying
\begin{enumerate}
\item[(i)] $Q'(\obu)$ is self-adjoint and $\ker Q'(\obu)=\bV^\perp$;
\item[(ii)] There exists $\delta>0$ such that $Q'(\obu)_{|\bV}\leq -\delta I_{\bV}$.
\end{enumerate}

\begin{example}\label{boltzeg}
Our main example is the steady Boltzmann equation
\be\label{Boltz}
\xi_1 f_x=Q(f),\quad x\in \R^1, \, \xi\in \R^3,
\ee
where $f=f(x,\xi)$ denotes density at spatial point $x$ of particles with velocity $\xi$;
$$
Q (f) := \int_{\RR^3}\int_{S^2} \big( f( \xi') f (\xi'_*)  -  f(\xi) f(\xi_*) \big) C(\Omega, \xi - \xi_*) d \Omega d\xi_*
$$
is a collision operator, with $\xi_* \in \RR^3$, $\Omega \in S^2$, and
$ \xi' = \xi + \big(\Omega\cdot(\xi_* - \xi) \big) \Omega$,
$\xi'_* = \xi_*- \big(\Omega \cdot  (\xi_* - \xi) \big) \Omega$;
and $C$ is a specified collision kernel;
see, e.g., \cite{Ce,Gl} for further details.
In the hard sphere case, $C (\Omega, \xi) = \big| \Omega \cdot \xi \big|$,
this can be put in form \eqref{nonlinear} satisfying (H1)-(H2) by the coordinate change
\be\label{change}
f\to \langle \xi \rangle^{1/2}f, \quad Q \to \langle \xi \rangle^{1/2}Q, \qquad \langle \xi\rangle:=\sqrt{1+|\xi|^2},
\ee
with $\bH$ the standard square-root Maxwellian-weighted $L^2$ space in variable $\xi$
and $A=\xi_1/\langle \xi\rangle$ \cite{MZ}.
Note that $A$ has no kernel on $\bH$, but essential spectra
$\xi_1/\langle \xi\rangle \to 0$ as $\xi_1\to 0$: an {\it essential singularity}.
\end{example}

Our analysis continues a program begun in \cite{PZ1} to develop dynamical systems tools for degenerate equations
\eqref{nonlinear}, suitable for the treatment of existence
and stability of kinetic shock and boundary layers in Boltzmann-type equations.
Similarly as in \cite{PZ1}, our basic strategy is, in the perturbation equations of \eqref{nonlinear}
about $\obu$, to isolate by direct computation a center subspace flow
$w_\rmc'=Jw_\rmc + g_\rmc$, and a hyperbolic (stable/unstable) subspace flow $\Gamma_0 w_\rmh' = E_0w_\rmh + g_\rmh$,
coupled by quadratic order nonlinearities $g_\rmc$ and $g_\rmh$. Here $J$ is a finite-dimensional matrix in Jordan form, and
$\Gamma_0$, $E_0$ are self-adjoint bounded operators, with $E_0$ negative definite and $\Gamma_0$ one-to-one but not
boundedly invertible (see \eqref{canon} and derivation) on $\tbV$ a finite codimension subspace of $\bH$ introduced in Section~\ref{s2},
then construct the center manifold by a fixed-point iteration based on inversion of the linear operators
$(\partial_x-J)$ and $(\Gamma_0 \partial_x -E_0)$ in a negatively exponentially weighted space in $x$.

As noted in \cite{PZ1}, a key point is that $(\Gamma_0 \partial_x -E_0)^{-1}$ is bounded in $H^1(\R,\tbV)$ {\it but not
in $C_{\mathrm{b}}(\R,\tbV)$};
hence, we are prevented from applying the usual sup-norm based arguments in the evolutionary variable $x$ \cite{B,HI,Z1,Zode}.
Accordingly, we carry out our fixed-point iteration instead in the Sobolev space $H^1(\R,\bH)$, a modification that,
as in \cite{PZ1}, costs a surprising amount of difficulty.
Interestingly, the difficulties in the two (stable vs. center manifold) cases are essentially complementary.
In the stable manifold case \cite{PZ1}, where the analysis is on $\R_+$, the main difficulty was in handling traces
at the boundary $x=0$.
In the center manifold case, where the analysis is on the whole space $\R$, the difficulty is rather with
regularity, specifically $H^1$ analysis in a negatively weighted space. 
In particular, we find it necessary to work in a mixed norm 
$\|f\|_{\cZ_{\gamma,\beta}(\bH)}:=(\|f\|_{L^2_{-\gamma}}^2+\|f'\|_{L^2_{-\beta}}^2)^{\frac12}$, with $\beta > \gamma$,
in order to obtain contraction of our fixed point iteration.  
The introduction of these spaces, along with the associated contraction estimates, we consider as one of the main technical
novelties of this paper.
The presence of an additional weight, along with derivative terms, considerably complicates the usual
argument for higher regularity via a cascade of spaces with decreasing weights;
see the treatment of smoothness of substitution operators in Appendix \ref{appendix}.

\subsection{Formal Chapman--Enskog expansion}
The Implicit Function Theorem yields the standard result of existence of an isolated
finite-dimensional manifold of equilibria through the base point $\obu$.
\begin{lemma}\label{l:eq}
	Assume that $\obu\in\ker Q$ is an equilibrium satisfying Hypotheses (H1) and (H2). Then,
there exists local to $\obu$ a unique $C^\infty$ (in Fr\'echet sense) manifold
of equilibria $\cE$, tangent at $\obu$ to $\bV^\perp$,
expressible in coordinates $\bw:=\bu-\obu$ as a $C^\infty$
graph $v_*:\bV^\perp \to\bV$.
\end{lemma}
Denote by $u=P_{\bV^\perp}\bu$, $v=P_{\bV}\bu$, where $P_{\bV^\perp}$ and $P_{\bV}$ are the orthogonal
projections onto associated to the decomposition $\bH=\bV^\perp \oplus \bV$.
Denoting $A_{11}=P_{\bV^\perp}A_{|\bV^\perp}$, $A_{12}=P_{\bV^\perp}A_{|\bV}$, we obtain the
standard fact that \eqref{nonlinear} admits a \textit{conservation law}
\be\label{cons}
(A_{11}u + A_{12}v)'=0.
\ee
The formal, \textit{first-order Chapman--Enskog} approximation of near-equilibrium behavior,
based on the assumption that deviations $v-v_*(u)$ from equilibrium
are small compared to variations in $u$, and their derivatives are small compared to the derivative
of $u$ (see, e.g., \cite{L}), is given by
\be\tag{${\rm CE}_1$}\label{ce}
(f_*(u))'=0, \; v\equiv v_*(u), \quad \hbox{\rm where $ f_*(u):=A_{11}u + A_{12}v_*(u)$},
\ee
corresponding to the steady problem for the system of hyperbolic conservation laws \cite{La,Sm}
\be\label{hyp}
h_*(u)_t + f_*(u)_x=0, \quad \hbox{\rm where $h_*(u):=P_{\bV^\perp}A^0u+ P_{\bV}A^0v_*(u)$,}
\ee
i.e., flow along equilibrium manifold $\cE$ (parametrized by $u$) governed by
\be\label{f*q}
f_*(u)\equiv q=\const.
\ee
The \textit{second-order Chapman-Enskog} approximation, corresponding to
$ h_*(u)_t + f_*(u)_x=D_*u_{xx}, $ is
\be\tag{${\rm CE}_2$}\label{ce2}
u'= D_*^{-1}(f_*(u)- q), \quad
\hbox{\rm where $ f_*(u):=A_{11}u + A_{12}v_*(u)$ and $D_*:=-A_{12} E^{-1}  A_{12}^*$},
\ee
with $E:= Q'(\obu)_{|\bV}$ denoting the restriction of $Q'(\obu)$ to its range.
See \cite{L,MZ,MZ2} for further details.

A secondary goal of this paper is to relate the rigorous center-manifold flow of \eqref{nonlinear} to
the first- and second-order Chapman--Enskog systems \eqref{ce} and \eqref{ce2}.
To this end, notice, first, that the set $\cE$ of equilibria of \eqref{nonlinear} is precisely the set of
solutions of \eqref{ce}, which in turn is the set of equilibria of \eqref{ce2}.
Thus, {\it at the level of equilibria, all three models exactly correspond.}

\subsubsection{Case structure}\label{s:class}
Next, we distinguish the {\it noncharacteristic case} $\det f_*'(\bar u)\neq 0$
and {\it characteristic case} $\det f_*'(\bar u)=0$, according as the characteristic velocities $\lambda_j(u)\in \sigma (f'_*(u))$ of
\eqref{hyp} are nonvanishing at $\bar u=P_{\bV}\obu$ or not.
In the noncharacteristic case, the Inverse Function Theorem yields that $f_*$ maps a neighborhood of $\bar u$
one-to-one onto a neighborhood of $\bar q:= f_*(\bar u)$, hence each fiber of \eqref{ce} is trivial, consisting
of a single equilibrium. Likewise, comparing dimensions, it is easily seen that the center manifold of \eqref{ce2} at $\bar u$ is
just the set of all equilibrium solutions, consisting of constant states $u(x)\equiv 0$ \cite{MaP}.

The characteristic case is more interesting, admitting nontrivial dynamics.
We distinguish two important subcases, the {\it simple, genuinely nonlinear}
and the {\it linearly degenerate} case \cite{La,Sm},
again having to do with structure of the first-order system \eqref{hyp}, both of which (and no others)
arise for Example~\ref{boltzeg}.
The simple, genuinely nonlinear case consists of the situation that $f_*'(\bar u)$ has a simple zero eigenvalue
with associated unit eigenvector $\obr$, for which \cite{Sm}
\be\tag{GNL}\label{gnl}
\Lambda:=\obr\cdot f_*''(\bar u)(\obr,\obr)\neq 0.
\ee
In this case, \eqref{ce} corresponds to a {\it fold bifurcation} \cite{Fre}, with $f_*$ mapping a disk around
$\bar u$ to a topological half-disk, with covering number two.
Moreover, points $u_1$, $u_2$ with the same image $q$, corresponding to equilibria of \eqref{nonlinear} and \eqref{ce2},
satisfy the Rankine-Hugoniot jump condition
\be\tag{RH}\label{rh}
f_*(u_1)=f_*(u_2),
\ee
corresponding to a discontinuous ``Lax-type'' standing-shock solution of \eqref{hyp}, \eqref{ce} \cite{La,Sm}.
Such a solution, having infinite derivative, does not satisfy in any obvious ways the assumptions in deriving the
formal approximation \eqref{ce}; however, a center-manifold analysis \cite{MaP} shows that the corresponding
fiber \eqref{ce2} of the associated second-order system contains a heteroclinic connection joining these two equilibria,
or {\it viscous shock profile.}  See \cite{MaP} for further discussion.

The linearly degenerate case consists of the opposite extreme, that, not only $\Lambda=0$, but
\be\tag{LDG}\label{ldg}
\hbox{\rm The solutions of \eqref{ce} consists, locally, of $\emptyset$ or an $m$-parameter manifold $\Delta$,}
\ee
where $m=\dim \ker f_*'(\bar u)$, given as the integral manifold of $m$ characteristic eigenvectors $\mathbf{e}_j(u)$
with common eigenvalue $\lambda_j(u)=\lambda(u)$ vanishing at $\bar u$, and constant along $\Delta$.
Thus, the $(\mathrm{dim}\bV^\perp+m)$-dimensional center manifold of \eqref{ce2} consists of the union of fibers \eqref{ce2} either composed entirely
of equilibria or having none; it therefore {\it admits no heteroclinic or homoclinic connections, nor even solutions approaching
an equilibrium as $x\to +\infty$ or $x\to -\infty$.} For further discussion, see Section~\ref{s:bif}.

\begin{example}\label{boltzceeg}
	For Example~\ref{boltzeg},
	the steady Boltzmann equation with hard sphere potential,
	$\cE$ is the set of Maxwellian distributions
	\be\label{M}
	M_u(\xi)  = \rho(4\pi e/3)^{-3/2}  e^{-  | \xi - v |^2(4 e/3)^{-1} },
	\ee
indexed by $u=(\rho, v^T,e)^T \in \R^5$, where $\rho$ represents density, $v\in \R^3$ velocity, and $e$ internal energy,
and \eqref{ce2} is the steady compressible Navier-Stokes (cNS) system with monatomic equation of state,
or {\it hydrodynamic limit} \cite{HLyZ,LiuYu,MZ}.
The corresponding first-order system \eqref{ce}, the compressible Euler equation, possesses two simple genuinely nonlinear
``acoustic'' characteristics $\lambda_1=v_1-\overline{c}$, $\lambda_5=v_1+\overline{c}$, where $\overline{c}>0$ is sound speed, and three linearly degenerate ``entropic/vorticity'' characteristics $\lambda_2=\lambda_3=\lambda_4=v_1$ \cite{Sm}.
\end{example}

\subsection{Main results}\label{s:main}
Under assumptions (H1) and (H2), we find (see Lemma \ref{l2.10}) that, under the linearized flow of \eqref{nonlinear}
about $\obu$, $\bH$ decomposes into invariant subspaces $\bH_\rmc\oplus\bH_\rms\oplus\bH_\rmu$, where
$\bH_\rmc$ is a finite-dimensional center subspace of dimension $\dim \bV^\perp + \dim \ker A_{11}$ and
$\bH_\rms$ and $\bH_\rmu$ are (typically infinite-dimensional) stable and unstable subspaces in the standard sense of
(nondegenerate) dynamical systems.
Our first main result asserts, likewise, existence of a center manifold in the usual dynamical systems
sense (cf. \cite{B,HI,Z1,Zode}).

\begin{theorem}\label{t1.1}
	Assume that $\obu\in\ker Q$ is an equilibrium satisfying (H1), (H2).
	Then, for any integer $k\geq2$
there exists local to $\obu$ a $C^k$ center manifold $\cM_\rmc$ (not necessary unique), tangent at $\obu$ to $\bH_\rmc$,
expressible in coordinates $\bw:=\bu-\obu$ as a $C^k$ graph $\cJ_\rmc:\bH_\rmc\to\bH_\rms\oplus\bH_\rmu$,
that is locally invariant under the flow of equation \eqref{nonlinear}, and contains all solutions
that remain bounded and sufficiently close to $\obu$ in forward and backward
time.
\end{theorem}
Once existence of a center manifold is established, one may obtain existence of small-amplitude shock profiles
by adapting the center manifold arguments of \cite{MaP,MasZ2} in the finite-dimensional case.
Here, we give instead a particularly simple normal form argument under the additional assumption of genuine nonlinearity \eqref{gnl},
whereas the arguments of \cite{MaP,MasZ2} were for the general case.
Similarly as in \cite{MaP,MasZ2}, the main idea is to use the fact that equilibria are predicted by the Rankine-Hugoniot shock
conditions for \eqref{ce} to deduce normal form information from the structure of the first-order Chapman-Enskog approximation.
Our second main result relates behavior on the center manifold to that of \eqref{ce2}.
\begin{theorem}\label{t1.2}
	Assume that $\obu\in\ker Q$ is an equilibrium satisfying (H1), (H2).
	In the noncharacteristic case, the center manifolds of \eqref{nonlinear} at $\obu$
	and \eqref{ce2} at $\bar u=P_{\bV^\perp}\obu$ consist entirely of equilibria, with trivial (constant) flow.
	In the characteristic case \eqref{gnl}, the center manifolds of \eqref{nonlinear} and
	\eqref{ce2} both consist of the union of one-dimensional fibers parametrized by $q\in \R^r$,
	governed by approximate Burgers flows: specifically,
	setting $u_1:=\obr \cdot u$, $q_1:=\obr \cdot q$, $q=(q_1,\tilde q)$, and
	without loss of generality (see Section~\ref{s:approx}) taking $\tilde q=0$,  the flow
	\be\label{burgers}
	u_1'=\varkappa^{-1}\big(-q_1 + \Lambda u_1^2/2\big) + \mathcal{O}(|u_1|^3 + |q_1||u_1|+ |q_1|^2),
	\ee
	where $\varkappa:=\obr^T D_* \obr>0$, $\obr$, $D_*$, $q$ as in \eqref{gnl}, \eqref{ce2}.
	In particular, there exist
	local heteroclinic (Lax shock) connections for $q_1 \Lambda>0$ between endstates
	$u_1^\pm \approx \sqrt{2 q_1/\Lambda}$.
	In the characteristic case \eqref{ldg}, the center manifolds of \eqref{nonlinear} and
	\eqref{ce2} consist of the union of $m$-dimensional\footnote{Here $m=\mathrm{dim}f_*'(\bar u)$} fibers with approximate flow, taking
	again without loss of generality $\tilde q=0$,
	\be\label{ldgflow}
	u_1'= -\varkappa^{-1}q_1 +\mathcal{O}(|q_1||u_{1}|+|q_1|^2),
	\ee
	$u_1, q_1\in \R^m$, $\varkappa\in \R^{m\times m}$,
	either consisting entirely of equilibria, or entirely
	of solutions leaving the vicinity of $\obu$ ($\bar u$) at both $x\to +\infty$ and $x\to -\infty$;
	in particular, there are no local heteroclinic (shock type) solutions, or (boundary layer type)
	solutions converging to equilibria as $x\to +\infty$ or $x\to -\infty$.
\end{theorem}

\begin{corollary}\label{c1.3}
Assume that $\obu\in\ker Q$ is an equilibrium satisfying (H1), (H2),
in the characteristic case \eqref{gnl}, and let $k$ be an integer $\geq 2$.
Then, local to $\obu$, ($\bar u$), each pair of points $u_\pm$ corresponding to a standing Lax-type shock
of \eqref{ce} has a corresponding viscous shock solution $u_{CE}$ of \eqref{ce2} and
relaxation shock solution $\bu_{REL}=(u_{REL},v_{REL})$ of \eqref{nonlinear}, satisfying for all $j\leq k-2$:
\begin{equation}\label{finalbds}
\begin{aligned}
	\big|\partial_x^j ( u_{REL}- u_{CE})(x)\big| &\le C \eps^{j+2}e^{-\delta  \eps|x|},\\
	\big|\partial_x^j \big(v_{REL}-v_*(u_{CE})\big)(x)\big| &\le C \eps^{j+2}e^{-\delta  \eps|x|},\\
|\partial_x^j (u_{REL}-u_\pm)(x)|&\le C \eps^{j+1}e^{-\delta \eps|x|}, \quad x\gtrless 0,\\
\end{aligned}
\end{equation}
for some $\delta>0$, $C>0$,  where $\eps:=|u_+-u_-|$, with also $\lambda(u_{REL}(x))$ monotone in $x$,
where $\lambda(u)$ is the simple eigenvalue of $f_*'(u)$ vanishing at $u=\bar u$.
Up to translation, these are the unique such solutions.
\end{corollary}
\br\label{kawrmk}
We do not assume as in \cite{MZ}
the usual ``genuine coupling'' or Kawashima condition that
no eigenvector of $A$ lie in the kernel of $Q'(\obu)$,
which would imply (see \cite{Y}) that \eqref{ce2} be of Kawashima class \cite{K}:
in particular, that viscosity coefficient $D_*$ be nonnegative semidefinite.
What takes the place of this condition is the assumption that $A$ has no kernel, which implies the weakened Kawashima
condition that no zero eigenvector of $A$ lie in the kernel of $Q'(\obu)$, which is sufficient that $\varkappa>0$
in Theorem~\ref{t1.2}.
As follows from the center manifold analysis of \cite{MaP}, this is enough for existence of small-amplitude
shock profiles for \eqref{ce2}, independent of the nature of $D_*$.
\er

\subsubsection{Boltzmann's equation}\label{s:ap}
Applying Theorems~\ref{t1.1}, ~\ref{t1.2} and Corollary~\ref{c1.3} to Example~\ref{boltzeg}, we immediately
(i) obtain existence of a center manifold, and (ii) recover and substantially
sharpen the fundamental result \cite{CN} of existence of small-amplitude Boltzmann shocks,
both with respect to the space $\bH$ determined by the (slight strengthening of the) classical square-root Maxwellian weighted norm
$\|f\|_{\bH}:=\|\langle \cdot\rangle^{1/2} M_{\bar u}^{-1/2}(\cdot) f(\cdot)\|_{L^2}$ \cite{Ce,G,MZ}.
Adapting a bootstrap argument of \cite{MZ}, we obtain the following improvement.
For any $1/2\leq \sigma<1$, denote by $\bY^\sigma$ the Hilbert space determined by norm
\be\label{Hs}
\|f\|_{\bY^\sigma}:= \|\langle \cdot\rangle^{1/2} M_{\bar u}^{-\sigma}(\cdot) f(\cdot)\|_{L^2}.
\ee
\begin{proposition}\label{p1.4}
For the steady Boltzmann equation with hard sphere potential (Example~\ref{boltzeg}),
for any $1/2\leq \sigma<1$ and integer $k\geq 2$,
there exists in the vicinity of any Maxwellian equilibrium $\obu=M_{\bar u}$
a $C^k(\bH_\rmc,\bY^\sigma)$ center manifold $\cM_\rmc \subset \bY^\sigma$ in the sense of Theorem~\ref{t1.1},
tangent at $\obu$ to $\bH_\rmc\subset \bY^\sigma$,
and expressible in coordinates $\bw:=\bu-\obu$ as a $C^k$ graph $\cJ_\rmc:\bH_\rmc\to(\bH_\rms\oplus\bH_\rmu)\cap \bY^\sigma$.
\end{proposition}
Proposition \ref{p1.4} implies that the center manifold, including any small-amplitude shock profiles, lies in the
space $\bY^\sigma$ of functions bounded in $L^2(\RR^3,e^{|\xi|^2}\rmd\xi)$ with a near-Maxwellian weight.
In particular, we obtain ``sharp localization in velocity'' of small-amplitude Boltzmann shock profiles,
recovering the strongest current existence result obtained in \cite{MZ}, plus the additional information of
monotonicity of $\lambda(u_{REL}(x))$ along the profile not available by the Sobolev-based fixed point iteration arguments of \cite{CN,MZ}.
This description of velocity-localization of the center manifold is sharp, as may be seen by the fact that the
equilibrium manifold $\cE$, contained in any center manifold,
itself lies in this space and no stronger one, changes in energy $e$ effectively changing the power of the Gaussian distribution
in the Maxwellian formula \eqref{M}.
It validates in a strong sense the formal Chapman-Enskog picture of near-equilibrium
behavior as governed essentially by the flow along the equilibrium manifold $\cE$, and the Grad hierarchy of
moment-closure approximations \cite{G} based on Hermite polynomials in velocity.

\subsection{Discussion and open problems}\label{s:discussion}
Writing the key infinite-dimensional hyperbolic equation $\Gamma_0 w_\rmh-E_0w_\rmh=g_\rmh$
formally as $ w_\rmh' = \Gamma_0^{-1} E_0w_\rmh + \Gamma_0^{-1} g_\rmh$,
we see that this is in general an ill-posed equation in both forward and backward $x$, due to non-bounded invertibility of
$\Gamma_0$, with the additional difficulty that the unbounded operator $\Gamma_0^{-1}$ also acts on the source
term $g_\rmh$.  In the latter sense, it is similar in flavor to quasilinear PDE problems involving maximal
regularity analysis.
Center manifolds for ill-posed evolutionary systems involving maximal regularity
have been treated by Mielke in \cite{Mi} and others, see, e.g., \cite{HI} and references therein.
The present, semilinar analysis, though different in particulars, seems to belong to this general family of results.

In the case of Boltzmann's equation \eqref{Boltz}, Liu and Yu \cite{LiuYu} have investigated
existence of center manifolds in a (weighted $L^\infty(x,\xi)$) {\it Banach space} setting, using
rather different methods of time-regularization and detailed pointwise bounds,
pointing out that monotonicity of $\lambda(\bar u)$ follows from center manifold reduction
and describing physical applications of center manifold theory
to condensation and subsonic/supersonic transition in Milne's problem.

A larger goal, beyond existence and construction of invariant manifolds,
is to develop dynamical systems tools for systems \eqref{Relax-EQ}
analogous to those developed for finite-dimensional
viscous shock and relaxation systems in \cite{GZ,MasZ2,Z2,Z3,Z4,Z5,ZH,ZS},
sufficient to treat 1- and multi-D stability by the techniques of those papers.
See in particular the discussion of \cite[Remark 4.2.1(4), p. 55]{Z4}, proposing a path toward stability of Boltzmann
shock profiles, which reduces the problem to description of the resolvent kernel in a small neighborhood of the
origin.

Such methods would apply in principle also to large-amplitude shocks, provided profiles exist and are
spectrally stable.
The development of numerical and or analytical methods for the treatment of existence and stability of large-amplitude
kinetic shocks we regard as a further, very interesting open problem.
Indeed, the {\it structure problem} discussed by Truesdell, Ruggeri, Boillat, and others, of existence and description of large-amplitude Boltzmann shocks, is one of the fundamental open problems in the theory, and
(because of more accurate fit to experiments than predictions of Navier-Stokes theory)
an important motivation for their study; see, e.g., the discussion of \cite{BR}.

\medskip

{\bf A glossary of notation:}\,  For $p\geq 1$, $J\subseteq\RR$ and $\bX$ a Banach space,
$L^p(J,\bX)$ are the usual Lebesgue spaces on $J$ with values in
$\bX$, associated with Lebesgue measure $\rmd x$ on $J$. Similarly,
$L^p(J,\bX;w(x)\rmd x)$ are the weighted spaces with a weight $w\geq0$. The respective spaces of bounded continuous functions on $J$
are denoted by $C_{\rm b}(J,\bX)$ and $C_{\rm b}(J,\bX;w(x))$.
$H^s(\RR, \bX)$, $s> 0$, is the usual Sobolev space of $\bX$ valued functions. The identity operator on a Banach space $\bX$ is denoted by $Id$ (or by $Id_{\bX}$ if its dependence on $\bX$
needs to be stressed). The set of bounded linear operators from a Banach space $\bX$ to itself is denoted by $\cB(\bX)$.
For an operator $T$ on a Hilbert space we use $T^\ast$, $\dom(T)$, $\ker T$, $\im T$, $\sigma(T)$, $R(\lambda,T)=(\lambda-T)^{-1}$ and $T_{|\bY}$ to denote the adjoint, domain, kernel, range, spectrum,  resolvent operator and the restriction of $T$ to a subspace $\bY$ of $\bX$.  If $B:J\to\cB(\bX)$ then $M_B$ denotes the operator of multiplication by $B(\cdot)$ in $L^p(J,\bX)$ or $C_{\rm b}(J,\bX)$. The Fourier transform of a Borel measure $\mu$ is defined by $(\mathcal{F}\mu)(\omega)=\int_{\mathbb{R}}{e^{-2\pi ix\omega}d\mu(x)}$.

\section{Linearized equations}\label{s2}
In this section we study the qualitative properties of the equation obtained by linearizing equation \eqref{nonlinear}
about the equilibrium $\obu$, and its perturbations by an inhomogeneous source term.
Throughout this section we assume Hypotheses (H1) and (H2). Our goal is to prove that the linearized equation,
\begin{equation}\label{linear}
A\bu'=Q'(\obu)\bu
\end{equation}
exhibits an exponential trichotomy on $\bH$ and to precisely describe the center, stable and unstable subspaces associated to this equation. A major difficulty when treating the linearized equation \eqref{linear} is given by the fact that the linear operator $A^{-1}Q'(\obu)$ does not generate a $C_0$-semigroup on $\bH$. Therefore, it is not straightforward to prove the existence of solutions of Cauchy problems associated to \eqref{linear} in forward time nor on backward time. Our first task is to show that the linearization decouples, which is a key point of our analysis. Denoting $E=Q'(\obu)_{|\bV}$, from
Hypothesis (H2) we infer following \cite{MZ,PZ1} that the bounded linear operators $A$ and $Q'(\obu)$ have the decomposition
\begin{equation}\label{decomp-AT}
A=\begin{bmatrix}
A_{11}&A_{12}\\
A_{21}&A_{22}\end{bmatrix}:\bV^\perp\oplus\bV\to\bV^\perp\oplus\bV,\quad Q'(\obu)=\begin{bmatrix}
0&0\\
0&E\end{bmatrix}:\bV^\perp\oplus\bV\to\bV^\perp\oplus\bV,
\end{equation}
where $E$ is symmetric negative definite (hence invertible).
Next, we denote by $P_{\bV}$ and $P_{\bV^\perp}$ the orthogonal projections onto $\bV$ and $\bV^\perp$, $u=P_{\bV^\perp}\bu$ and $v=P_{\bV}\bu$. From \eqref{decomp-AT} we obtain that equation \eqref{linear} is equivalent to the system
\begin{equation}\label{linear-decomposed}
\left\{\begin{array}{ll} A_{11}u'+A_{12}v'=0,\\
A_{21}u'+A_{22}v'=Ev. \end{array}\right.
\end{equation}
As noted in \cite{MZ,PZ1},
with this form, $A_{11}$ is exactly the Jacobian of the reduced ``equilibrium'' equation \eqref{ce} obtained by formal
Chapman-Enskog expansion. We distinguish the {\it noncharacteristic case} $\det A_{11}\neq 0$ and
the {\it characteristic case} $\det A_{11}=0$ according as this reduced hyperbolic system is noncharacteristic or not.
We turn our attention to the perturbation of the system \eqref{linear-decomposed} obtained by adding
a forcing term $f$ in the second equation, modeling nonlinear effects (recall that $\im Q=\bV$, so that nonlinearities enter in the second
equation only):
\begin{equation}\label{linear-sys-perturbed}
\left\{\begin{array}{ll} A_{11}u'+A_{12}v'=0,\\
A_{21}u'+A_{22}v'=Ev+f. \end{array}\right.
\end{equation}
We leave the function space of $f$ unspecified for the moment; ultimately,
it will be a negatively weighted $H^1$ space comprising functions growing at sufficiently slow exponential rate.

In this section we will show that system \eqref{linear-sys-perturbed}
is equivalent to a system of equations consisting of three finite-dimensional equations that can be readily integrated and
an infinite-dimensional equation of the form $\Gamma_0\tv'=E_0\tv+g$, with $\Gamma_0,E_0$ bounded linear operators on a finite codimension subspace of $\bV$ that can be treated
using the frequency-domain reformulation used in \cite{LP2,LP3,PZ1}. In the case when $A_{11}$ is invertible, one can solve for $u'$ in terms of $v'$ in the first equation of \eqref{linear-sys-perturbed} and then focus on the second equation. In the general case, when $A_{11}$ is \textit{not necessarily invertible}, we first decompose $\bV^\perp$ as follows:
since $A$ and $P_{\bV^\perp}$ are self-adjoint operators on $\bH$ and $A_{11}=P_{\bV^\perp}A_{|\bV^\perp}$, we have that $A_{11}$ is self-adjoint on $\bV^\perp$, which implies that
$\bV^\perp=\ker A_{11}\oplus\im A_{11}$. We denote by $P_{\ker A_{11}}$ and $P_{\im A_{11}}$ the orthogonal projectors onto $\ker A_{11}$ and $\im A_{11}$ associated to this decomposition. Next, we introduce the linear operators $\tA_{12}:\bV\to\im A_{11}$ and $T_{12}:\bV\to\ker A_{11}$ defined by $\tA_{12}=P_{\im A_{11}}A_{12}$ and $T_{12}=P_{\ker A_{11}}A_{12}$. In the next lemma we summarize some of the elementary properties of $\tA_{12}$ and $T_{12}$.
\begin{lemma}\label{r2.1} Assume Hypotheses (H1) and (H2). Then, the following assertions hold true.
\begin{enumerate}
\item[(i)] $\ker T_{12}^*=\{0\}$, $\im T_{12}=\ker A_{11}$, $\ker T_{12}\ne\{0\}$;
\item[(ii)] The linear operator $\tA_{11}=(A_{11})_{|\im A_{11}}$ is self-adjoint and invertible on $\im A_{11}$.
\end{enumerate}
\end{lemma}
\begin{proof}
(i) Since $A$ is a self-adjoint operator on $\bH$ by Hypothesis (H1) from \eqref{decomp-AT} we conclude that $A_{21}=A_{12}^*$. Thus, one can readily check that $T_{12}^*=(A_{12}^*)_{|\ker A_{11}}=(A_{21})_{|\ker A_{11}}$. Let $u\in\ker T_{12}^*\subseteq\dom(T_{12}^*)=\ker A_{11}\subseteq\bV^\perp$. It follows that $A_{11}u=A_{21}u=0$, which implies $Au=A_{11}u+A_{21}u=0$. From Hypothesis (H1) we obtain that $u=0$, proving that $\ker T_{12}^*=\{0\}$. Since $\im T_{12}$ is finite dimensional, we infer that it is a closed subspace of $\ker A_{11}$. Hence, $\im T_{12}=\big(\ker T_{12}^*\big)^\perp=\ker A_{11}$. Next, we assume for a contradiction that $\ker T_{12}=\{0\}$. Since $T_{12}\in\mathcal{B}(\bV,\ker A_{11})$, it follows that $\dim\bV\leq\dim\ker A_{11}\leq\dim\bV^\perp<\infty$, which is a contradiction. Assertion (ii) follows immediately since the linear operator $A_{11}$ is self-adjoint on $\bV^\perp$.
\end{proof}
To treat system \eqref{linear-sys-perturbed} we first introduce the subspaces $\bV_1=\im T_{12}^*$ and $\tbV=\ker T_{12}$. In what follows $P_{\bV_1}$ and $P_{\tbV}$ are the orthogonal projectors onto $\bV_1$ and $\tbV$, respectively. Denoting by $u_1=P_{\ker A_{11}}u$, $\tu=P_{\im A_{11}}u$, $v_1=P_{\bV_1}v$ and $\tv=P_{\tbV}v$ and applying the projectors $P_{\ker A_{11}}$ and $P_{\im A_{11}}$, respectively, to the first equation of \eqref{linear-sys-perturbed} we obtain that
\begin{equation}\label{u1-tildeu}
T_{12}v_1'=T_{12}v'=P_{\ker A_{11}}(A_{11}u'+A_{12}v')=0,\quad\tA_{11}\tu'+\tA_{12}\tv'=P_{\im A_{11}}(A_{11}u'+A_{12}v')=0.
\end{equation}
Moreover, since $(A_{21})_{|\ker A_{11}}=T_{12}^*$ and $(A_{21})_{|\im A_{11}}=\tA_{12}^*$ we have that the second equation of \eqref{linear-sys-perturbed} is equivalent to
\begin{equation}\label{inter-eq2}
T_{12}^*u_1'+\tA_{12}^*\tu'+A_{22}v'=Ev+f.
\end{equation}
Since $v_1\in\bV_1=\im T_{12}^*$, from \eqref{u1-tildeu} we conclude that $v'_1\in\ker T_{12}\cap\im T_{12}^*=\{0\}$, hence $v_1'=0$. In addition, since the linear operator $\tA_{11}$ is invertible on $\im A_{11}$ by Lemma~\ref{r2.1}(ii), we infer that $\tu'=-\tA_{11}^{-1}\tA_{12}\tv'$. Summarizing, \eqref{u1-tildeu} is equivalent to
\begin{equation}\label{u1-tildeu2}
v_1'=0,\quad\tu'=-\tA_{11}^{-1}\tA_{12}\tv'.
\end{equation}
Next, we solve for $u_1$ in terms of $v$ in \eqref{inter-eq2}. Multiplying this equation by $P_{\bV_1}$, from \eqref{u1-tildeu2}, we obtain that
\begin{equation}\label{u1-1}
T_{12}^*u_1'+P_{\bV_1}(A_{22}-\tA_{12}^*\tA_{11}^{-1}\tA_{12})\tv'=P_{\bV_1}Ev+P_{\bV_1}f
\end{equation}
From Lemma~\ref{r2.1}(i) we have that $(T_{12}^*)^{-1}$ is well-defined and bounded, linear operator from $\bV_1=\im T_{12}^*$ to $\ker A_{11}$. Thus, we can solve in \eqref{u1-1} for $u_1$ as follows:
\begin{equation}\label{u1-2}
u_1'=\Gamma_1\tv'+E_1v+(T_{12}^*)^{-1}P_{\bV_1}f.
\end{equation}
Here the linear operators $\Gamma_1:\bV\to\ker A_{11}$ and $E_1:\bH\to\ker A_{11}$ are defined by
\begin{equation}\label{def-Gamma1-E1}
\Gamma_1=(T_{12}^*)^{-1}(\tA_{12}^*\tA_{11}^{-1}\tA_{12}-A_{22})\in\mathcal{B}(\bV,\ker A_{11}),\quad E_1=(T_{12}^*)^{-1}P_{\bV_1}E\in\mathcal{B}(\bH,\ker A_{11}).
\end{equation}
Since $v'=\tv'$ by \eqref{u1-tildeu2}, multiplying equation \eqref{inter-eq2} by $P_{\tbV}$,  we infer that
\begin{equation}\label{prelim-tilde-v}
P_{\tbV}(A_{22}-\tA_{12}^*\tA_{11}^{-1}\tA_{12})\tv'=P_{\tbV}E\tv+P_{\tbV}Ev_1+P_{\tbV}f.
\end{equation}
From \eqref{u1-tildeu2}, \eqref{u1-2} and \eqref{prelim-tilde-v} we conclude that the system \eqref{linear-sys-perturbed} is equivalent to the system
\begin{equation}\label{linear-sys-perturbed2}
\left\{\begin{array}{ll} u_1'=\Gamma_1\tv'+E_1(v_1+\tv)+(T_{12}^*)^{-1}P_{\bV_1}f,\\
\tu'=-\tA_{11}^{-1}\tA_{12}\tv',\\
v_1'=0\\
\Gamma_0\tv'=E_0\tv+P_{\tbV}Ev_1+P_{\tbV}f,\end{array}\right.
\end{equation}
where the linear operators $\Gamma_0,E_0:\tbV\to\tbV$ are defined by
\begin{equation}\label{def-Gamma0-E0}
\Gamma_0=P_{\tbV}(A_{22}-\tA_{12}^*\tA_{11}^{-1}\tA_{12})_{|\tbV}\in\mathcal{B}(\tbV),\quad E_0=P_{\tbV}E_{|\tbV}\in\mathcal{B}(\tbV).
\end{equation}
We note that the first three equation of \eqref{linear-sys-perturbed2} can easily be integrated. The fourth equation of \eqref{linear-sys-perturbed2} is of the form
\begin{equation}\label{tildev-eq1}
\Gamma_0\tv'=E_0\tv+g,
\end{equation}
where $g:\RR\to\tbV$ is a constant perturbation $P_{\tbV}Ev_1+P_{\tbV}f$
of the (bounded) projection
of $f$ onto $\tbV$.

\subsection{Inhomogeneous equations}\label{s:inhom}
To understand the solutions of the perturbed equation \eqref{tildev-eq1}, it is crucial that we study the properties of the linear operators $\Gamma_0$ and $E_0$.
\begin{lemma}\label{l2.2}
Assume Hypotheses (H1) and (H2).  Then, the linear operators $\Gamma_0$ and $E_0$ satisfy the following conditions.
\begin{enumerate}
\item[(i)] $\Gamma_0$ is self-adjoint and one-to-one on $\tbV$;
\item[(ii)] The operator $E_0$ is self-adjoint, negative definite and invertible with bounded inverse on $\tbV$;
\item[(iii)] The operator $2\pi\rmi\omega\Gamma_0-E_0$ is invertible on $\tbV$ for any $\omega\in\RR$;
\item[(iv)] $\sup_{\omega\in\RR}\|(2\pi\rmi\omega\Gamma_0-E_0)^{-1}\|<\infty$.
\end{enumerate}
\end{lemma}
\begin{proof}
(i) Since the linear operator $A$ is self-adjoint, from \eqref{decomp-AT}, we obtain that $A_{22}$ is self-adjoint on $\bV$. In addition, since $P_{\tbV}$ is an orthogonal projector, and hence self-adjoint, from Lemma~\ref{r2.1} and \eqref{def-Gamma0-E0}, we conclude that $\Gamma_0$ is self-adjoint. Let $\tv\in\tbV=\ker T_{12}$ such that $\tv\in\ker\Gamma_0$, that is $P_{\tbV}(A_{22}-\tA_{12}^*\tA_{11}^{-1}\tA_{12})\tv=0$. Since $\tv\in\ker T_{12}$ it follows that $A_{12}\tv=\tA_{12}\tv+T_{12}\tv=\tA_{12}\tv$. Let $\tu=-\tA_{11}^{-1}\tA_{12}\tv\in\im A_{11}$. From the definition of $\tA_{11}$ in Lemma~\ref{r2.1}(ii), we have that $A_{11}\tu=\tA_{11}\tu=-\tA_{12}\tv=-A_{12}\tv$, which implies that
\begin{equation}\label{2.2-1}
A_{11}\tu+A_{12}\tv=0.
\end{equation}
Since $(A_{21})_{|\im A_{11}}=\tA_{12}^*$, we have that $\tA_{12}^*\tA_{11}^{-1}\tA_{12}\tv=-\tA_{12}^*\tu=-A_{21}\tu$. Since $\tv\in\ker\Gamma_0$, we obtain that $P_{\tbV}(A_{21}\tu+A_{22}\tv)=0$. Hence, $A_{21}\tu+A_{22}\tv\in \bV_1=\im T_{12}^*$, which implies that there exists $u_1\in\ker A_{11}$ such that $A_{21}\tu+A_{22}\tv=T_{12}^*u_1=A_{21}u_1$. Hence,
\begin{equation}\label{2.2-2}
A_{21}(\tu-u_1)+A_{22}\tv=0.
\end{equation}
Since $u_1\in\ker A_{11}$ from \eqref{decomp-AT}, \eqref{2.2-1} and \eqref{2.2-2}, we infer that $A(\tu-u_1+\tv)=0$. Since $A$ is one-to-one, $\tu-u_1\in\ker A_{11}\oplus\im A_{11}=\bV^\perp$ and $\tv\in\tbV\subset\bV$ we conclude that $\tv=0$, proving that $\ker\Gamma_0=\{0\}$.

Assertion (ii) follows from Hypothesis (H2) since $E\leq -\delta I_{\bV}$ and the projection $P_{\tbV}$ is orthogonal, and hence, self-adjoint. Denoting by $\cL_0:\RR\to\mathcal{B}(\tbV)$ the operator-valued function defined by $\cL_0(\omega)=2\pi\rmi\omega\Gamma_0-E_0$, from (i) and (ii) we obtain that $\mathrm{Re}\cL_0(\omega)=-E_0$ for any $\omega\in\RR$. From \eqref{def-Gamma0-E0} we have that
\begin{equation}\label{2.2-3}
\mathrm{Re}\langle \cL_0(\omega)\tv,\tv\rangle=-\langle E_0\tv,\tv\rangle=-\langle E\tv,\tv\rangle\geq\delta\|\tv\|^2\quad\mbox{for any}\quad\omega\in\RR, \tv\in\tbV.
\end{equation}
From \eqref{2.2-3} we immediately conclude that
\begin{equation}\label{2.2-4}
\|\cL_0(\omega)\tv\|\geq\delta\|\tv\|\quad\mbox{for any}\quad\omega\in\RR, \tv\in\tbV.
\end{equation}
From \eqref{2.2-4} we obtain that $\cL_0(\omega)$ is one-to-one and its range is closed on $\tbV$ for any $\omega\in\RR$. From \eqref{2.2-3} one can readily check that $\ker \cL_0(\omega)^*=\{0\}$ for any $\omega\in\RR$, proving (iii). Assertion (iv) follows from \eqref{2.2-4}.
\end{proof}
Lemma~\ref{l2.2} shows that the pair of operators $(\Gamma_0,E_0)$ satisfies Hypothesis(S) as stated in \cite[Section 3]{PZ1}. In the next lemma we summarize some of the important consequences from \cite{PZ1}.
\begin{lemma}[\cite{PZ1}]\label{r2.3} Assume Hypotheses (H1) and (H2). Then, the following assertions hold true.
\begin{enumerate}
\item[(i)] The linear operator $S_{\Gamma_0,E_0}=\Gamma_0^{-1}E_0:\dom(S_{\Gamma_0,E_0})=\{\tv\in\tbV:E_0\tv\in\im\Gamma_0\}\to\tbV$ generates an exponentially stable bi-semigroup on $\tbV$, that is, there exist $\tbV_\rms$ and $\tbV_\rmu$ two closed subspaces of $\tbV$, invariant under $S_{\Gamma_0,E_0}$, such that $\tbV=\tbV_\rms\oplus\tbV_\rmu$ and $(S_{\Gamma_0,E_0})_{|\tbV_\rms}$ and $-(S_{\Gamma_0,E_0})_{|\tbV_\rmu}$ generate exponentially stable, $C_0$-semigroups denoted $\{\tT_\rms(x)\}_{x\geq 0}$ and
    $\{\tT_\rmu(x)\}_{x\geq 0}$, having decay rate $-\nu(\Gamma_0,E_0)<0$;
 \item[(ii)] $\rmi\RR\subseteq\rho(S_{\Gamma_0,E_0})$ and $R(2\pi\rmi\omega,S_{\Gamma_0,E_0})=(2\pi\rmi\omega-S_{\Gamma_0,E_0})^{-1}=\big(\cL_0(\omega)\big)^{-1}\Gamma_0$ for all $\omega\in\RR$;
	\item[(iii)] There exists $c>0$ such that $\|R(2\pi\rmi\omega,S_{\Gamma_0,E_0})\|\leq \frac{c}{1+|\omega|}$ for all $\omega\in\RR$;
\item[(iv)] The Green function $\cG_{\Gamma_0,E_0}:\RR\to\cB(\tbV)$ defined by
	\begin{equation}\label{Green}
	\cG_{\Gamma_0,E_0}(x)=\left\{\begin{array}{l l}
	\tT_\rms(x)\tP_\rms & \; \mbox{if $x\geq0$ }\\
	-\tT_\rmu(-x)\tP_\rmu & \; \mbox{if $x<0$}\\
	\end{array} \right.,
	\end{equation}
\end{enumerate}
decays exponentially at $\pm\infty$. Here $\tP_{\rms/\rmu}$ denote the projections onto $\tbV_{\rms/\rmu}$ associated to the dichotomy decomposition $\tbV=\tbV_\rms\oplus\tbV_\rmu$. Moreover,
\begin{equation}\label{Green-properties}
\cF\cG_{\Gamma_0,E_0}(\cdot)\tv=R(2\pi\rmi\cdot,S_{\Gamma_0,E_0})\tv\;\mbox{for any}\; \tv\in\tbV.
\end{equation}
\end{lemma}
Now we have all the ingredients needed to treat solutions of equation \eqref{tildev-eq1} for functions $g\in L^2_{\mathrm{loc}}(\RR,\tbV)$. Our approach is the following: we first take Fourier Transform in \eqref{tildev-eq1} and then solve for $\cF\tv$ using the results from Lemma~\ref{l2.2} and Lemma~\ref{r2.3}. Next, we introduce the operator-valued function $R_{\Gamma_0,E_0}:\RR\to\mathcal{B}(\tbV)$ defined by $R_{\Gamma_0,E_0}(\omega)=(2\pi\rmi\omega\Gamma_0-E_0)^{-1}$.
We recall the definition of mild solutions of \eqref{tildev-eq1}.
\begin{definition}\label{d2.4}
We  say that
\begin{enumerate}
\item[(i)] The function $\tv$ is a \textit{mild} solution of \eqref{tildev-eq1} on $[x_0,x_1]$ if $\tv\in L^2([x_0,x_1],\tbV)$ satisfies
\begin{equation*} (\cF\tv_{|[x_0,x_1]})(\omega)=R(2\pi\rmi\omega,S_{\Gamma_0,E_0})\big(e^{-2\pi\rmi\omega x_0}\tv(x_0)-e^{-2\pi\rmi\omega x_1}\tv(x_1)\big)+R_{\Gamma_0,E_0}(\omega)(\cF g_{|[x_0,x_1]})(\omega)
\end{equation*}
for almost all $\omega\in\RR$;
\item[(ii)] The function $\tv$  is a \textit{mild} solution of \eqref{tildev-eq1} on $\RR$ if it is a \textit{mild} solution of \eqref{tildev-eq1} on $[x_0,x_1]$ for any $x_0,x_1\in\RR$.
\end{enumerate}
\end{definition}
Next, we introduce the linear operator $\cK_0:L^2(\RR,\tbV)\to L^2(\RR,\tbV)$ by $\cK_0g=\cF^{-1}M_{R_{\Gamma_0,E_0}}\cF g$, where $M_{R_{\Gamma_0,E_0}}$ denotes the multiplication operator on $L^2(\RR,\tbV)$ by the operator valued function $R_{\Gamma_0,E_0}$. From Lemma~\ref{l2.2}(iv) we have that $\sup_{\omega\in\RR}\|R_{\Gamma_0,E_0}(\omega)\|<\infty$, which proves that $\cK_0$ is well defined and bounded on $L^2(\RR,\tbV)$. Since we need to solve equation \eqref{tildev-eq1} for functions $g:\RR\to\tbV$ that are perturbations by constants of functions from $L^2(\RR,\tbV)$, we need to study how to extend the Fourier multiplier $\cK_0$ to a bounded, linear operator on $L^2_{-\alpha}(\RR,\tbV)$ and $H^1_{-\alpha}(\RR,\tbV)$, where $\alpha\in (0,\nu(\Gamma_0,E_0))$ is a small exponential weight.
To prove these results, we need the following result on convolutions.
\begin{lemma}\label{r2.5}
Assume $\cW:\RR\to\cB(\tbV)$ is a piecewise strongly continuous operator valued function such that $\|\cW(x)\|\leq ce^{-\nu|x|}$ for all $x\in\RR$. Then, for any $\alpha\in (0,\nu)$ we have that $\cW*g\in L^2_{-\alpha}(\RR,\tbV)$ and
\begin{equation}\label{L2alpha-estimate}
\|\cW*g\|_{L^2_{-\alpha}}\leq c\|g\|_{L^2_{-\alpha}}\quad\mbox{for any}\quad f\in L^2_{-\alpha}(\RR,\tbV).
\end{equation}
\end{lemma}
\begin{proof}
Let $g\in L^2_{-\alpha}(\RR,\tbV)$. Since $\cW$ decays exponentially at $\pm\infty$ and $\alpha\in (0,\nu)$ we have that
\begin{align}\label{2.5-1}
\|(\cW*g)(x)\|^2&\leq \Big(\int_{\RR}e^{-\nu|x-y|}\|g(y)\|\,\rmd y\Big)^2=\Big(\int_{\RR}e^{-\frac{\nu-\alpha}{2}|x-y|}e^{-\frac{\nu+\alpha}{2}|x-y|}\|g(y)\|\,\rmd y\Big)^2\nonumber\\
&\leq \int_{\RR} e^{-(\nu-\alpha)|x-y|}\,\rmd y\int_{\RR}e^{-(\nu+\alpha)|x-y|}\|g(y)\|^2\,\rmd y=\frac{2}{\nu-\alpha}\int_{\RR}e^{-(\nu+\alpha)|x-y|}\|g(y)\|^2\,\rmd y
\end{align}
Taking Fourier Transform one can readily check that
\begin{equation}\label{2.5-2}
e^{-2\alpha|\cdot|}*e^{-(\nu+\alpha)|\cdot|}=\frac{2(\nu+\alpha)}{(\nu+\alpha)^2-4\alpha^2}e^{-2\alpha|\cdot|}-\frac{4\alpha}{(\nu+\alpha)^2-4\alpha^2}e^{-(\nu+\alpha)|\cdot|},
\end{equation}
which implies that
\begin{equation}\label{2.5-3}
\int_{\RR} e^{-2\alpha|x|} e^{-(\nu+\alpha)|x-y|}\,\rmd x\leq \frac{2(\nu+\alpha)}{(\nu+3\alpha)(\nu-\alpha)} e^{-2\alpha|y|}\quad\mbox{for any}\quad y\in\RR.
\end{equation}
From \eqref{2.5-1} and \eqref{2.5-3} it follows that
\begin{align}\label{2.5-4}
\int_{\RR} e^{-2\alpha|x|}&\|(\cW*g)(x)\|^2\,\rmd x\leq \frac{2}{\nu-\alpha}\int_{\RR}\int_{\RR}e^{-2\alpha|x|}e^{-(\nu+\alpha)|x-y|}\|g(y)\|^2\,\rmd y\,\rmd x\nonumber\\
&=\frac{2}{\nu-\alpha}\int_{\RR}\Big(\int_{\RR}e^{-2\alpha|x|}e^{-(\nu+\alpha)|x-y|}\,\rmd x\Big)\|g(y)\|^2\,\rmd y\leq\frac{4(\nu+\alpha)}{(\nu+3\alpha)(\nu-\alpha)^2}\|g\|^2_{L^2_{-\alpha}}.
\end{align}
From \eqref{2.5-4} we conclude that $\cW*g\in L^2_{-\alpha}(\RR,\tbV)$ and that \eqref{L2alpha-estimate} holds true, proving the lemma.
\end{proof}
\begin{lemma}\label{l2.6}
Assume Hypotheses (H1) and (H2).
Then, for any $\alpha\in (0,\nu(\Gamma_0,E_0))$ the Fourier multiplier $\cK_0$ can be extended to a bounded, linear operator on $L^2_{-\alpha}(\RR,\tbV)$ and on $H^1_{-\alpha}(\RR,\tbV)$.
\end{lemma}
\begin{proof}
First, we introduce $\psi:\RR\to\RR$ the function defined by $\psi(x)=e^{-\alpha\langle x\rangle}$, where $\langle x\rangle=(1+x^2)^{\frac{1}{2}}$. One can readily check that $\psi\in H^2(\RR)$. To prove that $\cK_0$ can be extended to a bounded, linear operator on $L^2_{-\alpha}(\RR,\tbV)$, it is enough to prove that $\|\cK_0g\|_{L^2_{-\alpha}}\leq c\|g\|_{L^2_{-\alpha}}$ for any $g\in L^2(\RR,\tbV)$.

Fix $g\in L^2(\RR,\tbV)$. From \cite[Lemma 4.10]{PZ1} we have that $\psi\cK_0g=\cK_0\big(\psi g+\psi'(\cG_{\Gamma_0,E_0}^**g)\big)$. Clearly, $\psi g\in L^2(\RR,\tbV)$. From Lemma~\ref{r2.3}(iv) and Lemma~\ref{r2.5} we obtain that $\cG_{\Gamma_0,E_0}^**g\in L^2_{-\alpha}(\RR,\tbV)$. Since $|\psi'(x)|\leq ce^{-\alpha|x|}$ for any $x\in\RR$, we conclude that $\psi' (\cG_{\Gamma_0,E_0}^**g)\in L^2(\RR,\tbV)$ and
\begin{align}\label{2.6-1}
\|\cK_0g\|_{L^2_{-\alpha}}&\leq c\|\psi\cK_0g\|_2=c\|\cK_0(\psi g+\psi'(\cG_{\Gamma_0,E_0}^**g))\|_2\leq c\|\psi g+\psi'(\cG_{\Gamma_0,E_0}^**g)\|_2\nonumber\\
&\leq c\|\psi g\|_2+\|\psi'(\cG_{\Gamma_0,E_0}^**g)\|_2\leq c\|g\|_{L^2_{-\alpha}}+c\|\cG_{\Gamma_0,E_0}^**g\|_{L^2_{-\alpha}}\leq c\|g\|_{L^2_{-\alpha}},
\end{align}
proving that the Fourier multiplier $\cK_0$ can be extended to a bounded, linear operator on $L^2_{-\alpha}(\RR,\tbV)$.

Next, we fix $g\in H^1_{-\alpha}(\RR,\tbV)\cap L^2(\RR,\tbV)$. Using again Lemma~\ref{r2.3}(iv) and Lemma~\ref{r2.5}, we infer that $\cG_{\Gamma_0,E_0}^**g\in L^2_{-\alpha}(\RR,\tbV)$. Because $|\psi'(x)|+|\psi''(x)|\leq ce^{-\alpha|x|}$ for any $x\in\RR$, it follows that
\begin{equation}\label{2.6-2}
\psi'(\cG_{\Gamma_0,E_0}^**g)\in L^2(\RR,\tbV)\quad\mbox{and}\quad \psi''(\cG_{\Gamma_0,E_0}^**g)\in L^2(\RR,\tbV).
\end{equation}
Since $g\in H^1_{-\alpha}(\RR,\tbV)$, from Lemma~\ref{r2.5} we have that $\cG_{\Gamma_0,E_0}^**g\in H^1_{\mathrm{loc}}(\RR,\tbV)$ and $(\cG_{\Gamma_0,E_0}^**g)'=\cG_{\Gamma_0,E_0}^**g'\in L^2_{-\alpha}(\RR,\tbV)$. Thus, $\psi'(\cG_{\Gamma_0,E_0}^**g')\in L^2(\RR,\tbV)$. From \eqref{2.6-2} we conclude that $\psi'(\cG_{\Gamma_0,E_0}^**g)\in H^1_{\mathrm{loc}}(\RR,\tbV)$ and
\begin{equation}\label{2.6-3}
\big(\psi'(\cG_{\Gamma_0,E_0}^**g)\big)'=\psi''(\cG_{\Gamma_0,E_0}^**g)+\psi'(\cG_{\Gamma_0,E_0}^**g')\in L^2(\RR,\tbV).
\end{equation}
From Lemma~\ref{r2.5}, \eqref{2.6-2} and \eqref{2.6-3} we infer that $\psi'(\cG_{\Gamma_0,E_0}^**g)\in H^1(\RR,\tbV)$ and
\begin{align}\label{2.6-4}
\|\psi'(\cG_{\Gamma_0,E_0}^**g)\|_{H^1}&\leq c\|\psi'(\cG_{\Gamma_0,E_0}^**g)\|_2+c\|\big(\psi'(\cG_{\Gamma_0,E_0}^**g)\big)'\|_2\nonumber\\
&\leq c\|\cG_{\Gamma_0,E_0}^**g\|_{L^2_{-\alpha}}+c\|\psi''(\cG_{\Gamma_0,E_0}^**g)\|_2+c\|\psi'(\cG_{\Gamma_0,E_0}^**g')\|_2\nonumber\\
&\leq c \|g\|_{L^2_{-\alpha}}+c\|\cG_{\Gamma_0,E_0}^**g\|_{L^2_{-\alpha}}+c\|\cG_{\Gamma_0,E_0}^**g'\|_{L^2_{-\alpha}}\nonumber\\
&\leq c \|g\|_{L^2_{-\alpha}}+c \|g'\|_{L^2_{-\alpha}}=c \|g\|_{H^1_{-\alpha}}.
\end{align}
Since $\sup_{\omega\in\RR}\|R_{\Gamma_0,E_0}(\omega)\|<\infty$ by Lemma~\ref{l2.2}(iv), it follows that the Fourier multiplier $\cK_0$ can be extended to a bounded, linear operator on $H^1(\RR,\tbV)$. Since $\psi g, \psi'(\cG_{\Gamma_0,E_0}^**g)\in H^1(\RR,\tbV)$ we obtain that $\psi\cK_0g=\cK_0(\psi g+\psi'(\cG_{\Gamma_0,E_0}^**g))\in H^1(\RR,\tbV)$, and thus $\cK_0g\in H^1_{-\alpha}(\RR,\tbV)$. Summarizing, from \eqref{2.6-4} we conclude that
\begin{align}\label{2.6-5}
\|\cK_0g\|_{H^1_{-\alpha}}&\leq c\|\psi\cK_0g\|_{H^1}=c\|\cK_0(\psi g+\psi'(\cG_{\Gamma_0,E_0}^**g))\|_{H^1}\leq c\|\psi g+\psi'(\cG_{\Gamma_0,E_0}^**g)\|_{H^1}\nonumber\\
&\leq c\|\psi g\|_{H^1}+\|\psi'(\cG_{\Gamma_0,E_0}^**g)\|_{H^1}\leq c\|g\|_{H^1_{-\alpha}}.
\end{align}
From \eqref{2.6-5} it follows that the Fourier multiplier $\cK_0$ can be extended to a bounded, linear operator on $H^1_{-\alpha}(\RR,\tbV)$, proving the lemma.
\end{proof}
To simplify the notation, in the sequel we denote the extensions of $\cK_0$ to $L^2_{-\alpha}(\RR,\tbV)$ and $H^1_{-\alpha}(\RR,\tbV)$ by the same symbol. Moreover, from the definition of $\cK_0$ one can readily check that
\begin{equation}\label{deriv-cK0}
(\cK_0g)'=\cK_0g'\quad\mbox{for any}\quad g\in H^1_{-\alpha}(\RR,\tbV).
\end{equation}
Next, we will study the smoothness and uniqueness of solutions of \eqref{tildev-eq1} in the case when the function $g\in H^1_{-\alpha}(\RR,\tbV)$. To prove these results, we need the following lemma. \begin{lemma}\label{l2.7}
Assume Hypotheses (H1) and (H2), $\tv\in L^2(\RR,\tbV)$, $\mu$ is an $\tbV$-valued finite Borel measure such that
$\cF\tv=M_{R_{\Gamma_0,E_0}}\cF\mu$, and $\psi\in C_0^\infty(\RR)$. Then,
\begin{equation}\label{phi-lemma}
\widehat{\psi\tv}(\omega)-R_{\Gamma_0,E_0}(\omega)\Gamma_0\widehat{\phi'\tv}(\omega)=R_{\Gamma_0,E_0}(\omega)\int_{\RR} e^{-2\pi\rmi\omega x}\phi(x)\,\rmd \mu(x)\quad\mbox{for any}\quad\omega\in\RR.
\end{equation}
\end{lemma}
\begin{proof}
First, we note that the function $R_{\Gamma_0,E_0}$ satisfies the equation
\begin{equation}\label{2.7-1}
R_{\Gamma_0,E_0}(\omega_1)-R_{\Gamma_0,E_0}(\omega_2)=2\pi\rmi(\omega_2-\omega_1)R_{\Gamma_0,E_0}(\omega_1)\Gamma_0R_{\Gamma_0,E_0}(\omega_2)\quad\mbox{for any}\quad\omega_1,\omega_2\in\RR.
\end{equation}
The lemma follows from \eqref{2.7-1} by a long, but fairly simple computation. Indeed,
\begin{align}\label{2.7-2}
\widehat{\psi\tv}(\omega)&-R_{\Gamma_0,E_0}(\omega)\Gamma_0\widehat{\phi'\tv}(\omega)=(\widehat{\psi}*\widehat{\tv})(\omega)-R_{\Gamma_0,E_0}(\omega)\Gamma_0(\widehat{\phi'}*\widehat{\tv})(\omega)\nonumber\\
&=\int_{\RR}\widehat{\phi}(\omega-\theta)\widehat{\tv}(\theta)\rmd\theta-R_{\Gamma_0,E_0}(\omega)\Gamma_0\int_{\RR}\widehat{\phi'}(\omega-\theta)\widehat{\tv}(\theta)\rmd\theta\nonumber\\
&=\int_{\RR}\widehat{\phi}(\omega-\theta)\widehat{\tv}(\theta)\rmd\theta-R_{\Gamma_0,E_0}(\omega)\Gamma_0\int_{\RR}2\pi\rmi(\omega-\theta)\widehat{\phi}(\omega-\theta)\widehat{\tv}(\theta)\rmd\theta\nonumber\\
&=\int_{\RR}\widehat{\phi}(\omega-\theta)\Big(I_{\tbV}-2\pi\rmi(\omega-\theta)R_{\Gamma_0,E_0}(\omega)\Gamma_0\Big)\widehat{\tv}(\theta)\rmd\theta\nonumber\\
&=\int_{\RR}\widehat{\phi}(\omega-\theta)\Big(R_{\Gamma_0,E_0}(\theta)-2\pi\rmi(\omega-\theta)R_{\Gamma_0,E_0}(\omega)\Gamma_0R_{\Gamma_0,E_0}(\theta)\Big)\widehat{\mu}(\theta)\rmd\theta\nonumber\\
&=R_{\Gamma_0,E_0}(\omega)\int_{\RR}\widehat{\phi}(\omega-\theta)\widehat{\mu}(\theta)\rmd\theta=R_{\Gamma_0,E_0}(\omega)(\widehat{\phi}*\widehat{\mu})(\omega)\nonumber\\&=R_{\Gamma_0,E_0}(\omega)\int_{\RR} e^{-2\pi\rmi\omega x}\phi(x)\rmd\mu(x)
\end{align}
for any $\omega\in\RR$, proving the lemma.
\end{proof}
\begin{lemma}\label{l2.8}
Assume Hypotheses (H1) and (H2). Then, the following assertions hold true.
\begin{enumerate}
\item[(i)] For any $g\in H^1_{-\alpha}(\RR,\tbV)$, $\cK_0g$ is a mild solution of \eqref{tildev-eq1} on $[x_0,x_1]$ for any $x_0,x_1\in\RR$;
\item[(ii)] If $g\in L^2_{-\alpha}(\RR,\tbV)$ and $\tv\in L^2_{-\alpha}(\RR,\tbV)$ is a mild solution of equation \eqref{tildev-eq1} on $\RR$, then $\tv=\cK_0g$.
\end{enumerate}
\end{lemma}
\begin{proof} First, we consider a sequence of smooth functions $\phi_n\in C^\infty_0(\RR)$ such that $0\leq\phi_n\leq 1$, $\phi_n(x)=1$ for any $x\in [-n,n]$, $\phi_n(x)=0$ whenever
$|x|\geq n+1$ and $\sup_{n\in\NN}\|\phi_n'\|_\infty<\infty$ for any $n\geq 1$. One can readily check that $\phi_n\to1$ and $\phi_n'\to 0$ simple as $n\to\infty$. Moreover, from Lebesgue Dominated Convergence Theorem one can readily check that
\begin{align}\label{2.8-1}
&\phi_n'g\to 0\;\mbox{in}\;L^2_{-\alpha}(\RR,\tbV)\;{as}\;n\to\infty\;\mbox{for any}\;g\in L^2_{-\alpha}(\RR,\tbV),\nonumber\\
&\phi_ng\to g\;\mbox{in}\;L^2_{-\alpha}(\RR,\tbV)\;{as}\;n\to\infty\;\mbox{for any}\;g\in L^2_{-\alpha}(\RR,\tbV)\nonumber\\
&\phi_ng\to g\;\mbox{in}\;H^1_{-\alpha}(\RR,\tbV)\;{as}\;n\to\infty\;\mbox{for any}\;g\in H^1_{-\alpha}(\RR,\tbV).
\end{align}
\noindent\textit{Proof of (i)}. Fix $g\in H^1_{-\alpha}(\RR,\tbV)$. Since $\phi_n\in C^\infty_0(\RR)$ we have that $\phi_ng\in H^1(\RR,\tbV)$, and thus $\tv_n:=\cK_0(\phi_ng)\in H^1(\RR,\tbV)$ for any $n\geq 1$. From the definition of the Fourier multiplier $\cK_0$, we immediately obtain that
\begin{equation}\label{2.8-2}
\cF(\Gamma_0\tv_n'-E\tv_n)(\omega)=(2\pi\rmi\omega\Gamma_0-E_0)\widehat{\tv_n}(\omega)=\widehat{\phi_ng}(\omega)\quad\mbox{for any}\quad\omega\in\RR,
\end{equation}
which proves that $\Gamma_0\tv_n'=E\tv_n+\phi_ng$  for any $n\geq 1$. Since $H^1$-solutions of \eqref{tildev-eq1} are also mild solutions, we have that
\begin{equation}\label{2.8-3}
\big((\cF\tv_n)_{|[x_0,x_1]}\big)(\omega)=R(2\pi\rmi\omega,S_{\Gamma_0,E_0})\big(e^{-2\pi\rmi\omega x_0}\tv_n(x_0)-e^{-2\pi\rmi\omega x_1}\tv_n(x_1)\big)+R_{\Gamma_0,E_0}(\omega)(\cF \phi_ng_{|[x_0,x_1]})(\omega)
\end{equation}
for any $\omega\in\RR$ and $n\geq 1$. From \eqref{2.8-1} we have that $\phi_ng\to g$ in $H^1_{-\alpha}(\RR,\tbV)$ as $n\to\infty$. From Lemma~\ref{l2.6} we conclude that $\tv_n=\cK_0(\phi_ng)\to \cK_0g$ in $H^1_{-\alpha}(\RR,\tbV)$ as $n\to\infty$. It follows that $(\tv_n)_{|[x_0,x_1]}\to (\cK_0g)_{|[x_0,x_1]}$ and $(\phi_ng)_{|[x_0,x_1]}\to g_{|[x_0,x_1]}$ in $L^2(\RR,\tbV)$ and $\tv_n(x)\to(\cK_0g)(x)$ for any $x\in\RR$, as $n\to\infty$. Passing to the limit in \eqref{2.8-3} we obtain that $\cK_0g$ is a mild solution of \eqref{tildev-eq1} on $[x_0,x_1]$ for any $x_0,x_1\in\RR$, proving (i).

\noindent\textit{Proof of (ii)}. Assume $g\in L^2_{-\alpha}(\RR,\tbV)$ and $\tv\in L^2_{-\alpha}(\RR,\tbV)$ is a mild solution of \eqref{tildev-eq1}. We define the sequence of functions $z_n=\phi_n\tv$, $n\geq 1$. First, we note that $z_n=\phi_n(\tv\chi_{[-n-1,n+1]})$ for any $n\geq 1$. Since $\tv$ is a mild solution of \eqref{tildev-eq1} we have that
\begin{equation}\label{2.8-4}
\big(\cF\tv\chi_{[-n-1,n+1]}\big)(\omega)=\big(\cF\tv_{|[-n-1,n+1]}\big)(\omega)=R_{\Gamma_0,E_0}(\omega)\widehat{\mu_n}(\omega)\quad\mbox{for any}\quad\omega\in\RR, n\geq 1,
\end{equation}
where $\mu_n:\mathrm{Bor}(\RR)\to\tbV$ is the Borel measure defined by
\begin{equation}\label{2.8-5}
\mu_n(\Omega)=\mathrm{Dirac}_{-n-1}(\Omega)\Gamma_0\tv(-n-1)-\mathrm{Dirac}_{n+1}(\Omega)\Gamma_0\tv(n+1)+\int_{\Omega\cap[-n-1,n+1]}g(x)\rmd x.
\end{equation}
From Lemma~\ref{l2.7} it follows that
\begin{align}\label{2.8-6}
\widehat{z_n}(\omega)&-R_{\Gamma_0,E_0}(\omega)\Gamma_0 \big(\cF(\phi_n'\tv_{|[-n-1,n+1]})\big)(\omega)=R_{\Gamma_0,E_0}(\omega)\int_{\RR}e^{-2\pi\rmi\omega x}\phi_n(x)\rmd\mu_n(x)\nonumber\\
&=R_{\Gamma_0,E_0}(\omega)\Big(e^{2\pi\rmi\omega(n+1)}\phi_n(-n-1)\Gamma_0\tv(-n-1)-e^{-2\pi\rmi\omega(n+1)}\phi_n(n+1)\Gamma_0\tv(n+1)\Big)\nonumber\\
&\quad+R_{\Gamma_0,E_0}(\omega)\int_{\RR}e^{-2\pi\rmi\omega x}\phi_n(x)\chi_{[-n-1,n+1]}(x)g(x)\rmd x\nonumber\\
&=R_{\Gamma_0,E_0}(\omega)\widehat{\phi_ng}(\omega)\quad\mbox{for any}\quad\omega\in\RR, n\geq 1.
\end{align}
Since $\phi_n'\chi_{[-n-1,n+1]}=\phi_n'$ for any $n\geq 1$ from \eqref{Green-properties} and \eqref{2.8-6} we infer that
\begin{equation}\label{2.8-7}
z_n=\cG_{\Gamma_0,E_0}*(\phi_n'\tv)+\cK_0(\phi_ng)\quad\mbox{for any}\quad n\geq 1.
\end{equation}
Since $\tv\in L^2_{-\alpha}(\RR,\tbV)$ from Lemma~\ref{r2.5} and \eqref{2.8-1} we infer that $\cG_{\Gamma_0,E_0}*(\phi_n'\tv)\to 0$ and $z_n=\phi_n\tv\to\tv$ in $L^2_{-\alpha}(\RR,\tbV)$ as $n\to\infty$.
Moreover, since $g\in L^2_{-\alpha}(\RR,\tbV)$ from Lemma~\ref{l2.6} and \eqref{2.8-1} we have that $\cK_0(\phi_ng)\to \cK_0g$ in $L^2_{-\alpha}(\RR,\tbV)$ as $n\to\infty$. Passing to the limit in \eqref{2.8-7} we conclude that $\tv=\cK_0g$, proving the lemma.
\end{proof}
To finish this section, we use Lemma~\ref{l2.8} to prove an identity useful in the sequel. Let $\rmone$ be the function identically equal to one on the hole line.
From \eqref{deriv-cK0} we have that $\big(\cK_0(z\rmone)\big)'=\cK_0(z\rmone)'=0$, which proves that $\cK_0(z\rmone)$ is a constant function for any $z\in\tbV$. Since $z\rmone\in H^1_{-\alpha}(\RR,\tbV)$, from Lemma~\ref{l2.8}(ii) we have that $\cK_0(z\rmone)$ is the unique $H^1_{-\alpha}$ solution of equation $\Gamma_0\tv'=E_0\tv+z\rmone$. However, one can readily check that $-E_0^{-1}z\rmone$ is a solution this equation, which implies that
\begin{equation}\label{multiplier-constant}
\cK_0(z\rmone)=-E_0^{-1}z\rmone\quad\mbox{for any}\quad z\in\tbV.
\end{equation}

\subsection{Linear flow in characteristic and noncharacteristic case}\label{char}
In this subsection we prove that equation \eqref{linear} exhibits an exponential trichotomy on $\bH$ with finite dimensional center subspace. To prove this result, we solve the system \eqref{linear-sys-perturbed2} for the case when $f\equiv 0$. First, we look for the center subspace, the space of all vectors on $\bH$ that can be propagated in backward and forward time and whose associated solutions grow slower than any exponential. To define the center subspace we introduce the space
\begin{equation}\label{def-Vc}
\bV_\rmc=\{v=(v_1,\tv)\in\bV:\tv=-E_0^{-1}P_{\tbV}Ev_1\}.
\end{equation}
\begin{lemma}\label{l2.9}
Assume Hypotheses (H1) and (H2). Then, for any $w_0=(u^0_1,\tu^0,v_1^0,\tv^0)^{\mathrm{T}}\in\bH_\rmc:=\bV^\perp\oplus\bV_\rmc$ there exists a unique solution $\bu_\rmc$ of \eqref{linear} on $\RR$
such that $\bu_\rmc(0)=w_0$, given by
\begin{equation}\label{center-flow}
\bu_\rmc(x,w_0)=(u_1^0+x E_1(I_{\bH}-E_0^{-1}P_{\tbV}E)v_1^0,\tu^0,v_1^0,-E_0^{-1}P_{\tbV}Ev_1^0)^{\mathrm{T}}\in\bH_\rmc\quad\mbox{for any}\quad x\in\RR.
\end{equation}
\end{lemma}
\begin{proof}
Fix $w_0=(u^0_1,\tu^0,v_1^0,\tv^0)^{\mathrm{T}}\in\bV^\perp\oplus\bV_\rmc$ and assume that $\bu=(u_1,\tu,v_1,\tv)^{\mathrm{T}}$ is a mild solution of \eqref{linear} such that $\bu(0)=w_0\in\bH_\rmc$.
Since $\bu(0)\in\bH_\rms$, we obtain that $\tv(0)=-E_0^{-1}P_{\tbV}Ev_1(0)$. Using that equation \eqref{linear} is equivalent to
\eqref{linear-sys-perturbed2} with $f\equiv0$, from the third equation we conclude that $v_1(x)=v_1(0)=v_1^0$ for any $x\in\RR$. Since any constant function belongs to $H^1_{-\alpha}(\RR,\tbV)$, from Lemma~\ref{l2.8} and \eqref{multiplier-constant} we infer that $\tv=\cK_0(P_{\tbV}Ev_1(0)\rmone)=-E_0^{-1}P_{\tbV}Ev_1(0)\rmone=-E_0^{-1}P_{\tbV}Ev_1^0\rmone\in H^1_{-\alpha}(\RR,\tbV)$.
Integrating in the first two equations of \eqref{linear-sys-perturbed2}, we obtain that
\begin{equation}\label{2.9-1}
u_1(x)=u_1^0+x E_1(I_{\bH}-E_0^{-1}P_{\tbV}E)v_1^0,\quad\tu(x)=\tu^0\quad\mbox{for any}\quad x\in\RR,
\end{equation}
which shows that $\bu=\bu_\rmc(\cdot,w_0)$. On the other hand, one can readily check that $\bu_\rmc(\cdot,w_0)$ is a solution of \eqref{linear} and $\bu_\rmc(x,w_0)\in\bH_\rmc$ for any $x\in\RR$, proving the lemma.
\end{proof}
We note that an alternative way of constructing the center subspace of equation \eqref{linear} is the following: first, we note that any vector from $\bV^\perp$ is a \textit{stationary mode}. Next, we show that since $E=Q'(\obu)_{|\bV}$ is negative definite, we have \textit{generalized zero-modes} of height one, but no \textit{generalized zero-modes} of height two. Moreover, all the remaining modes are hyperbolic.

Next, we prove that there exists a direct complement of $\bH_\rmc$, not necessarily orthogonal, on which equation \eqref{linear} has an exponential dichotomy. To define the dichotomy decomposition, we use that the operator $S_{\Gamma_0,E_0}=\Gamma_0^{-1}E_0$ generates a bi-semigroup on $\tbV$ and the decomposition $\tbV=\tbV_\rms\oplus\tbV_\rmu$ from Lemma~\ref{r2.3}.
\begin{lemma}\label{l2.10}
Assume Hypotheses (H1) and (H2). Then, the following assertions hold true.
\begin{enumerate}
\item[(i)] Assume that $\bu=(u_1,\tu,v_1,\tv)^{\mathrm{T}}$ is a solution of \eqref{linear} such that $v_1(0)=0$. Then,
\begin{equation}\label{inter-dich}
u_1'=(\Gamma_1+E_1E_0^{-1}\Gamma_0)\tv',\;\tu'=-\tA_{11}^{-1}\tA_{12}\tv',\;v_1\equiv0,\;\tv'=S_{\Gamma_0,E_0}\tv;
\end{equation}
\item[(ii)] Equation \eqref{linear} has an exponential dichotomy on $\RR$ on a direct complement of $\bH_\rmc$, with dichotomy subspaces given by
\begin{equation}\label{dichtomy-subspaces}
\bH_{\rms/\rmu}=\Big\{\big((\Gamma_1+E_1E_0^{-1}\Gamma_0)\tv,-\tA_{11}^{-1}\tA_{12}\tv,0,\tv\big)^{\mathrm{T}}:\tv\in\tbV_{\rms/\rmu}\Big\};
\end{equation}
\item[(iii)] The Hilbert space $\bH$ decomposes as follows: $\bH=\bH_\rmc\oplus\bH_\rms\oplus\bH_\rmu$. Moreover, the trichotomy projection onto $\bH_\rmc$ parallel to $\bH_\rms\oplus\bH_\rmu$ is given
for any $\bu=(u_1,\tu,v_1,\tv)^{\mathrm{T}}\in\bH$ by
\begin{equation}\label{center-proj}
P_\rmc\bu=\Big(u_1-(\Gamma_1+E_1E_0^{-1}\Gamma_0)(\tv+E_0^{-1}P_{\tbV}Ev_1),\tu+\tA_{11}^{-1}\tA_{12}(\tv+E_0^{-1}P_{\tbV}Ev_1),v_1,-E_0^{-1}P_{\tbV}Ev_1\Big)^{\mathrm{T}}.
\end{equation}
\end{enumerate}
\end{lemma}
\begin{proof} (i) Since \eqref{linear} is equivalent to the system \eqref{linear-sys-perturbed2} with $f\equiv0$, from the third equation we immediately conclude that $v_1(x)=v_1(0)=0$ for any $x\in\RR$. Plugging in $v_1\equiv0$, from the fourth equation we obtain that $\tv'=S_{\Gamma_0,E_0}\tv$.  Since $\tv=E_0^{-1}\Gamma_0\tv'$ from the first equation we conclude that $u_1'=\Gamma_1\tv'+E_1\tv=(\Gamma_1+E_1E_0^{-1}\Gamma_0)\tv'$, proving (i). Assertion (ii) follows from (i) and the fact that the linear operator $S_{\Gamma_0,E_0}$ generates an exponentially stable bi-semigroup on $\tbV$. Indeed, from equation \eqref{inter-dich} and Lemma~\ref{r2.3}(i), one can readily check that solutions of \eqref{linear} that decay exponentially at $\pm\infty$ are given by
\begin{equation}\label{2.10-1}
\bu_{\rms/\rmu}(x)=\Big((\Gamma_1+E_1E_0^{-1}\Gamma_0)\tT_{\rms/\rmu}(\pm x)\tv(0),-\tA_{11}^{-1}\tA_{12}\tT_{\rms/\rmu}(\pm x)\tv(0),0,\tT_{\rms/\rmu}(\pm x)\tv(0)\Big)^{\mathrm{T}}\
\end{equation}
for any $x\in\RR_\pm$. Since $\tbV=\tbV_\rms\oplus\tbV_\rmu$ from \eqref{dichtomy-subspaces} we obtain that
\begin{equation}\label{2.10-2}
\bH_\rms\oplus\bH_\rmu=\Big\{\big((\Gamma_1+E_1E_0^{-1}\Gamma_0)\tv,-\tA_{11}^{-1}\tA_{12}\tv,0,\tv\big)^{\mathrm{T}}:\tv\in\tbV\Big\}.
\end{equation}
From \eqref{def-Vc} and \eqref{2.10-2} we immediately conclude that $\bH_\rmc\cap(\bH_\rms\oplus\bH_\rmu)=\{0\}$. Moreover,
\begin{align}\label{2.10-3}
(u_1,&\tu,v_1,\tv)^{\mathrm{T}}=(u_1,\tu,v_1,-E_0^{-1}P_{\tbV}Ev_1)^{\mathrm{T}}+(0,0,0,\tv+E_0^{-1}P_{\tbV}Ev_1)^{\mathrm{T}}\nonumber\\
&=\Big(u_1-(\Gamma_1+E_1E_0^{-1}\Gamma_0)(\tv+E_0^{-1}P_{\tbV}Ev_1),\tu+\tA_{11}^{-1}\tA_{12}(\tv+E_0^{-1}P_{\tbV}Ev_1),v_1,-E_0^{-1}P_{\tbV}Ev_1\Big)^{\mathrm{T}}\nonumber\\
&+\Big((\Gamma_1+E_1E_0^{-1}\Gamma_0)(\tv+E_0^{-1}P_{\tbV}Ev_1),-\tA_{11}^{-1}\tA_{12}(\tv+E_0^{-1}P_{\tbV}Ev_1),0,\tv+E_0^{-1}P_{\tbV}Ev_1\Big)^{\mathrm{T}},
\end{align}
proving that $\bH=\bH_\rmc\oplus\bH_\rms\oplus\bH_\rmu$. The formula \eqref{center-proj} follows shortly from \eqref{2.10-3}, proving the lemma.
\end{proof}

\subsection{Solutions of the inhomogeneous equation}\label{perturbed-eq}
Having established the exponential trichotomy of the linear flow of equation \eqref{linear}, we conclude this section by analyzing the solutions of equation \eqref{linear-sys-perturbed} for general functions $f\in L^2_{-\alpha}(\RR,\bV)$.
To prove such a result, we first recall the boundedness property of the Volterra operator on weighted spaces with negative weight. For the sake of completeness we give the details below.
\begin{lemma}\label{r2.11}
The Volterra operator $\cV:L^2_{-\alpha}(\RR,\bH)\to H^1_{-\alpha}(\RR,\bH)$ defined by $(\cV g)=\displaystyle\int_0^x g(y)\rmd y$ is bounded for any $\alpha>0$.
\end{lemma}
\begin{proof}
Fix $g\in L^2_{-\alpha}(\RR,\bH)$ and let $h:\RR\to\bH$ be defined by $h(x)=e^{-\alpha|x|}g(x)$. One can readily check that
\begin{equation}\label{2.11-1}
e^{-\alpha x}(\cV g)(x)=\int_0^x e^{-\alpha(x-y)}h(y)\rmd y=\big(\varphi_+*(\chi_{\RR_+}h)\big)(x)\quad\mbox{for any}\quad x\geq 0,
\end{equation}
where $\varphi_+:\RR\to\RR$ is defined by $\varphi_+(x)=\left\{\begin{array}{l l}
e^{-\alpha x} & \; \mbox{if $x\geq0$ }\\
0 & \; \mbox{if $x<0$}\\
\end{array} \right.$.
Similarly, we have that
\begin{equation}\label{2.11-2}
e^{\alpha x}(\cV g)(x)=-\int_x^0 e^{\alpha(x-y)}h(y)\rmd y=\big(\varphi_-*(\chi_{\RR_-}h)\big)(x)\quad\mbox{for any}\quad x\leq 0,
\end{equation}
where $\varphi_-:\RR\to\RR$ is defined by $\varphi_-(x)=\left\{\begin{array}{l l}
0 & \; \mbox{if $x\geq0$ }\\
e^{\alpha x} & \; \mbox{if $x<0$}\\
\end{array} \right.$. Summarizing, from \eqref{2.11-1} and \eqref{2.11-2} we obtain that
\begin{equation}\label{2.11-3}
e^{-\alpha|x|}(\cV g)(x)=\chi_{\RR_+}(x)\big(\varphi_+*(\chi_{\RR_+}h)\big)(x)+\chi_{\RR_-}(x)\big(\varphi_-*(\chi_{\RR_-}h)\big)(x)\quad\mbox{for any}\quad x\in\RR.
\end{equation}
Since $|\varphi_\pm(x)|\leq e^{-\alpha|x|}$ for any $x\in\RR$ and $h\in L^2(\RR,\bH)$, we infer that $\varphi_\pm*f\in L^2(\RR,\bH)$ and $\|\varphi_\pm*f\|_2\leq\frac{1}{\alpha}\|f\|_2$ for any $f\in L^2(\RR,\bH)$. From \eqref{2.11-3} it follows that $\cV g\in L^2_{-\alpha}(\RR,\bH)$ and
\begin{align}\label{2.11-4}
\|\cV g\|_{L^2_{-\alpha}}^2&=\|\chi_{\RR_+}\big(\varphi_+*(\chi_{\RR_+}h)\big)\|_2^2+\|\chi_{\RR_-}\big(\varphi_-*(\chi_{\RR_-}h)\big)\|_2^2\nonumber\\
&\leq \frac{1}{\alpha^2}\|\chi_{\RR_+}h\|_2^2+\frac{1}{\alpha^2}\|\chi_{\RR_-}h\|_2^2=\frac{1}{\alpha^2}\|h\|_2^2=\frac{1}{\alpha^2}\|g\|_{L^2_{-\alpha}}^2.
\end{align}
Furthermore, we have that $\cV g\in H^1_{\mathrm{loc}}(\RR,\bH)$ and $(\cV g)'=g\in L^2_{-\alpha}(\RR,\bH)$. From \eqref{2.11-4} we conclude that $\cV g\in H^1_{-\alpha}(\RR,\bH)$ and $\|\cV g\|_{H^1_{-\alpha}(\RR,\bH)}\leq (1+\alpha^{-2})^{1/2}\|g\|_{L^2_{-\alpha}(\RR,\bH)}$ for any $g\in L^2_{-\alpha}(\RR,\bH)$.
\end{proof}
Now we have all the ingredients needed to analyze the solutions of \eqref{linear-sys-perturbed}. We summarize our results in the following lemma.
\begin{lemma}\label{l2.11}
Assume Hypotheses (H1) and (H2). Then, the following assertions hold true.
\begin{enumerate}
\item[(i)] If $f\in L^2_{-\alpha}(\RR,\bV)$ and $\bu=(u_1,\tu,v_1,\tv)\in L^2_{-\alpha}(\RR,\bH)$ is a mild solution of \eqref{linear-sys-perturbed}, then
\begin{equation}\label{bu-fixed}
\bu(x)=P_\rmc\bu(0)+f_\rmc(x,v_1(0))+(\cV\Gamma_3f)(x)+\Gamma_4(\cK_0P_{\tbV}f)(x)\quad\mbox{for any}\quad x\in\RR,
\end{equation}
where, the linear operators $\Gamma_3,\Gamma_4:\bH\to\bH$ are defined by
\begin{equation}\label{def-Gamma-3-4}
\Gamma_3=(T_{12}^*)^{-1}P_{\bV_1}-E_1E_0^{-1}P_{\tbV}\in\cB(\bH),\quad\Gamma_4=\Big(\Gamma_1+E_1E_0^{-1}\Gamma_0-\tA_{11}^{-1}\tA_{12}+P_{\tbV}\Big)P_{\tbV}\in\cB(\bH).
\end{equation}
and the function $f_\rmc:\RR\times\bV_1\to\bH$ is defined by
\begin{equation}\label{def-f-star}
f_\rmc(x,v_1)=x E_1(I-E_0^{-1}P_{\tbV}E)v_1.
\end{equation}
\item[(ii)] The function $\bu=w_0+f_\rmc(\cdot,P_{\bV_1}w_0)+\cV\Gamma_3f+\Gamma_4\cK_0P_{\tbV}f$ is a $H^1_{-\alpha}$-solution of \eqref{linear-sys-perturbed} for any $f\in H^1_{-\alpha}(\RR,\bV)$ and $w_0\in\bH_\rmc$.
\end{enumerate}
\end{lemma}
\begin{proof}
Since \eqref{linear-sys-perturbed} is equivalent to \eqref{linear-sys-perturbed2}, we immediately conclude that $v_1(x)=v_1(0)$ for any $x\in\RR$, just like in the case of equation \eqref{linear}. Since any constant function belongs to $H^1_{-\alpha}(\RR)$, from the fourth equation of \eqref{linear-sys-perturbed2} and Lemma~\ref{l2.8}(ii) we obtain that
\begin{equation}\label{pert-tildev}
\tv=\cK_0(P_{\tbV}Ev_1(0)\rmone)+\cK_0P_{\tbV}f.
\end{equation}
Next, we integrate the second equation of \eqref{linear-sys-perturbed2} to find that
\begin{align}\label{pert-tildeu}
\tu(x)&=\tu(0)+\tA_{11}^{-1}\tA_{12}\big(\tv(0)-\tv(x)\big)=\tu(0)+\tA_{11}^{-1}\tA_{12}\tv(0)-\tA_{11}^{-1}\tA_{12}\cK_0(P_{\tbV}Ev_1(0)\rmone)(x)\nonumber\\
&\qquad\qquad-\tA_{11}^{-1}\tA_{12}(\cK_0P_{\tbV}f)(x)\quad\mbox{for any}\quad x\in\RR
\end{align}
Since $v_1\equiv v_1(0)\rmone$, combining the first and fourth equations we have that
\begin{equation}\label{inter-pert-u1}
u_1'=(\Gamma_1+E_1E_0^{-1}\Gamma_0)\tv'+E_1v_1(0)\rmone+(T_{12}^*)^{-1}P_{\bV_1}f-E_1E_0^{-1}P_{\tbV}(Ev_1(0)\rmone+f).
\end{equation}
Integrating, from \eqref{inter-pert-u1}, we conclude that
\begin{align}\label{pert-u1}
&u_1(x)=u_1(0)-(\Gamma_1+E_1E_0^{-1}\Gamma_0)\tv(0)+(\Gamma_1+E_1E_0^{-1}\Gamma_0)\cK_0(P_{\tbV}Ev_1(0)\rmone)(x)\nonumber\\
&+x E_1(I-E_0^{-1}P_{\tbV}E)v_1(0)+(\Gamma_1+E_1E_0^{-1}\Gamma_0)(\cK_0P_{\tbV}f)(x)+\cV\Big((T_{12}^*)^{-1}P_{\bV_1}f-E_1E_0^{-1}P_{\tbV}f\Big)(x)
\end{align}
for any $x\in\RR$.
Here $\cV$ denotes the Volterra operator defined by $(\cV g)(x)=\displaystyle\int_0^x g(y)\rmd y$.
From \eqref{pert-tildev}, \eqref{pert-tildeu}, \eqref{pert-u1} and \eqref{multiplier-constant} we conclude that
\begin{align}\label{solution-pert-linear}
\bu(x)&=u_1(0)-(\Gamma_1+E_1E_0^{-1}\Gamma_0)\tv(0)+\tu(0)+\tA_{11}^{-1}\tA_{12}\tv(0)+v_1(0)+x E_1(I-E_0^{-1}P_{\tbV}E)v_1(0)\nonumber\\
&\quad-\Big(\Gamma_1+E_1E_0^{-1}\Gamma_0-\tA_{11}^{-1}\tA_{12}+P_{\tbV}\Big)E_0^{-1}P_{\tbV}Ev_1(0)+\cV\Big((T_{12}^*)^{-1}P_{\bV_1}f-E_1E_0^{-1}P_{\tbV}f\Big)(x)\nonumber\\
&\quad+\Big(\Gamma_1+E_1E_0^{-1}\Gamma_0-\tA_{11}^{-1}\tA_{12}+P_{\tbV}\Big)(\cK_0P_{\tbV}f)(x)\nonumber\\
&=P_\rmc\bu(0)+f_\rmc(x,v_1(0))+(\cV\Gamma_3f)(x)+\Gamma_4(\cK_0P_{\tbV}f)(x)\quad\mbox{for any}\quad x\in\RR.
\end{align}
\noindent\textit{Proof of (ii)}. Fix $w_0\in\bH_\rmc$ and $f\in H^1_{-\alpha}(\RR,\bH)$. From Lemma~\ref{l2.6} and Lemma~\ref{r2.11} we have that $\cV\in\cB(L^2_{-\alpha}(\RR,\bV),H^1_{-\alpha}(\RR,\bV))$ and $\cK_0\in\cB(H^1_{-\alpha}(\RR,\tbV))$, which implies that $\bu=w_0+f_\rmc(\cdot,P_{\bV_1}w_0)+\cV\Gamma_3f+\Gamma_4\cK_0P_{\tbV}f\in H^1_{-\alpha}(\RR,\bH)$. Using the usual notation, we denote by $u_1=P_{\ker A_{11}}\bu$, $\tu=P_{\im A_{11}}\bu$, $v_1=P_{\bV_1}\bu$ and $\tv=P_{\tbV}\bu$. To prove assertion (ii) we will prove that $\bu=(u_1,\tu,v_1,\tv)^{\mathrm{T}}$ satisfies the system \eqref{linear-sys-perturbed2}.
From \eqref{def-Gamma-3-4} we obtain that
\begin{equation}\label{2.12-1}
P_{\ker A_{11}}\Gamma_3=(T_{12}^*)^{-1}P_{\bV_1}-E_1E_0^{-1}P_{\tbV},\quad P_{\im A_{11}}\Gamma_3=0,\quad P_{\bV_1}\Gamma_3=0,\quad P_{\tbV}\Gamma_3=0,
\end{equation}
\begin{equation}\label{2.12-2}
P_{\ker A_{11}}\Gamma_4=\Gamma_1+E_1E_0^{-1}\Gamma_0,\quad P_{\im A_{11}}\Gamma_4=-\tA_{11}^{-1}\tA_{12}P_{\tbV},\quad P_{\bV_1}\Gamma_4=0,\quad P_{\tbV}\Gamma_4=P_{\tbV}.
\end{equation}
Next, we multiply $\bu$ by the respective projections onto the orthogonal decomposition $\bH=\ker A_{11}\oplus\im A_{11}\oplus\bV_1\oplus\tbV$ and use \eqref{def-f-star}, \eqref{2.12-1} and \eqref{2.12-2} repeatedly. One can readily check that $v_1\equiv P_{\bV_1}w_0$, and thus $v_1'=0$. From \eqref{def-Vc} we have that $P_{\tbV}w_0=-E_0^{-1}P_{\tbV}EP_{\bV_1}w_0$. Hence, from \eqref{multiplier-constant} we obtain that
\begin{equation}\label{2.12-3}
\tv=-E_0^{-1}P_{\tbV}E P_{\bV_1}w_0\rmone+P_{\tbV}\cK_0P_{\tbV}f=\cK_0\big(P_{\tbV}EP_{\bV_1}w_0\rmone+P_{\tbV}f\big).
\end{equation}
Since $v_1\equiv P_{\bV_1}w_0$ and $P_{\tbV}EP_{\bV_1}w_0\rmone+P_{\tbV}f\in H^1_{-\alpha}(\RR,\tbV)$ from Lemma~\ref{l2.8}(i) we infer that $\Gamma_0\tv'=E_0\tv+P_{\tbV}E_1v_1+P_{\tbV}f$. In addition, we have that
\begin{align}\label{2.12-4}
\tu&=P_{\im A_{11}}w_0\rmone-\tA_{11}^{-1}\tA_{12}\cK_0P_{\tbV}f\nonumber\\
&=P_{\im A_{11}}w_0\rmone-\tA_{11}^{-1}\tA_{12}E_0^{-1}P_{\tbV}E P_{\bV_1}w_0\rmone-\tA_{11}^{-1}\tA_{12}\big(-E_0^{-1}P_{\tbV}E P_{\bV_1}w_0\rmone+\cK_0P_{\tbV}f\big)\nonumber\\
&=P_{\im A_{11}}w_0\rmone-\tA_{11}^{-1}\tA_{12}E_0^{-1}P_{\tbV}E P_{\bV_1}w_0\rmone-\tA_{11}^{-1}\tA_{12}\tv,
\end{align}
which implies that $\tu'=-\tA_{11}^{-1}\tA_{12}\tv'$. Finally, we have that
\begin{align}\label{2.12-5}
u_1(x)&=P_{\ker A_{11}}w_0+x  E_1(I-E_0^{-1}P_{\tbV}E)P_{\bV_1}w_0+\big((T_{12}^*)^{-1}P_{\bV_1}-E_1E_0^{-1}P_{\tbV}\big)(\cV f)(x)\nonumber\\
&\qquad\qquad+(\Gamma_1+E_1E_0^{-1}\Gamma_0)(\cK_0P_{\tbV}f)(x)\nonumber\\
&=P_{\ker A_{11}}w_0+x E_1(I-E_0^{-1}P_{\tbV}E)P_{\bV_1}w_0+\big((T_{12}^*)^{-1}P_{\bV_1}-E_1E_0^{-1}P_{\tbV}\big)(\cV f)(x)\nonumber\\
&\qquad\qquad+(\Gamma_1+E_1E_0^{-1}\Gamma_0)\tv(x)+(\Gamma_1+E_1E_0^{-1}\Gamma_0)E_0^{-1}P_{\tbV}E P_{\bV_1}w_0
\end{align}
for any $x\in\RR$. Differentiating in \eqref{2.12-5} and since $\tv=E_0^{-1}\Gamma_0\tv'-E_0^{-1}P_{\tbV}f-E_0^{-1}P_{\tbV}E_1v_1$ it follows that
\begin{align}\label{2.12-6}
u_1'&=E_1(I-E_0^{-1}P_{\tbV}E)P_{\bV_1}w_0\rmone+(\Gamma_1+E_1E_0^{-1}\Gamma_0)\tv'+\big((T_{12}^*)^{-1}P_{\bV_1}-E_1E_0^{-1}P_{\tbV}\big) f\nonumber\\
&=\Gamma_1\tv'+E_1\Big(v_1+E_0^{-1}\Gamma_0\tv'-E_0^{-1}P_{\tbV}E v_1-E_0^{-1}P_{\tbV}f\Big)+(T_{12}^*)^{-1}P_{\bV_1} f\nonumber\\
&=\Gamma_1\tv'+E_1(v_1+\tv)+(T_{12}^*)^{-1}P_{\bV_1} f,
\end{align}
proving the lemma.
\end{proof}

\section{Center Manifold construction}\label{CM}
In this section we construct a center manifold of solutions of equation \eqref{nonlinear} tangent to the center subspace $\bH_\rmc$ at $\obu$ expressible in coordinates $\bw=\bu-\obu$. Throughout this section we assume Hypotheses (H1) and (H2).
Making the change of variables $\bw=\bu-\obu$ in \eqref{nonlinear} and since $Q(\bu)=B(\bu,\bu)$ is a bilinear map on $\bH$, we obtain the equation
\begin{equation}\label{center-nonlinear}
A\bw'=2B(\obu,\bw)+Q(\bw).
\end{equation}
Denoting by $u=P_{\bV^\perp}\bw$ and $v=P_{\bV}\bw$, since $E=Q'(\obu)_{|\bV}=2B(\cdot,\obu)_{|\bV}$ we have that equation \eqref{center-nonlinear} is equivalent to the system
\begin{equation}\label{nonlinear-sys-perturbed}
\left\{\begin{array}{ll} A_{11}u'+A_{12}v'=0,\\
A_{21}u'+A_{22}v'=Ev+B(u+v,u+v). \end{array}\right.
\end{equation}
Using Lemma~\ref{l2.11}(i), we immediately conclude that system \eqref{nonlinear-sys-perturbed} is equivalent to
\begin{equation}\label{fix-point}
\bw=P_\rmc\bw(0)+f_\rmc(\cdot,P_{\bV_1}\bw(0))+\cK B(\bw,\bw),
\end{equation}
where $f_\rmc$ is defined in \eqref{def-f-star} and $\cK f=\cV\Gamma_3f+\Gamma_4\cK_0P_{\tbV}f$. From Lemma~\ref{l2.6} and Lemma~\ref{r2.11} we have that $\cK$ is a bounded linear operator on $L^2_{-\alpha}(\RR,\bH)$ and
$H^1_{-\alpha}(\RR,\bH)$ for any $\alpha\in(0,\nu(\Gamma_0,E_0))$, where $-\nu(\Gamma_0,E_0)$ is the decay rate of the bi-semigroup generated by $S_{\Gamma_0,E_0}=\Gamma_0^{-1}E_0$. To prove the existence of our center manifold we will prove that equation \eqref{fix-point} has a unique solution in $H^1_{-\alpha}(\RR,\bH)$ for any $w_0=P_\rmc\bw(0)\in\bH_\rmc$ small enough.
Unlike the stable manifold construction done
in the earlier paper \cite{PZ1}, we will carry this out on all of $\RR$.

We point out that $B(\bw,\bw)$ might not belong to $H^1_{-\alpha}(\RR,\bH)$ for all $\bw\in H^1_{-\alpha}(\RR,\bH)$. Therefore, we cannot use the Contraction Mapping Theorem to prove existence and uniqueness of solutions of \eqref{fix-point} right away. To overcome this difficulty, we localize the problem by using the truncation of the nonlinearity technique, which is used in a variety of situations such as the construction of finite-dimensional center manifolds or the Hartman-Grobman Theorem. Unlike the
finite-dimensional case where the fixed point argument is done in the space of continuous functions growing slower than $e^{\alpha|\cdot|}$ with the usual negatively weighted supremum norm, in our case it would be ideal to work on $H^1_{-\alpha}(\RR,\bH)$ with its usual norm. However, it is not clear if the estimates needed to prove that the cutoff nonlinearity is a strict contraction are possible, since some of the terms obtained by differentiating the cutoff nonlinearity are neither small nor in the correct weighted space. Instead it seems one must prove existence of solutions of equation \eqref{fix-point} using an approximation argument, and then prove uniqueness by invoking the weak-star compactness of unit balls in Hilbert spaces.
This may be recognized as the standard variant of Picard iteration used in quasilinear hyperbolic theory, e.g.,
a ``bounded high norm/contractive low norm'' version of the Banach fixed-point theorem \cite{Ma}.

\subsection{Existence of a center manifold of solutions}\label{CM-existence}
We note that equation \eqref{fix-point} is equivalent to
\begin{equation}\label{augmented-var}
\bw=F_\rmc(\cdot,P_\rmc\bw(0))+\cK N(\bw),
\end{equation}
where the functions $F_\rmc:\RR\times\bH_\rmc\to\bH$ and $N:\bH\to\bH$ are defined by
\begin{equation}\label{def-F-star}
F_\rmc(x,w_0)=w_0+f_\rmc(x,P_{\bV_1}w_0),\quad N(h)=B(h,h).
\end{equation}
Next, we introduce the truncated system which is better suited to apply a contraction mapping like argument. Let $\rho\in C^\infty_0(\RR)$ be a smooth function such that $\rho(s)=1$ for any $s\in [-1,1]$ and $\rho(s)=0$ whenever $|s|\geq 2$. In this section we will prove existence and uniqueness properties of solutions of equation
\begin{equation}\label{fix-point-truncated}
\bw=F_\rmc(\cdot,w_0)+\cN_\eps(\bw),\quad\mbox{for}\quad w_0\in\bH_\rmc,\;\eps>0
\end{equation}
where the function $\cN_\eps: H^1_{-\alpha}(\RR)\to L^2_{-\alpha}(\RR,\bH)$ is defined by
\begin{equation}\label{def-cNeps}
\cN_\eps(\bw)=\cK\Big[\rho\Big(\frac{\|\bw\|^2}{\eps^2}\Big)B(\bw,\bw)\Big].
\end{equation}
We note that any solution $\bw$ of \eqref{fix-point-truncated}, small enough, is a solution of \eqref{fix-point}. To check that we can apply a contraction mapping-type
argument on \eqref{fix-point-truncated}, we need to check the properties of the function $N_\eps:\bH\to\bH$ defined by
\begin{equation}\label{def-Neps}
N_\eps(h)=\rho\Big(\frac{\|h\|^2}{\eps^2}\Big)B(h,h).
\end{equation}
\begin{remark}\label{r3.1}
Since $B$ is a bilinear map on $\bH$ and $\rho\in C^\infty_0(\RR)$ we have that $N_\eps$ is of class $C^\infty$ on $\bH$ for any $\eps>0$. Moreover, one can readily check that
\begin{equation}\label{partial-Neps}
N_\eps'(h)g=\frac{2}{\eps^2}\rho'\Big(\frac{\|h\|^2}{\eps^2}\Big)\langle h,g\rangle B(h,h)+2\rho\Big(\frac{\|h\|^2}{\eps^2}\Big)B(h,g)
\end{equation}
for any $h,g\in\bH$. Since $\rho(s)=0$ whenever $|s|\geq 2$, from \eqref{partial-Neps} it follows that there exists $c>0$ such that
\begin{equation}\label{deriv-Neps-est}
\|N_\eps'(h)\|\leq c\eps,\quad \|N_\eps''(h)\|\leq c\quad\mbox{for any}\quad h\in\bH,\;\eps>0.
\end{equation}
\end{remark}

A crucial role in our construction is played by the following mixed-norm function space: for any $\beta>\gamma>0$ we define $\cZ_{\gamma,\beta}(\bH)=L^2_{-\gamma}(\RR,\bH)\cap H^1_{-\beta}(\RR,\bH)$ endowed with the norm
\begin{equation}\label{mixed-norm}
\|f\|_{\cZ_{\gamma,\beta}}=\Big(\|f\|_{L^2_{-\gamma}}^2+\|f'\|_{L^2_{-\beta}}^2\Big)^{1/2}.
\end{equation}
We note that this norm is equivalent to the $(\|\cdot\|_{L^2_{-\gamma}}^2+\|\cdot\|_{H^1_{-\beta}}^2)^{1/2}$ norm on $\cZ_{\gamma,\beta}(\bH)$, therefore it induces a Hilbert space structure on $\cZ_{\gamma,\beta}(\bH)$.
We note also the following $L^\infty$ Sobolev embedding estimate.\footnote{
	Though we shall not use it, this implies evidently the $L^2$ embedding
$\cZ_{\gamma,\beta}(\bH)\hookrightarrow L^2_{-\nu}(\bH)$ for any $\nu> (\beta+\gamma)/2$. }

\begin{lemma}\label{lsob}
$\cZ_{\gamma,\beta}(\bH)\hookrightarrow L^\infty_{-(\beta +\gamma)/2}(\bH)$;
	equivalently, for any $x\in\RR$, $f\in\cZ_{\gamma,\beta}(\bH)$,
\begin{align}\label{sob}
e^{-2\beta|x|}\|f(x)\|^2
\leq c e^{-(\beta-\gamma)|x|}\|f\|_{\cZ_{\gamma,\beta}}^2.
\end{align}
\end{lemma}
\begin{proof}
Take without loss of generality $x\geq 0$, and $e^{-\beta \cdot} f(\cdot)$ in the Schwartz class on $\R_+$.  Then,
$$
\begin{aligned}
	e^{-2\beta x} \|f(x)\|^2 &=\int_x^\infty (e^{-2\beta y} \|f(y)\|^2)'\rmd y
=
-2\beta \int_x^\infty e^{-2\beta y} \|f(y)\|^2 \rmd y
+
2\int_x^\infty e^{-2\beta y} \langle f(y), f'(y)\rangle \rmd y\\
&\leq c
e^{-2(\beta-\gamma) x}\int_x^\infty e^{-2\gamma y} \|f(y)\|^2 \rmd y
+
e^{-(\beta-\gamma) x} \int_x^\infty (e^{-\gamma y} \|f(y)\|)(e^{-\beta y} \|f'(y)\|) \rmd y,
\end{aligned}
$$
from which the result follows by Cauchy-Schwarz/Young's inequality.
\end{proof}


In the next lemma we collect
some
estimates satisfied by the map $\cN_\eps$ on $H^1_{-\alpha}(\RR,\bH)$ and $\cZ_{\gamma,\beta}(\bH)$.
\begin{lemma}\label{l3.2}
Assume Hypotheses (H1) and (H2) and let $0<\alpha<\gamma<\beta<\frac{1}{2}\nu(\gamma_0,E_0)$. Then, the following estimates hold true:
\begin{enumerate}
\item[(i)] $F_\rmc(\cdot,w_0)\in H^1_{-\alpha}(\RR,\bH)$ and $\|F_\rmc(\cdot,w_0)\|_{H^1_{-\alpha}}\leq c\|w_0\|$ for any $w_0\in\bH_\rmc$;
\item[(ii)] $\cN_\eps(\bw)\in H^1_{-\alpha}(\RR,\bH)$ for any $\bw\in H^1_{-\alpha}(\RR,\bH)$ and $\eps>0$. Moreover,
\begin{equation}\label{cNeps-H1-bound}
\|\cN_\eps(\bw)\|_{H^1_{-\alpha}}\leq c\eps \|\bw\|_{H^1_{-\alpha}}\quad\mbox{for any}\quad\bw\in H^1_{-\alpha}(\RR,\bH);
\end{equation}
\item[(iii)] $\cN_\eps(\bw)\in \cZ_{\gamma,\beta}(\bH)$ for any $\bw\in \cZ_{\gamma,\beta}(\bH)$ and $\eps>0$. In addition,
\begin{equation}\label{cNeps-cZ-bound}
\|\cN_\eps(\bw)\|_{\cZ_{\gamma,\beta}}\leq c\eps \|\bw\|_{\cZ_{\gamma,\beta}}\quad\mbox{for any}\quad\bw\in \cZ_{\gamma,\beta}(\bH);
\end{equation}
\item[(iv)] If $0<2\alpha<\beta-\gamma$, then for any $\delta>0$ and any $\bw_1,\bw_2\in H^1_{-\alpha}(\RR,\bH)\subset\cZ_{\gamma,\beta}(\bH)$ such that $\|\bw_1\|_{H^1_{-\alpha}}\leq \delta$ and $\|\bw_2\|_{H^1_{-\alpha}}\leq \delta$ we have that
\begin{equation}\label{cNeps-Lipschitz}
\|\cN_\eps(\bw_1)-\cN_\eps(\bw_2)\|_{\cZ_{\gamma,\beta}}\leq c(\eps+\delta) \|\bw_1-\bw_2\|_{\cZ_{\gamma,\beta}}.
\end{equation}
\end{enumerate}
\end{lemma}
\begin{proof} From \eqref{def-f-star} and \eqref{def-F-star} we have that $F_\rmc$ is an affine function. Moreover, from \eqref{center-proj} it follows that the projection $P_\rmc$ is bounded, which proves (i). Since $B$ is a bilinear map on $\bH$, from \eqref{def-Neps} we have that
\begin{equation}\label{3.2-1}
\|N_\eps(h)\|\leq \|B\|\rho\Big(\frac{\|h\|^2}{\eps^2}\Big)\|h\|^2\leq (2\|B\|\eps)\|h\|\quad\mbox{for any}\quad h\in\bH.
\end{equation}
Since the linear operator $\cK$ can be extended to a bounded linear operator on $L^2_{-\varkappa}(\RR,\bH)$ for $\varkappa=\alpha,\gamma$ by Lemma~\ref{l2.6} and Lemma~\ref{r2.11}, from \eqref{3.2-1} we have that
\begin{equation}\label{3.2-2}
\|\cN_\eps(\bw)\|_{L^2_{-\varkappa}}\leq c\eps \|\bw\|_{L^2_{-\varkappa}}\quad\mbox{for any}\quad\bw\in L^2_{-\varkappa}(\RR,\bH),\quad\mbox{for}\quad\varkappa=\alpha,\gamma.
\end{equation}
Since $N_\eps$ is of class $C^\infty$ by Remark~\ref{r3.1}, we obtain that $N_\eps\circ\bw\in H^1_{\mathrm{loc}}(\RR,\bH)$ and
$\big(N_\eps\circ\bw\big)'=(N_\eps'\circ\bw)\bw'$ for any $\bw\in H^1_{\mathrm{loc}}(\RR,\bH)$. From \eqref{deriv-Neps-est} it follows that
\begin{equation}\label{3.2-3}
\big\|\big(\cN_\eps(\bw)\big)'\big\|_{L^2_{-\varkappa}}=\|(N_\eps'\circ\bw)\bw'\|_{L^2_{-\varkappa}}\leq
\|N_\eps'\circ\bw\|_\infty\|\bw'\|_{L^2_{-\varkappa}}\leq c\eps\|\bw'\|_{L^2_{-\varkappa}}.
\end{equation}
for any $\bw\in H^1_{-\varkappa}(\RR,\bH)$ for $\varkappa=\alpha,\beta$. Using again Lemma~\ref{l2.6} and Lemma~\ref{r2.11}, one can readily check that the linear operator $\cK$ can be extended to a bounded linear operator on $H^1_{-\alpha}(\RR,\bH)$ and $\cZ_{\gamma,\beta}(\bH)$. Assertions (ii) and (iii) follow shortly from \eqref{3.2-2} and \eqref{3.2-3}.

To start the proof of (iv), we fix $\bw_1,\bw_2\in H^1_{-\alpha}(\RR,\bH)$,  $\|\bw_1\|_{H^1_{-\alpha}}\leq \delta$ and $\|\bw_2\|_{H^1_{-\alpha}}\leq \delta$. Since $N_\eps$ is of class $C^\infty$ on $\bH$ by Remark~\ref{r3.1}, from \eqref{deriv-Neps-est} it follows that there exists a constant $c>0$ independent of $\eps>0$ such that
\begin{equation}\label{3.2-4}
\|N_\eps(h_1)-N_\eps(h_2)\|\leq c\eps \|h_1-h_2\|\;\mbox{and}\;\|N_\eps'(h_1)-N_\eps'(h_2)\|\leq c\|h_1-h_2\|\;\mbox{for any}\; h_1,h_2\in\bH,
\end{equation}
which implies that
\begin{equation}\label{3.2-5}
\|N_\eps\circ\bw_1-N_\eps\circ\bw_2\|_{L^2_{-\gamma}}\leq c\eps\|\bw_1-\bw_2\|_{L^2_{-\gamma}}\leq c\eps\|\bw_1-\bw_2\|_{\cZ_{\gamma,\beta}}.
\end{equation}
Since $2\alpha<\beta-\gamma$, from \eqref{sob} and \eqref{3.2-4} we obtain that
\begin{align}\label{3.2-6}
\Big\|\Big(N_\eps\circ\bw_1&-N_\eps\circ\bw_2\Big)'\Big\|_{L^2_{-\beta}}^2\leq 2\|(N_\eps'\circ\bw_1)(\bw_1'-\bw_2')\|_{L^2_{-\beta}}^2+2\|(N_\eps'\circ\bw_1-N_\eps'\circ\bw_2)\bw_2'\|_{L^2_{-\beta}}^2\nonumber\\
&\leq c^2\eps^2\|\bw_1'-\bw_2'\|_{L^2_{-\beta}}^2+c^2\int_\RR e^{-2\beta|x|} \|\big(N_\eps'(\bw_1(x))-N_\eps'(\bw_2(x))\big)\bw_2'(x)\|^2\,\rmd x\nonumber\\
&\leq c^2\eps^2\|\bw_1-\bw_2\|_{\cZ_{\gamma,\beta}}^2+c^2\int_\RR e^{-2\beta|x|} \|\bw_1(x)-\bw_2(x)\|^2\|\bw_2'(x)\|^2\,\rmd x\nonumber\\
&\leq c^2\eps^2\|\bw_1-\bw_2\|_{\cZ_{\gamma,\beta}}^2+c^2\Big(\int_\RR e^{-(\beta-\gamma)|x|} \|\bw_2'(x)\|^2\,\rmd x\Big)\|\bw_1-\bw_2\|_{\cZ_{\gamma,\beta}}^2\nonumber\\
&\leq c^2\eps^2\|\bw_1-\bw_2\|_{\cZ_{\gamma,\beta}}^2+c^2\Big(\int_\RR e^{-2\alpha|x|} \|\bw_2'(x)\|^2\,\rmd x\Big)\|\bw_1-\bw_2\|_{\cZ_{\gamma,\beta}}^2\nonumber\\
&\leq c^2\eps^2\|\bw_1-\bw_2\|_{\cZ_{\gamma,\beta}}^2+c^2\|\bw_2\|_{H^1_{-\alpha}}^2\|\bw_1-\bw_2\|_{\cZ_{\gamma,\beta}}^2\leq c^2(\eps^2+\delta^2)\|\bw_1-\bw_2\|_{\cZ_{\gamma,\beta}}^2.
\end{align}
Since $\cK$ can be extended to a bounded linear operator on $\cZ_{\gamma,\beta}(\bH)$ and
$$\sup_{0<\alpha<\frac{1}{2}\nu(\Gamma_0,E_0)}\|\cK\|_{\cB(H^1_{-\alpha}(\RR,\bH))}<\infty,$$
assertion (iv) follows from \eqref{3.2-5} and \eqref{3.2-6}, proving the lemma.
\end{proof}
\begin{remark}\label{keyrmk}
We note that the corresponding $H^1_{-\beta}$ version of \eqref{cNeps-Lipschitz} {\it does not} hold. Indeed,
$$
\begin{aligned}
&\|(N_\eps\circ\bw_1-N_\eps\circ\bw_2)'\|_{L^2_{-\beta}}^2 + 2\|(N_\eps'\circ\bw_1)(\bw_1'-\bw_2')\|_{L^2_{-\beta}}^2
\geq 2\|(N_\eps'\circ\bw_1-N_\eps'\circ\bw_2)\bw_2'\|_{L^2_{-\beta}}^2 \\
&=  2\int_\RR e^{-2\beta|x|} \| \Big(\big(\int_0^1N_\eps''(s\bw_1(x)+(1-s)\bw_2(x))\rmd s\big)(\bw_1(x)- \bw_2(x))\Big)\bw_2'(x)\|^2 \,\rmd x  \\
\end{aligned}
$$
with $\|(N_\eps'\circ\bw_1)(\bw_1'-\bw_2')\|_{L^2_{-\beta}}^2 \leq c\eps \|\bw_1-\bw_2\|_{H^1_{-\beta}}$ shows
that $\cN_\eps$ is in general neither contractive in $H^1_{-\beta}(\RR,\bH)$ (since $ N_\eps''$ is merely bounded, not small)
nor even Lipschitz (since $e^{-2\beta|x|}$ can compensate for growth of either
$ \|\bw_1(x)- \bw_2(x)\|^2$ or $\|\bw_2'(x)\|^2$, but not both, in the integral on the righthand side).
\end{remark}
At this point we fix $0<\alpha\ll\frac{1}{2}\nu(\gamma_0,E_0)$ and $\gamma<\beta<\frac{1}{2}\nu(\gamma_0,E_0)$ such that $0<2\alpha<\beta-\gamma$.
Next, we introduce the function $\cT_\eps:\bH_\rmc\times H^1_{-\alpha}(\RR,\bH)\to H^1_{-\alpha}(\RR,\bH)$ defined by $\cT_\eps(w_0,\bw)=F_\rmc(\cdot,w_0)+\cN_\eps(\bw)$. We note that the function $\cT_\eps$ is well-defined by Lemma~\ref{l3.2}(ii). We are now ready to state our existence and uniqueness result.
\begin{lemma}\label{l3.3}
Assume Hypotheses (H1) and (H2). Then, there exists $\eps_0>0$ such that for any $\delta>0$ small enough there exists $\eps_1:=\eps_1(\delta)>0$ such that for any $w_0\in\overline{B}_{\bH_\rmc}(0,\eps_1)$ the equation $\bw=\cT_{\eps_0}(w_0,\bw)$ has a unique solution in $H^1_{-\alpha}(\RR,\bH)$, denoted $\obw(\cdot,w_0)$, satisfying the condition
\begin{equation}\label{obu-delta-estimate}
\|\obw(\cdot,w_0)\|_{H^1_{-\alpha}(\RR,\bH)}\leq\delta\quad\mbox{for any}\quad w_0\in\overline{B}_{\bH_\rmc}(0,\eps_1).
\end{equation}
\end{lemma}
\begin{proof} From Lemma~\ref{l3.2}(i) and (ii) have that
\begin{equation}\label{3.3-1}
\|\cT_\eps(w_0,\bw)\|_{H^1_{-\alpha}}\leq c\|F_\rmc(\cdot,w_0)\|_{H^1_{-\alpha}}+\|\cN_{\eps}(\bw)\|_{H^1_{-\alpha}}\leq c\|w_0\|+c\eps \|\bw\|_{H^1_{-\alpha}}\leq c\eps_1+c\eps\delta
\end{equation}
for any $\eps>0$, $\delta>0$, $w_0\in\overline{B}_{\bH_\rmc}(0,\eps_1)$ and $\bw\in\overline{B}_{H^1_{-\alpha}(\RR,\bH)}(0,\delta)$. Here the constant $c>0$ is independent of $\alpha>0$.
We choose $\eps_0>0$ and $\eps_1=\eps_1(\delta)$ such that $c\eps_0<\frac{1}{4}$ and $c\eps_1<\frac{\delta}{2}$. From \eqref{3.3-1} we obtain that
\begin{equation}\label{3.3-2}
\|\cT_{\eps_0}(w_0,\bw)\|_{H^1_{-\alpha}}\leq\delta\quad\mbox{for any}\quad w_0\in\overline{B}_{\bH_\rmc}(0,\eps_1),\;\bw\in\overline{B}_{H^1_{-\alpha}(\RR,\bH)}(0,\delta).
\end{equation}
Moreover, from Lemma~\ref{l3.2}(iv) we have that
\begin{align}\label{3.3-3}
\|\cT_{\eps_0}(w_0,\bw_1)&-\cT_{\eps_0}(w_0,\bw_2)\|_{\cZ_{\gamma,\beta}}=\|\cN_{\eps_0}(\bw_1)-\cN_{\eps_0}(\bw_2)\|_{\cZ_{\gamma,\beta}}\leq c(\eps_0+\delta) \|\bw_1-\bw_2\|_{\cZ_{\gamma,\beta}}\nonumber\\&\leq \frac{1}{2} \|\bw_1-\bw_2\|_{\cZ_{\gamma,\beta}}\quad\mbox{for any}\quad w_0\in\overline{B}_{\bH_\rmc}(0,\eps_1),\;\bw_1,\bw_2\in\overline{B}_{H^1_{-\alpha}(\RR,\bH)}(0,\delta)
\end{align}
and any $\delta>0$ such that $c\delta<\frac{1}{4}$. To prove the existence of solutions of equation $\bw=\cT_{\eps_0}(w_0,\bw)$ in $H^1_{-\alpha}(\RR,\bH)$ we fix $\delta>0$ small enough and $w_0\in\overline{B}_{\bH_\rmc}(0,\eps_1(\delta))$. Also,  we introduce the sequence $(\obw_n)_{n\geq1}$ defined recursively by the equation
\begin{equation}\label{3.3-4}
\obw_{n+1}=\cT_{\eps_0}(w_0,\obw_n),\quad n\geq 1,\quad\mbox{and}\quad \obw_1=0.
\end{equation}
Using induction, from \eqref{3.3-2} we obtain that
\begin{equation}\label{3.3-5}
\sup_{n\geq 1}\|\obw_n\|_{H^1_{-\alpha}}\leq\delta.
\end{equation}
Moreover, from \eqref{3.3-3} it follows that
\begin{equation}\label{3.3-6}
\|\obw_{n+1}-\obw_n\|_{\cZ_{\gamma,\beta}}=\|\cT_{\eps_0}(w_0,\bw_n)-\cT_{\eps_0}(w_0,\bw_{n-1})\|_{\cZ_{\gamma,\beta}}\leq\frac{1}{2}\|\bw_n-\bw_{n-1}\|_{\cZ_{\gamma,\beta}}\;\mbox{for any}\;n\geq 1.
\end{equation}
Using induction again, we conclude that $\|\obw_{n+1}-\obw_n\|_{\cZ_{\gamma,\beta}}\leq 2^{-n}\|\obw_2\|_{\cZ_{\gamma,\beta}}$ for any $n\geq 1$, which implies that there exists $\obw(\cdot,w_0)\in\cZ_{\gamma,\beta}(\bH)$ such that
\begin{equation}\label{3.3-7}
\obw_n\to\obw(\cdot,w_0)\quad\mbox{in}\quad \cZ_{\gamma,\beta}(\bH)\quad\mbox{as}\quad n\to\infty.
\end{equation}
Since the closed ball of any Hilbert space is weakly compact, from \eqref{3.3-5} we infer that there exists a subsequence $(\obw_{n_k})_{k\geq 1}$ that is weakly convergent to an element of $\overline{B}_{H^1_{-\alpha}(\RR,\bH)}(0,\delta)$. From \eqref{3.3-7} it follows that $\obw(\cdot,w_0)\in \overline{B}_{H^1_{-\alpha}(\RR,\bH)}(0,\delta)$, proving the existence and the estimate \eqref{obu-delta-estimate}. From \eqref{3.3-3} it follows that  equation $\bw=\cT_{\eps_0}(w_0,\bw)$ cannot have more than one solution in $H^1_{-\alpha}(\RR,\bH)$, proving the lemma.
\end{proof}
\begin{lemma}\label{l3.4}
Assume Hypotheses (H1) and (H2). Then, there exists $\eta_0>0$ such that $\obw(\cdot,w_0)$ is a solution of equation \eqref{center-nonlinear} on $(-\eta_0,\eta_0)$ for any $w_0\in\overline{B}_{\bH_\rmc}(0,\eps_1(\delta_0))$.
\end{lemma}
\begin{proof}
First, we recall the conditions satisfied by $\eps_0>0$ and $\eps_1=\eps_1(\delta)$ imposed in the proof of Lemma~\ref{l3.2}(iii): $c\eps_0<\frac{1}{4}$, $c\eps_1(\delta)<\frac{\delta}{2}$ and $c\delta<\frac{1}{4}$. Moreover, we have that the constant $c>0$ can be chosen big enough such that
\begin{equation}\label{3.4-1}
\|f(x)\|\leq ce^{\alpha|x|}\|f\|_{H^1_{-\alpha}}\quad\mbox{for any}\quad x\in\RR,\;f\in H^1_{-\alpha}(\RR,\bH).
\end{equation}
Next, we choose $\delta_0>0$ such that $c\delta_0<\frac{\eps_0}{2}$. From \eqref{obu-delta-estimate} and \eqref{3.4-1} we obtain that
\begin{equation}\label{3.4-2}
\|\obw(x,w_0)\|\leq \frac{\eps_0}{2}e^{\alpha|x|}\quad\mbox{for any}\quad x\in\RR, w_0\in\overline{B}_{\bH_\rmc}(0,\eps_1(\delta_0)).
\end{equation}
It follows that there exists $\eta_0>0$ such that
\begin{equation}\label{3.4-3}
\|\obw(x,w_0)\|\leq\eps_0\quad\mbox{for any}\quad x\in (-\eta_0,\eta_0), w_0\in\overline{B}_{\bH_\rmc}(0,\eps_1(\delta_0)).
\end{equation}
From Lemma~\ref{l3.3} and the definition of $\cN_{\eps_0}$ in \eqref{def-cNeps} we have that
\begin{equation}\label{3.4-5}
\obw(\cdot,w_0)=w_0+f_\rmc(\cdot,P_{\bV_1}w_0)+\cV\Gamma_3\obf+\Gamma_4\cK_0P_{\tbV}\obf,\quad\mbox{where}\quad\obf:=N_{\eps_0}\circ\obw(\cdot,w_0).
\end{equation}
Since $\obw(\cdot,w_0)\in H^1_{-\alpha}(\RR,\bH)$, and hence $\obf\in H^1_{-\alpha}(\RR,\bH)$, from Lemma~\ref{l2.11}(ii) we infer that $\obw(\cdot,w_0)$ satisfies the equation
\begin{equation}\label{3.4-6}
A\bw'(x)=Q'(\obu)\bw(x)+\obf(x)\quad\mbox{for any}\quad x\in\RR.
\end{equation}
Since the cut-off function $\rho$ is identically equal to one on $[-1,1]$, from \eqref{def-Neps} and \eqref{3.4-3} we have that $\obf(x)=B(\obw(x,w_0),\obw(x,w_0))$ for any $x\in (-\eta_0,\eta_0)$. From \eqref{3.4-6} we conclude that $\obw(\cdot,w_0)$ is a solution of equation \eqref{center-nonlinear} on $(-\eta_0,\eta_0)$ for any $w_0\in\overline{B}_{\bH_\rmc}(0,\eps_1(\delta_0))$, proving the lemma.
\end{proof}
We are now ready to introduce the center manifold defined by the trace at $x=0$ of the fixed point solution $\obw(\cdot,w_0)$:
\begin{equation}\label{center-manifold}
\cM_\rmc=\{\obw(0,w_0):w_0\in\overline{B}_{\bH_\rmc}(0,\eps_1(\delta_0))\}.
\end{equation}
\begin{lemma}\label{r3.5} Assume Hypotheses (H1) and (H2). Then, the following assertions hold true:
\begin{enumerate}
\item[(i)] $P_\rmc\obw(0,w_0)=w_0$ for any $w_0\in\overline{B}_{\bH_\rmc}(0,\eps_1(\delta_0))$;
\item[(ii)]
	$\cM_\rmc=\mathrm{Graph}(\cJ_\rmc)$, where $\cJ_\rmc:\overline{B}_{\bH_\rmc}(0,\eps_1(\delta_0))\to\bH_\rms\oplus\bH_\rmu$ is the function defined by
\begin{equation}\label{def-Jc}
\cJ_\rmc(w_0)=(I_{\bH}-P_\rmc)\obw(0,w_0).
\end{equation}
\end{enumerate}
\end{lemma}
\begin{proof} (i)
From \eqref{center-proj} and \eqref{def-Gamma-3-4} one can readily check that
\begin{equation}\label{3.5-1}
P_\rmc\Gamma_3=\Gamma_3,\;P_\rmc\Gamma_4\tv=0\;\mbox{for any}\;\tv\in\tbV,\;P_\rmc f_\rmc(x,v_1)=x E_1(I-E_0^{-1}P_{\tbV}E)v_1
\end{equation}
for any $x\in\RR$, $v_1\in\bV_1$. Multiplying by $P_\rmc$ in \eqref{3.4-5} and since $(\cV f)(0)=0$ for any $f\in L^2_{-\alpha}(\RR,\bH)$, we conclude that $P_\rmc\obw(0,w_0)=w_0$ for any $w_0\in\overline{B}_{\bH_\rmc}(0,\eps_1(\delta_0))$. Assertion (ii) follows immediately from (i).
\end{proof}
\begin{remark}\label{holderrmk}
Up to this point we could as well have used contraction of the fixed-point mapping in $L^2_{-\alpha}(\RR,\bH)$
and boundedness in $H^1_{-\alpha}(\RR,\bH)$.  This gives (by continuous dependence on parameters of contraction-mapping solutions)
$L^2_{-\alpha}(\RR,\bH)$ Lipschitz continuity of $\obw$ with respect to $w_0$, and
boundedness in $H^1_{-\alpha}(\RR,\bH)$,
yielding $C_{\mathrm{b}}^{1/2}(\RR,\bH)$ H\"older continuity of $\cJ_\rmc$ by the Sobolev embedding
$\|f(0)\|^2\lesssim \|f\|_{L^2_{-\alpha}} \|f'\|_{L^2_{-\alpha}}$.
Contraction in $\cZ_{\gamma,\beta}(\bH)$,
based on Lemma \ref{l3.2}(iv), is used to obtain Lipschitz and higher regularity.
\end{remark}

\subsection{$C^k$ smoothness of the center manifold}\label{CM-smoothness}
Our next task is to prove that the manifold $\cM_\rmc$ is smooth by showing that the function
\begin{equation}\label{def-Sigma-c}
\Sigma_\rmc:\overline{B}_{\bH_\rmc}(0,\eps_1(\delta_0))\to\overline{B}_{H^1_{-\alpha}(\RR,\bH)}(0,\delta_0)\;\mbox{defined by}\;\Sigma_\rmc(h)=\obw(\cdot,h)
\end{equation}
is of class $C^k$ in the $\cZ_{\gamma,\beta}(\bH)$ topology, for some appropriate weights $\gamma$ and $\beta$. To prove this result, we first prove that the function $\Sigma_\rmc$ is of class $C^1$ and we find a formula for its first order partial derivatives.
Building on this result, we then prove
higher order differentiability using the smoothness properties of substitution operators studied in Appendix~\ref{appendix}.

Our argument follows that used in \cite{Z1} (\cite{Z2}) to establish smoothness of center-stable (center) manifolds in the usual
$C_{\mathrm{b}}$ setting, using a general result on smooth dependence with respect to parameters of
a fixed point mapping $y=T(x,y)$ that is Fr\'echet differentiable in $y$ from a stronger to a weaker Banach space, with
differential $T_y$ extending to a bounded, contractive map on the weaker space \cite[Lemma 2.5, p. 53]{Z1} (\cite[Lemma 3,p. 132]{Z2}).
As the details are sufficiently different in the present $H^1$ setting, particularly
for higher regularity, we carry out the argument here in full.

\medskip
First, we note that the function $\Sigma_\rmc$ is Lipschitz in the $Z_{\gamma,\beta}(\bH)$-norm.
\begin{lemma}\label{r3.6}
Assume Hypotheses (H1) and (H2). Then, the function $\Sigma_\rmc$ satisfies
\begin{equation}\label{3.6-1}
\|\Sigma_\rmc(h_1)-\Sigma_\rmc(h_2)\|_{\cZ_{\gamma,\beta}}\leq c\|h_1-h_2\|\quad\mbox{for any}\quad h_1,h_2\in\overline{B}_{\bH_\rmc}(0,\eps_1(\delta_0)).
\end{equation}
\end{lemma}

\begin{proof}
Since $\Sigma_\rmc(h)=F_\rmc(\cdot,h)+\cN_{\eps_0}(\Sigma_\rmc(h))$ for any $h\in\overline{B}_{\bH_\rmc}(0,\eps_1(\delta_0))$,
from Lemma~\ref{l3.2}(iv) we obtain that
\begin{align}\label{3.6-2}
\|\Sigma_\rmc(h_1)-\Sigma_\rmc(h_2)\|_{\cZ_{\gamma,\beta}}&\leq \|F_\rmc(\cdot,h_1)-F_\rmc(\cdot,h_2)\|_{\cZ_{\gamma,\beta}}+\|\cN_\eps(\Sigma_\rmc(h_1))-\cN_\eps(\Sigma_\rmc(h_1))\|_{\cZ_{\gamma,\beta}}\nonumber\\
&\leq c\|h_1-h_2\|+c(\eps_0+\delta_0) \|\Sigma_\rmc(h_1)-\Sigma_\rmc(h_2)\|_{\cZ_{\gamma,\beta}}
\end{align}
for any $h_1,h_2\in\overline{B}_{\bH_\rmc}(0,\eps_1(\delta_0))$. Since $c\eps_0<\frac{1}{4}$ and $c\delta_0\leq\frac{1}{4}$,
\eqref{3.6-1} follows shortly from \eqref{3.6-2}.
\end{proof}

\begin{lemma}\label{l3.7}
Assume Hypotheses (H1) and (H2) and let $0<\alpha<\gamma<\beta<\frac{1}{2}\nu(\gamma_0,E_0)$ be three positive weights satisfying the condition $2\alpha<\beta-\gamma$. Then, the function $\Sigma_\rmc$ is of class $C^1$ in the $\cZ_{2\gamma,2\beta}(\bH)$ topology.
\end{lemma}
\begin{proof} Since the function $f_\rmc$ is a bilinear map from $\RR\times\bH$ to $\bH$, from \eqref{def-F-star} we infer that
\begin{equation}\label{3.7-1}
\mbox{the function}\;h\to F_\rmc(\cdot,h):\bH_\rmc\to H^1_{-\alpha}(\RR,\bH)\;\mbox{is of class}\;C^\infty.
\end{equation}
Since the cutoff function $\rho$ is a function of class $C^\infty$ with compact support and $B$ is a bilinear map on $\bH$, from \eqref{def-Neps} we have that the function $N_{\eps_0}$ is of class $C^\infty$ on $\bH$ and $\sup_{h\in\bH}\|N_{\eps_0}^{(j)}(h)\|<\infty$ for any $j\geq 0$. Since $\cK\in\cB(\cZ_{2\gamma,2\beta}(\bH))$ by Lemma~\ref{l2.6}, from Lemma~\ref{A1} it follows that
the function $\cN_{\eps_0}$ is of class $C^1$ from $\cZ_{\gamma,\beta}(\bH)$ to $\cZ_{2\gamma,2\beta}(\bH)$. Moreover, the linear operator $\cN_{\eps_0}'(f)$ can be extended to a bounded linear operator on $\cZ_{\gamma,\beta}(\bH)$ for any $f\in\overline{B}_{H^1_{-\alpha}(\RR,\bH)}(0,\delta_0)$. Since $\Sigma_\rmc(h)=\obw(\cdot,h)\in H^1_{-\alpha}(\RR,\bH)$ and  $\|\Sigma_\rmc(h)\|_{H^1_{-\alpha}}\leq\delta_0$ for any $h\in\overline{B}_{\bH_\rmc}(0,\eps_1(\delta_0))$ by Lemma~\ref{l3.3}, from \eqref{A1-2} and \eqref{A1-3} we have that
\begin{equation}\label{3.7-2}
\|\cN_{\eps_0}'(\Sigma_\rmc(h))\|_{\cB(\cZ_{\gamma,\beta}(\bH))}=\delta_0\mathcal{O}(\sup_{h\in\bH}\|N_{\eps_0}''(h)\|)+\mathcal{O}(\sup_{h\in\bH}\|N_{\eps_0}'(h)\|)\leq c(\eps_0+\delta_0)\leq\frac{1}{2}
\end{equation}
for any $h\in\overline{B}_{\bH_\rmc}(0,\eps_1(\delta_0))$. Doubling the weights $\gamma$ and $\beta$, we note that the linear operator $\cN_{\eps_0}'(\Sigma_\rmc(h))$ can be extended to a bounded linear operator on $\cZ_{2\gamma,2\beta}(\bH)$ and
\begin{equation}\label{3.7-3}
\|\cN_{\eps_0}'(\Sigma_\rmc(h))\|_{\cB(\cZ_{2\gamma,2\beta}(\bH))}\leq\frac{1}{2}\;\mbox{for any}\;h\in\overline{B}_{\bH_\rmc}(0,\eps_1(\delta_0)).
\end{equation}
\noindent{\bf Claim 1.} $\Sigma_\rmc$ is differentiable in the $\overline{B}_{\bH_\rmc}(0,\eps_1(\delta_0))\to\cZ_{2\gamma,2\beta}(\bH)$ topology and
\begin{equation}\label{3.7-4}
\Sigma_\rmc'(h)=\Big(I-\cN_{\eps_0}'(\Sigma_\rmc(h))\Big)^{-1}\pa_hF_\rmc(\cdot,h)\quad\mbox{for any}\quad h\in\overline{B}_{\bH_\rmc}(0,\eps_1(\delta_0)).
\end{equation}
First, we fix $h_0\in\overline{B}_{\bH_\rmc}(0,\eps_1(\delta_0))$. Since $\cN_{\eps_0}$ is of class $C^1$ from $\cZ_{\gamma,\beta}(\bH)$ to $\cZ_{2\gamma,2\beta}(\bH)$ by Lemma~\ref{A1}, we have that there exists $\cR_{\eps_0}:H^1_{-\alpha}(\RR,\bH)\to\cZ_{2\gamma,2\beta}(\bH)$ such that
\begin{align}\label{3.7-5}
\cR_{\eps_0}(f)\to 0\;\mbox{in}\;\cZ_{2\gamma,2\beta}(\bH)\;\mbox{as}\;&H^1_{-\alpha}(\RR,\bH)\ni f\to\Sigma_\rmc(h_0)\;\mbox{in}\;\cZ_{\gamma,\beta}(\bH)\nonumber\\
\cN_{\eps_0}(f)=\cN_{\eps_0}(\Sigma_\rmc(h_0))+\cN_{\eps_0}'(\Sigma_\rmc(h_0))&\big(f-\Sigma_\rmc(h_0)\big)+\|f-\Sigma_\rmc(h_0)\|_{\cZ_{\gamma,\beta}}\cR_{\eps_0}(f)
\end{align}
for any $f\in H^1_{-\alpha}(\RR,\bH)$. Since $\Sigma_\rmc(h)\in H^1_{-\alpha}(\RR,\bH)$ is a solution of equation $\bw=\cT_{\eps_0}(h,\bw)$ for any $h\in\overline{B}_{\bH_\rmc}(0,\eps_1(\delta_0))$ we have that
\begin{align}\label{3.7-6}
&\qquad\qquad\qquad\Sigma_\rmc(h)-\Sigma_\rmc(h_0)=F_\rmc(\cdot,h)+\cN_{\eps_0}(\Sigma_\rmc(h))-F_\rmc(\cdot,h_0)-\cN_{\eps_0}(\Sigma_\rmc(h_0))\nonumber\\
&=F_\rmc(\cdot,h)-F_\rmc(\cdot,h_0)+\cN_{\eps_0}(\Sigma_\rmc(h_0))\big(\Sigma_\rmc(h)-\Sigma_\rmc(h_0)\big)+\|\Sigma_\rmc(h)-\Sigma_\rmc(h_0)\|_{\cZ_{\gamma,\beta}}\cR_{\eps_0}(\Sigma_\rmc(h))
\end{align}
for any $h\in\overline{B}_{\bH_\rmc}(0,\eps_1(\delta_0))$. From \eqref{3.7-2} and \eqref{3.7-3} we infer that $I-\cN_{\eps_0}'(\Sigma_\rmc(h))$ is invertible on $\cZ_{2\gamma,2\beta}(\bH)$ and $\Big(I-\cN_{\eps_0}'(\Sigma_\rmc(h))\Big)^{-1}\cZ_{\gamma,\beta}(\bH)=\cZ_{\gamma,\beta}(\bH)$ for any $h\in\overline{B}_{\bH_\rmc}(0,\eps_1(\delta_0))$. From \eqref{3.7-6} we obtain that
\begin{align}\label{3.7-7}
\Sigma_\rmc(h)-\Sigma_\rmc(h_0)&=\Big(I-\cN_{\eps_0}'(\Sigma_\rmc(h_0))\Big)^{-1}\Big(F_\rmc(\cdot,h)-F_\rmc(\cdot,h_0)\Big)\nonumber\\
&+\|\Sigma_\rmc(h)-\Sigma_\rmc(h_0)\|_{\cZ_{\gamma,\beta}}\Big(I-\cN_{\eps_0}'(\Sigma_\rmc(h_0))\Big)^{-1}\cR_{\eps_0}(\Sigma_\rmc(h))
\end{align}
for any $h\in\overline{B}_{\bH_\rmc}(0,\eps_1(\delta_0))$. From Lemma~\ref{r3.6}, \eqref{3.7-2}, \eqref{3.7-3} and \eqref{3.7-7} we obtain that
\begin{align}\label{3.7-8}
\Big\|\Sigma_\rmc(h)&-\Sigma_\rmc(h_0)-\Big(I-\cN_{\eps_0}'(\Sigma_\rmc(h_0))\Big)^{-1}\pa_hF_\rmc(\cdot,h_0)(h-h_0)\Big\|_{\cZ_{2\gamma,2\beta}}\nonumber\\
&\leq\Big\|\Big(I-\cN_{\eps_0}'(\Sigma_\rmc(h_0))\Big)^{-1}\Big\|_{\cB(\cZ_{\gamma,\beta})}\Big\|F_\rmc(\cdot,h)-F_\rmc(\cdot,h_0)-\pa_hF_\rmc(\cdot,h_0)(h-h_0)\Big\|_{\cZ_{\gamma,\beta}}\nonumber\\
&+\Big\|\Sigma_\rmc(h)-\Sigma_\rmc(h_0)\Big\|_{\cZ_{\gamma,\beta}}\Big\|\Big(I-\cN_{\eps_0}'(\Sigma_\rmc(h_0))\Big)^{-1}\Big\|_{\cB(\cZ_{2\gamma,2\beta})}\Big\|\cR_{\eps_0}(\Sigma_\rmc(h))
\Big\|_{\cZ_{2\gamma,2\beta}}\nonumber\\
&\leq2\|F_\rmc(\cdot,h)-F_\rmc(\cdot,h_0)-\pa_hF_\rmc(\cdot,h_0)(h-h_0)\|_{\cZ_{\gamma,\beta}}+2\|h-h_0\|\|\cR_{\eps_0}(\Sigma_\rmc(h))\|_{\cZ_{2\gamma,2\beta}}
\end{align}
for any $h\in\overline{B}_{\bH_\rmc}(0,\eps_1(\delta_0))$. From \eqref{3.6-1} and \eqref{3.7-5} we infer that $\lim_{h\to h_0}\|\cR_{\eps_0}(\Sigma_\rmc(h))\|_{\cZ_{2\gamma,2\beta}}=0$. From \eqref{3.7-1} and \eqref{3.7-8} we conclude that $\Sigma_\rmc$ is differentiable at $h_0$ in the $\overline{B}_{\bH_\rmc}(0,\eps_1(\delta_0))\to\cZ_{2\gamma,2\beta}(\bH)$ topology, proving Claim 1.

\noindent{\bf Claim 2.} $\Sigma_\rmc'$ is continuous from $\overline{B}_{\bH_\rmc}(0,\eps_1(\delta_0))$ to $\cB\big(\bH_\rmc,\cZ_{2\gamma,2\beta}(\bH)\big)$.

We fix again $h_0\in\overline{B}_{\bH_\rmc}(0,\eps_1(\delta_0))$. From \eqref{3.7-4} we have that
\begin{align}\label{3.7-9}
\Sigma_\rmc'(h)&-\Sigma_\rmc'(h_0)=\Big(\Big(I-\cN_{\eps_0}'(\Sigma_\rmc(h))\Big)^{-1}-\Big(I-\cN_{\eps_0}'(\Sigma_\rmc(h_0))\Big)^{-1}\Big)\pa_hF_\rmc(\cdot,h_0)\nonumber\\
&\qquad\qquad+\Big(I-\cN_{\eps_0}'(\Sigma_\rmc(h))\Big)^{-1}\Big(\pa_hF_\rmc(\cdot,h)-\pa_hF_\rmc(\cdot,h_0)\Big)\nonumber\\
&=\Big(I-\cN_{\eps_0}'(\Sigma_\rmc(h))\Big)^{-1}\Big(\cN_{\eps_0}'(\Sigma_\rmc(h))-\cN_{\eps_0}'(\Sigma_\rmc(h_0))\Big)\Big(I-\cN_{\eps_0}'(\Sigma_\rmc(h_0))\Big)^{-1}\pa_hF_\rmc(\cdot,h_0)\nonumber\\
&\qquad\qquad+\Big(I-\cN_{\eps_0}'(\Sigma_\rmc(h))\Big)^{-1}\Big(\pa_hF_\rmc(\cdot,h)-\pa_hF_\rmc(\cdot,h_0)\Big)
\end{align}
for any $h\in\overline{B}_{\bH_\rmc}(0,\eps_1(\delta_0))$. From \eqref{3.7-2}, \eqref{3.7-3} and \eqref{3.7-9} we obtain that
\begin{align}\label{3.7-10}
\|\Sigma_\rmc'(h)&-\Sigma_\rmc'(h_0)\|_{\cB(\bH_\rmc,\cZ_{2\gamma,2\beta})}\leq 2\|\pa_hF_\rmc(\cdot,h)-\pa_hF_\rmc(\cdot,h_0)\|_{\cB(\bH_\rmc,\cZ_{2\gamma,2\beta})}\nonumber\\
&+2\Big\|\Big(\cN_{\eps_0}'(\Sigma_\rmc(h))-\cN_{\eps_0}'(\Sigma_\rmc(h_0))\Big)\Big(I-\cN_{\eps_0}'(\Sigma_\rmc(h_0))\Big)^{-1}\pa_hF_\rmc(\cdot,h_0)\Big\|_{\cB(\bH_\rmc,\cZ_{2\gamma,2\beta})}\nonumber\\
&\leq 4\Big\|\cN_{\eps_0}'(\Sigma_\rmc(h))-\cN_{\eps_0}'(\Sigma_\rmc(h_0))\Big\|_{\cB(\cZ_{\gamma,\beta},\cZ_{2\gamma,2\beta})}
\Big\|\pa_hF_\rmc(\cdot,h_0)\Big\|_{\cB(\bH_\rmc,\cZ_{\gamma,\beta})}\nonumber\\
&+2\|\pa_hF_\rmc(\cdot,h)-\pa_hF_\rmc(\cdot,h_0)\|_{\cB(\bH_\rmc,\cZ_{2\gamma,2\beta})}\quad\mbox{for any}\quad h\in\overline{B}_{\bH_\rmc}(0,\eps_1(\delta_0)).
\end{align}
Since the function $\cN_{\eps_0}$ is of class $C^1$ by Lemma~\ref{A1}, from \eqref{3.6-1} it follows that
\begin{equation}\label{3.7-11}
\lim_{h\to h_0}\Big\|\cN_{\eps_0}'(\Sigma_\rmc(h))-\cN_{\eps_0}'(\Sigma_\rmc(h_0))\Big\|_{\cB(\cZ_{\gamma,\beta},\cZ_{2\gamma,2\beta})}=0.
\end{equation}
From \eqref{3.7-1}, \eqref{3.7-10} and \eqref{3.7-11} we conclude that $\Sigma_\rmc'$ is continuous at $h_0$ in the $\overline{B}_{\bH_\rmc}(0,\eps_1(\delta_0))\to\cB\big(\bH_\rmc,\cZ_{2\gamma,2\beta}(\bH)\big)$ topology, proving Claim 2 and the lemma.
\end{proof}
Next, we focus on proving the higher order smoothness of $\Sigma_\rmc$ using \eqref{3.7-4} and Lemma~\ref{A2}. Since the function $h\to F_\rmc(\cdot,h):\bH_\rmc\to H^1_{-\alpha}(\RR,\bH)$ is of class $C^\infty$ as pointed out in \eqref{3.7-1}, we need to study the smoothness properties of the operator valued function
$(I-\cN_{\eps_0}'\circ\Sigma_\rmc)^{-1}$.
\begin{lemma}\label{l3.8}
Assume Hypotheses (H1) and (H2). Then, there exist $\tgamma<\tbeta<\frac{1}{2}\nu(\gamma_0,E_0)$, such that the function $\cN_{\eps_0}'\circ\Sigma_\rmc$ and the function
\begin{equation}\label{diff-inverse}
h\to \Big(I-\cN_{\eps_0}'(\Sigma_\rmc(h))\Big)^{-1}:\overline{B}_{\bH_\rmc}(0,\eps_1(\delta_0))\to\cB\big(\cZ_{2\gamma,2\beta}(\bH),\cZ_{\tgamma,\tbeta}(\bH)\big)\;\mbox{are of class}\;C^1.
\end{equation}
Moreover, the following formula holds true:
\begin{equation}\label{partial-deriv-inverse}
\pa_{h_\ell} \Big(I-\cN_{\eps_0}'(\Sigma_\rmc(h))\Big)^{-1}=\Big(I-\cN_{\eps_0}'(\Sigma_\rmc(h))\Big)^{-1} \pa_{h_\ell}\cN_{\eps_0}'(\Sigma_\rmc(h))\Big(I-\cN_{\eps_0}'(\Sigma_\rmc(h))\Big)^{-1}.
\end{equation}
\end{lemma}
\begin{proof}
Since $N_{\eps_0}$ is a function of class $C^\infty$ on $\bH$, we have that $N_{\eps_0}':\bH\to\cB(\bH)$ is of class $C^\infty$. Moreover, since the cutoff function $\rho$ has compact support,
from \eqref{def-Neps} we obtain that all the derivatives of $N_{\eps_0}$ are bounded. Hence, the function $N_{\eps_0}$ satisfies the conditions of Lemma~\ref{A2} for $p=0$. It follows that the function $\cW_1:\cZ_{2\gamma,2\beta}(\bH)\to\cB\big(\cZ_{2\gamma,2\beta}(\bH),\cZ_{\gamma_1,\beta_1}(\bH)\big)$ defined by
\begin{equation}\label{3.8-1}
(\cW_1(f)z\big)(x)=N_{\eps_0}'(f(x))z(x)\;\mbox{for}\;f,z\in\cZ_{\gamma,\beta}(\bH),\,x\in\RR
\end{equation}
is of class $C^1$, where $\gamma_1=4\beta+2\gamma$ and $\beta_1=8\beta$. We note that $\beta_1-\gamma_1>2\beta-2\gamma>2\alpha$. Moreover, the weights $\alpha$, $\beta$ and $\gamma$ can be chosen small enough such that $\beta_1<\frac{1}{2}\nu(\Gamma_0,E_0)$. Since $\cN_{\eps_0}'(\Sigma_\rmc(h))=\cK\cW_1(\Sigma_\rmc(h))$ for any $h\in\overline{B}_{\bH_\rmc}(0,\eps_1(\delta_0))$,  $\cK\in\cB(\cZ_{\gamma_1,\beta_1}(\bH))$ by Lemma~\ref{l2.6} and Lemma~\ref{r2.11}, and $\Sigma_\rmc$ is of class $C^1$ by Lemma~\ref{l3.7}, we conclude that
\begin{equation}\label{first-result}
\cN_{\eps_0}'\circ\Sigma_\rmc\;\mbox{is of class}\; C^1\;\mbox{from}\;\overline{B}_{\bH_\rmc}(0,\eps_1(\delta_0))\;\mbox{to}\; \cB\big(\cZ_{2\gamma,2\beta}(\bH),\cZ_{\gamma_1,\beta_1}(\bH)\big).
\end{equation}
Since the function $\cN_{\eps_0}'\circ\Sigma_\rmc$ is of class $C^1$ only in the weaker $\cB\big(\cZ_{2\gamma,2\beta}(\bH),\cZ_{\gamma_1,\beta_1}(\bH)\big)$ topology, rather than the $\cB\big(\cZ_{2\gamma,2\beta}(\bH)\big)$ topology, we \textit{cannot} infer \eqref{diff-inverse} immediately. To overcome this issue, we use the fact that for any two bounded operators $S_j$, $j=1,2$, with $\|S_1\|<1$ and $\|S_2\|<1$ we have that
\begin{equation}\label{3.8-2}
(I-S_1)^{-1}-(I-S_2)^{-1}=(I-S_1)^{-1}(S_1-S_2)(I-S_2)^{-1}.
\end{equation}
Moreover, we need to adjust the weights accordingly. Therefore, we introduce $\gamma_j$ and $\beta_j$, $j=2,3$, by the formula $\gamma_j=2\beta_{j-1}+\gamma_{j-1}$ and $\beta_j=4\beta_{j-1}$, for $j=2,3$ and set $\gamma_0:=2\gamma$ and $\beta_0=2\beta$. We note that $\beta_3-\gamma_3>\beta_2-\gamma_2>\beta_1-\gamma_1>\beta-\gamma>2\alpha$. The original weights $\alpha$, $\beta$ and $\gamma$ can be chosen small enough such that $\beta_3<\frac{1}{2}\nu(\Gamma_0,E_0)$.

\noindent{\bf Claim 1.} The function $h\to \Big(I-\cN_{\eps_0}'(\Sigma_\rmc(h))\Big)^{-1}:\overline{B}_{\bH_\rmc}(0,\eps_1(\delta_0))\to\cB\big(\cZ_{\gamma_j,\beta_j}(\bH),\cZ_{\gamma_{j+1},\beta_{j+1}}(\bH)\big)$ is continuous for any $j=0,1,2$. Arguing similar to \eqref{3.7-2}, we have that the constants $\eps_0>0$ and $\delta_0>0$ can be chosen small enough such that $\cN_{\eps_0}'(\Sigma(h))$   can be extended to a bounded linear operator on $\cZ_{\gamma_j,\beta_j}(\bH)$ and
\begin{equation}\label{3.8-3}
\|\cN_{\eps_0}'(\Sigma_\rmc(h))\|_{\cB(\cZ_{\gamma_j,\beta_j}(\bH))}\leq\frac{1}{2}\;\mbox{for any}\;h\in\overline{B}_{\bH_\rmc}(0,\eps_1(\delta_0)),\,j=0,1,2,3.
\end{equation}
From \eqref{3.8-2} one can readily check that
\begin{align}\label{3.8-5}
\Big(I&-\cN_{\eps_0}'(\Sigma_\rmc(h))\Big)^{-1}-\Big(I-\cN_{\eps_0}'(\Sigma_\rmc(h_0))\Big)^{-1}\nonumber\\&=\Big(I-\cN_{\eps_0}'(\Sigma_\rmc(h))\Big)^{-1}
\Big( \cN_{\eps_0}'(\Sigma_\rmc(h))-\cN_{\eps_0}'(\Sigma_\rmc(h_0))\Big)
\Big(I-\cN_{\eps_0}'(\Sigma_\rmc(h_0))\Big)^{-1}
\end{align}
for any $h\in\overline{B}_{\bH_\rmc}(0,\eps_1(\delta_0))$. From \eqref{3.8-3} and \eqref{3.8-5} we obtain that
\begin{align}\label{3.8-4}
\Big\|\Big(I-\cN_{\eps_0}'(\Sigma_\rmc(h))\Big)^{-1}&-\Big(I-\cN_{\eps_0}'(\Sigma_\rmc(h_0))\Big)^{-1}\Big\|_{\cB\big(\cZ_{\gamma_j,\beta_j},\cZ_{\gamma_{j+1},\beta_{j+1}}\big)}\leq \nonumber\\&\leq
4\Big\|\cN_{\eps_0}'(\Sigma_\rmc(h))-\cN_{\eps_0}'(\Sigma_\rmc(h_0))\Big\|_{\cB\big(\cZ_{\gamma_j,\beta_j},\cZ_{\gamma_{j+1},\beta_{j+1}}\big)}
\end{align}
for any $h,h_0\in\overline{B}_{\bH_\rmc}(0,\eps_1(\delta_0))$. Since the function $\cN_{\eps_0}'\circ\Sigma_\rmc$ is of class $C^1$, as shown above, Claim 1 follows shortly from \eqref{3.8-4}.

\noindent{\bf Claim 2.} The function $h\to \Big(I-\cN_{\eps_0}'(\Sigma_\rmc(h))\Big)^{-1}:\overline{B}_{\bH_\rmc}(0,\eps_1(\delta_0))\to\cB\big(\cZ_{\gamma_j,\beta_j}(\bH),\cZ_{\gamma_{j+2},\beta_{j+2}}(\bH)\big)$ has partial derivatives for any $j=0,1$. Moreover, \eqref{partial-deriv-inverse} holds true. Fix $j\in\{0,1\}$.  Let $\{\mathrm{e}_\ell\}_{\ell}$ be a basis in $\bH_\rmc$ and $s\in\RR$ small enough.
To prove Claim 2 we set $h=h_0+s\mathrm{e}_\ell$ in \eqref{3.8-5} for $s\in\RR$ small enough and pass to the limit as $s\to0$. Indeed, from \eqref{first-result} we obtain that
\begin{equation}\label{3.8-6}
\frac{1}{s}\big( \cN_{\eps_0}'(\Sigma_\rmc(h_0+s\mathrm{e}_\ell))-\cN_{\eps_0}'(\Sigma_\rmc(h_0))\big)\to\pa_{h_\ell}(\cN_{\eps_0}'\circ\Sigma_\rmc)(h_0)\;\mbox{as}\;s\to 0\;\mbox{in}\;\cB(\cZ_{\gamma_j,\beta_j}(\bH),\cZ_{\gamma_{j+1},\beta_{j+1}}(\bH)).
\end{equation}
Moreover, from Claim 1 we infer that
\begin{equation}\label{3.8-7}
\Big(I-\cN_{\eps_0}'(\Sigma_\rmc(h_0+s\mathrm{e}_\ell))\Big)^{-1}\to\Big(I-\cN_{\eps_0}'(\Sigma_\rmc(h_0))\Big)^{-1}\;\mbox{as}\;s\to 0\;\mbox{in}\;\cB(\cZ_{\gamma_{j+1},\beta_{j+1}}(\bH),\cZ_{\gamma_{j+2},\beta_{j+2}}(\bH)).
\end{equation}
In addition, from \eqref{3.8-3} it follows that $\Big(I-\cN_{\eps_0}'(\Sigma_\rmc(h_0))\Big)^{-1}\in \cB(\cZ_{\gamma_j,\beta_j}(\bH))$. Summarizing, it is now clear that Claim 2 follows shortly from \eqref{3.8-6} and \eqref{3.8-7}. To finish the proof of lemma, we prove that the partial derivatives of $(I-\cN_{\eps_0}'\circ\Sigma_\rmc)^{-1}$ are continuous. We set $\tgamma:=\gamma_3$ and $\tbeta:=\beta_3$. Passing to the limit for $h\to h_0$ in \eqref{partial-deriv-inverse}, the lemma follows from \eqref{first-result} and Claim 1.
\end{proof}
\begin{lemma}\label{l3.9}
Assume Hypotheses (H1) and (H2). Then, for any integer $k\geq2$ there exist $\ogamma<\obeta<\frac{1}{2}\nu(\gamma_0,E_0)$ such that the function $\Sigma_\rmc$ is of class $C^k$ from $\overline{B}_{\bH_\rmc}(0,\eps_1(\delta_0))$ to $\cZ_{\ogamma,\obeta}(\bH)$.
\end{lemma}
\begin{proof}
From Lemma~\ref{l3.8}, \eqref{3.7-1} and \eqref{3.7-4} we can immediately check that the function $\Sigma_\rmc$ is of class $C^2$ from $\overline{B}_{\bH_\rmc}(0,\eps_1(\delta_0))$ to $\cZ_{\tgamma,\tbeta}(\bH)$ for some weights $\tbeta>\tgamma>0$, satisfying the conditions $\tgamma>\alpha$ and $\tbeta<\frac{1}{2}\nu(\gamma_0,E_0)$. Next, we assume that $\Sigma_\rmc$ is of class $C^j$ from $\overline{B}_{\bH_\rmc}(0,\eps_1(\delta_0))$ to $\cZ_{\tgamma,\tbeta}(\bH)$ for some $j\geq 2$ and some weights $\tgamma<\tbeta<\frac{1}{2}\nu(\gamma_0,E_0)$. To prove the lemma, we show that from this assumption we can infer that the function $\Sigma_\rmc$ is of class $C^{j+1}$ for $j\leq k-1$.

Since the function $N_{\eps_0}$ defined in \eqref{def-Neps} is of class $C^\infty$ it follows that for any $2\leq m\leq j$ the function $L_m:\bH^m\to\cB(\bH)$ defined by
\begin{equation}\label{3.9-1}
\big(L_m(h_1,h_2,\dots,h_m)\big)g=N_{\eps_0}^{(m)}(h_1)\big(h_2,\dots,h_m,g)\;\mbox{for}\;g,h_1,h_2,\dots,h_m\in\bH
\end{equation}
is of class $C^\infty$. Moreover, since the cutoff function $\rho$ is of class $C^\infty$ with compact support, we have that $\sup_{h\in\bH}\|N_{\eps_0}^{(\ell)}(h)\|<\infty$ for any $\ell\geq 0$, which implies there exists a positive integer $p$ and $c_\ell>0$ such that
\begin{equation}\label{3.9-2}
\|L_m^{(\ell)}(h_1,h_2,\dots,h_m)\|\leq c_\ell\|(h_1,h_2,\dots,h_m)\|_{\bH^m}^p\;\mbox{for any}\;h_1,h_2,\dots,h_m\in\bH,\,\ell=1,2.
\end{equation}
From Lemma~\ref{A2} we obtain that there exist two weights $0<\ogamma<\obeta$ such that the function $\cW_m:\cZ_{\tgamma,\tbeta}(\bH^m)\to\cB\big(\cZ_{\tgamma,\tbeta}(\bH),\cZ_{\ogamma,\obeta}(\bH)\big)$ defined by
\begin{equation}\label{3.9-3}
\big(\cW_m(f)z\big)(x)=L_m(f(x))z(x),\;\mbox{for}\; f\in\cZ_{\tgamma,\tbeta}(\bH^m),\,z\in\cZ_{\tgamma,\tbeta}(\bH),\,x\in\RR,
\end{equation}
is of class $C^1$. The original weights $\alpha$, $\beta$ and $\gamma$ can be chosen small enough such that $\obeta<\frac{1}{2}\nu(\gamma_0,E_0)$. We recall that $\cN_{\eps_0}'(\Sigma_\rmc(h))=\cK\cW_1(\Sigma_\rmc(h))$ for any $h\in\overline{B}_{\bH_\rmc}(0,\eps_1(\delta_0))$, where $\cW_1$ is defined in \eqref{3.8-1}. Therefore, the partial derivatives of $\cN_{\eps_0}\circ\Sigma_\rmc$ can be expressed in terms of $\cW_m$, $2\leq m\leq j$ and the partial derivatives of $\Sigma_\rmc$.
Since $\cK\in\cB(\cZ_{\ogamma,\obeta}(\bH))$ by Lemma~\ref{l2.6} and Lemma~\ref{r2.11}, we infer that all partial derivatives of order less than $j$ of the function $\cN_{\eps_0}'\circ\Sigma_\rmc$ are of class $C^1$ from $\overline{B}_{\bH_\rmc}(0,\eps_1(\delta_0))$ to $\cB\big(\cZ_{\tgamma,\tbeta}(\bH),\cZ_{\ogamma,\obeta}(\bH)\big)$. Hence, the function $\cN_{\eps_0}'\circ\Sigma_\rmc$ is of class $C^{j+1}$ from $\overline{B}_{\bH_\rmc}(0,\eps_1(\delta_0))$ to $\cB\big(\cZ_{\tgamma,\tbeta}(\bH),\cZ_{\ogamma,\obeta}(\bH)\big)$. By modifying the choice of the original weights $\alpha$, $\beta$ and $\gamma$ once again, arguing similar to Lemma~\ref{l3.8} and taking partial derivative with respect to $h\in\bH_\rmc$ in \eqref{partial-deriv-inverse}, we have that the weights $\ogamma$ and $\obeta$ can be chosen such that the function
\begin{equation}\label{3.9-4}
h\to \Big(I-\cN_{\eps_0}'(\Sigma_\rmc(h))\Big)^{-1}:\overline{B}_{\bH_\rmc}(0,\eps_1(\delta_0))\to\cB\big(\cZ_{\tgamma,\tbeta}(\bH),\cZ_{\ogamma,\obeta}(\bH)\big)\;\mbox{is of class}\;C^{j+1}.
\end{equation}
From \eqref{3.7-4} and \eqref{3.9-4} we conclude that $\Sigma_\rmc$ is of class $C^{j+1}$ from $\overline{B}_{\bH_\rmc}(0,\eps_1(\delta_0))$ to $\cZ_{\ogamma,\obeta}(\bH)$. The lemma follows by repeating the argument above a \textit{finite} number of times.
\end{proof}
Using Lemma~\ref{l3.9} we can finally conclude that the center manifold $\cM_\rmc$ is of class $C^k$. Indeed, since the linear operator $\mathrm{Trace}_0:\cZ_{\ogamma,\obeta}(\bH)\to\bH$ defined by $\mathrm{Trace}_0(f)=f(0)$ is bounded, from \eqref{center-manifold} and \eqref{def-Sigma-c} we conclude that $\cM_\rmc$ is of class $C^k$.
	\begin{remark}\label{finitely-many-times}
We note that the iteration used to prove higher order regularity of the center manifold (Lemma~\ref{l3.9}) can be 
applied only finitely many times. The main reason is that at each step of the iteration we need to readjust the original weights $\alpha$, $\beta$ and $\gamma$ such that $\obeta<\nu(\Gamma_0,E_0)$. Since the weights increase by a factor of at least $2$ every time we make the adjustment, after $j$ steps the new weights are bigger than $\mathcal{O}(2^j)$, thus becoming greater than $\nu(\Gamma_0,E_0)$ after finitely many steps.
It is for this reason that our argument yields $C^k$ regularity for arbitrary but fixed $k$ rather than $C^\infty$.
\end{remark}

\subsection{Invariance of the center manifold}\label{CM-invariance}
To finish the proof of Theorem~\ref{t1.1} we need to prove that the center manifold $\cM_\rmc$ is invariant under the flow of equation \eqref{center-nonlinear} and to prove that it is tangent to the center subspace $\bH_\rmc$ at the equilibrium $\obu$.

\begin{lemma}\label{l3.10}
Assume Hypotheses (H1) and (H2). Then, the manifold $\cM_\rmc$ is invariant under the flow of equation \eqref{center-nonlinear}, in the sense that for each element of $\cM_\rmc$ there exists a solution of \eqref{center-nonlinear} that stays in $\cM_\rmc$ in finite time. Moreover, the manifold
$\cM_\rmc$ is tangent to the center subspace $\bH_\rmc$ at $\obu$.
\end{lemma}
\begin{proof}
Fix $y\in\cM_\rmc$. Then, from \eqref{center-manifold} it follows that there exists $w_0\in\overline{B}_{\bH_\rmc}(0,\eps_1(\delta_0))$ such that $y=\obw(0,w_0)$. From Lemma~\ref{l3.4}(ii) we know that $\obw(\cdot,w_0)$ is a solution of \eqref{center-nonlinear} on $(-\eta_0,\eta_0)$. Therefore, to prove the lemma it is enough to prove that $\obw(x,w_0)\in\cM_\rmc$ for $x$ in a neighborhood of $0$. Our strategy is to show that for any $x_0$ small enough there exists $\widetilde{w}_0(x_0)\in\overline{B}_{\bH_\rmc}(0,\eps_1(\delta_0))$ such that $\obw(\cdot+x_0,w_0)=\obw(\cdot,\widetilde{w}_0(x_0))$. First, we note that
\begin{equation}\label{3.10-1}
\big(\cF\cK_0g(\cdot+x_0)\big)(\omega)=\big(\cF(\cK_0g)(\cdot+x_0)\big)(\omega)=e^{2\pi\rmi\omega x_0}(2\pi\rmi\omega\Gamma_0-E_0)^{-1}\widehat{g}(\omega)
\end{equation}
for any $\omega\in\RR$, $g\in \cS(\RR,\tbV)$, the Schwartz class of $\tbV$-valued functions on $\RR$. It follows that
\begin{equation}\label{3.10-2}
(\cK_0g)(\cdot+x_0)=\cK_0g(\cdot+x_0)\quad\mbox{for any}\quad g\in L^2_{-\alpha}(\RR,\tbV).
\end{equation}
Similarly, one can readily check that
\begin{equation}\label{3.10-3}
(\cV f)(\cdot+x_0)=\cV(f(\cdot+x_0))+(\cV f)(x_0)\quad\mbox{for any}\quad f\in L^2_{-\alpha}(\RR,\bH).
\end{equation}
Since $\cK f=\cV\Gamma_3f+\Gamma_4\cK_0P_{\tbV}f$ for any $f\in L^2_{-\alpha}(\RR,\bH)$ from \eqref{3.10-2} and \eqref{3.10-3} we obtain that
\begin{equation}\label{3.10-4}
(\cK f)(\cdot+x_0)=\cK(f(\cdot+x_0))+\Gamma_3(\cV f)(x_0)\quad\mbox{for any}\quad f\in L^2_{-\alpha}(\RR,\bH).
\end{equation}
Since $\obw(\cdot,w_0)$ is the unique solution of equation $\bw=\cT_{\eps_0}(w_0,\bw)$, from \eqref{3.10-4} we have that
\begin{align}\label{3.10-5}
\obw(x+x_0,w_0)&=F_\rmc(x+x_0,w_0)+\big(\cK N_{\eps_0}(\obw(\cdot,w_0))\big)(x+x_0)\nonumber\\
&=F_\rmc(x+x_0,w_0)+\big(\cK N_{\eps_0}(\obw(\cdot+x_0,w_0))\big)(x)+\Gamma_3(\cV \obf)(x_0)
\end{align}
for any $x\in\RR$, where $\obf=\rho\Big(\frac{\|\obw(\cdot,w_0)\|^2}{\eps_0^2}\Big)B(\obw(\cdot,w_0),\obw(\cdot,w_0))$. Since $w_0\in\bH_\rmc=\bV^\perp\oplus\bV_\rmc$ from \eqref{def-Vc} we have that $P_{\tbV}w_0=-E_0^{-1}P_{\tbV}EP_{\bV_1}w_0$. From \eqref{def-f-star} we obtain that
\begin{equation}\label{3.10-6}
f_\rmc(x,w_0)=x E_1(I-E_0^{-1}P_{\tbV}E)P_{\bV_1}w_0=x E_1(P_{\bV_1}w_0+P_{\tbV}w_0)=x E_1P_{\bV}w_0\;\mbox{for any}\;x\in\RR.
\end{equation}
Moreover, from \eqref{def-F-star} we have that
\begin{equation}\label{3.10-7}
F_\rmc(x+x_0,w_0)+\Gamma_3(\cV \obf)(x_0)=w_0+x_0E_1P_{\bV}w_0+\Gamma_3(\cV \obf)(x_0)+x E_1P_{\bV}w_0
\end{equation}
for any $x\in\RR$. Since $\im E_1\subseteq\ker A_{11}$, $\im\Gamma_3\subseteq\ker A_{11}$ and $\ker A_{11}\subseteq\bV^\perp\subseteq\bH_\rmc$, we infer that
\begin{equation}\label{3.10-8}
\widetilde{w}_0(x_0):=w_0+x_0E_1P_{\bV}w_0+\Gamma_3(\cV \obf)(x_0)\in\bH_\rmc\quad\mbox{and}\quad P_{\bV}\widetilde{w}_0(x_0)=P_{\bV}w_0.
\end{equation}
From \eqref{3.10-6}, \eqref{3.10-7} and \eqref{3.10-8} we conclude that
\begin{equation}\label{3.10-9}
F_\rmc(x+x_0,w_0)+\Gamma_3(\cV \obf)(x_0)=\widetilde{w}_0(x_0)+x E_1P_{\bV}\widetilde{w}_0(x_0)=F_\rmc(x,\widetilde{w}_0(x_0))
\end{equation}
for any $x\in\RR$. From \eqref{3.10-5} and \eqref{3.10-9} we infer that $\obw(\cdot+x_0,w_0)$ is a solution of equation $\bw=\cT_{\eps_0}(\widetilde{w}_0(x_0),\bw)$. Since
$\lim_{x_0\to 0}\widetilde{w}_0(x_0)=w_0\in B_{\bH_\rmc}(0,\eps_1(\delta_0))$ it follows that there exists $\nu_0>0$ such that $\widetilde{w}_0(x_0)\in\overline{B}_{\bH_\rmc}(0,\eps_1(\delta_0))$ for any $x_0\in (-\nu_0,\nu_0)$. From Lemma~\ref{l3.3} we obtain that
\begin{equation}\label{3.10-10}
\obw(\cdot+x_0,w_0)=\obw(\cdot,\widetilde{w}_0(x_0))\;\mbox{for any}\;x_0\in (-\nu_0,\nu_0),
\end{equation}
which implies that $\obw(x_0,w_0)=\obw(0,\widetilde{w}_0(x_0))\in\cM_\rmc$ for any $x_0\in (-\nu_0,\nu_0)$, proving that the center manifold $\cM_\rmc$ is locally invariant under the flow of equation \eqref{center-nonlinear}.

To prove that the manifold $\cM_\rmc$ is tangent to the center subspace $\bH_\rmc$ at $\obu$, it is enough to show that $\cJ_\rmc'(0)=0$.
By uniqueness of fixed point solutions, we immediately conclude that $\obw(\cdot,0)=0$. Moreover, since $N_{\eps_0}'(0)=0$ and the function
$\bH_\rmc\ni w_0\to f_\rmc(\cdot,P_{\bV_1}w_0)\in H^1_{-\alpha}(\RR,\bH)$ is linear, from \eqref{def-f-star}
we infer that
\begin{equation}\label{3.10-11}
\big(\pa_{w_0}\obw(\cdot,0)\big)(w_0)=w_0\rmone+f_\rmc(\cdot,P_{\bV_1}w_0)
\end{equation}
for any $w_0\in\bH_\rmc$. Since $\im E_1\subseteq\bV^\perp$ and $\tbV\subset\bV$, it follows that
\begin{equation}\label{3.10-12}
\big(\pa_{w_0}\obw(0,0)\big)(w_0)=w_0+f_\rmc(0,P_{\bV_1}w_0)=w_0\quad\mbox{for any}\quad w_0\in\bH_\rmc.
\end{equation}
From \eqref{3.5-1} and \eqref{3.10-12} we conclude that $\big(\cJ_\rmc'(0)\big)(w_0)=0$ for any $w_0\in\bH_\rmc$, proving the lemma.
\end{proof}

\begin{proof}[Proof of Theorem 1.1]
Summarizing the results of this section, Theorem~\ref{t1.1} follows shortly from Lemma~\ref{r3.5},
Lemma~\ref{l3.9}, and Lemma~\ref{l3.10}. To finish the proof of the theorem we need to show that the center manifold $\cM_\rmc$ contains
the trace at $0$ of any bounded, smooth solution $\bu_0$ of equation \eqref{nonlinear} that stays sufficiently close to the equilibrium $\obu$. Indeed, in this case one can readily check that $\bw_0=\bu_0-\obu$ is a bounded, smooth solution of equation \eqref{center-nonlinear} that is small enough. It follows that $N_{\eps_0}(\bw_0(x))=B(\bw_0(x),\bw_0(x))$ for any $x\in\RR$.  From Lemma~\ref{l2.11}(i) we obtain that $\bw_0=P_\rmc\bw_0(0)+f_\rmc(\cdot,P_{\bV_1}\bw_0(0))+\cK N_{\eps_0}(\bw_0)$, that is $\bw_0$ is a solution of equation $\bw=\cT_{\eps_0}(P_\rmc\bw_0(0),\bw)$. From Lemma~\ref{l3.3} we infer that $\bw_0=\obw(\cdot,(P_\rmc\bw_0(0))$. Since $P_\rmc\bw_0(0)\in\bH_\rmc$ is small enough, we conclude that $\bw_0(0)\in\cM_\rmc$, proving the theorem.
\end{proof}

\section{Approximation of the center manifold}\label{s:approx}
Similarly as in the usual (nonsingular $A$) case,
the center manifold may be approximated to arbitrary order by formal Taylor expansion.

\subsection{Canonical form}\label{s:canon}
By the invertible change of coordinates
\ba\label{coord}
w_\rmc=\bp u_1 - (\Gamma_1 + E_1E_0^{-1} \Gamma_0)\tilde v \\
E_1(\Id - E_0^{-1}P_{\tilde {\bV}}E)v_1\\
\tilde u + \tilde A_{11}^{-1}\tilde A_{12}(\tilde v + E_0^{-1} P_{\tilde{\bV}}Ev_1)  \ep,
\quad
w_\rmh&=\tilde v+ E_0^{-1} P_{\tilde{\bV}}Ev_1,
\ea
$w_\rmc$ and $w_\rmh$ parametrizing center and hyperbolic subspaces,
we reduce \eqref{nonlinear-sys-perturbed} to {\it canonical form}:
\begin{equation}\label{canon}
\left\{\begin{array}{ll} w_\rmc'=Jw_\rmc + g_\rmc,\\
\Gamma_0 w_\rmh'= E_0w_\rmh + g_\rmh \end{array}\right.,
\end{equation}
where $J=\begin{bmatrix} 0 & I_r & 0 \\ 0&0 & 0\\ 0&0&0 \end{bmatrix}$ is a nilpotent block-Jordan form, $r=\dim\ker A_{11}$, and
\ba\label{g}
g_\rmc= \Big(\big(E_1E_0^{-1} P_{\tilde {\bV}}f + (T_{12}^*)^{-1} P_{\bV_1}f\big)^T, 0,0\Big)^T,
 \quad
 g_\rmh= P_{\tilde{\bV}}f.
\ea
Here, we have used the fourth equation of \eqref{nonlinear-sys-perturbed} to express
$\tilde v= E_0^{-1}(\Gamma_0 \tilde v' -P_{\tilde {\bV}}Ev_1- P_{\tilde {\bV}}f)$,
then substituted into the first equation to obtain
$$
\big( u_1 - (\Gamma_1 + E_1E_0^{-1} \Gamma_0)\tilde v\big)'=
E_1(\Id - E_0^{-1}P_{\tilde {\bV}}E)v_1 + \big( E_1E_0^{-1} P_{\tilde {\bV}}f + (T_{12}^*)^{-1} P_{\bV_1}f\big).
$$
The key point in showing bounded invertibility is to observe that the coefficient
$E_1(\Id - E_0^{-1}P_{\tilde {\bV}}E)$ of $v_1$ in the second component of $w_\rmc$
may be expressed as a bounded, boundedly invertible operator $(T_{12}^*)^{-1}$ applied to
$$
\big( P_{\bV_1}E_{|\bV_1}- P_{\bV_1}E_{|\tilde {\bV}}E_0^{-1}P_{\tilde {\bV}}E_{|\bV_1}\big),
$$
which, writing $E=\begin{bmatrix} E_{11}&E_{12}\\ E_{12}^*& E_{22}\end{bmatrix}:\bV^\perp\oplus\bV\to\bV^\perp\oplus\bV$ in block form, may be recognized as
$ E_{11}- E_{12}E_{22}^{-1}E_{12}^*$, which
is symmetric negative definite as a minor of the symmetric negative definite operator
\be\label{neg}
\begin{bmatrix} E_{11}- E_{12}E_{22}^{-1}E_{12}^* &0 \\ 0 & E_{22}\end{bmatrix}
=
\begin{bmatrix}\Id &- E_{12}E_{22}^{-1} \\ 0 & \Id\end{bmatrix}
\begin{bmatrix} E_{11}&E_{12}\\ E_{12}^*& E_{22}\end{bmatrix}
\begin{bmatrix} \Id & 0\\ - E_{22}^{-1}E_{12}^* & \Id \end{bmatrix}.
\ee
\br\label{fibers}
We record for later use that the tangent subspace at $(u,v)=(\zeta, 0)$ to the equilibrium manifold $\cE=\{(u,v)\in\bV^\perp\oplus\bV: Q(u,v)=0\}$
is given in the new coordinates by
\be\label{newtan}
w_\rmc=(\zeta_1, 0, \tilde \zeta), \quad w_\rmh=0,
\ee
as can also be seen directly by computing the subspace of equilibria of \eqref{canon} with $g=(g_\rmc,g_\rmh)=0$.
\er


\subsection{Taylor expansion}\label{s:taylor}
From the canonical form \eqref{canon}, the computation of the formal Taylor expansion for the center manifold goes
exactly as in the usual (nonsingular $A$) case.
Taking $f=B(\bu-\obu,\bu-\obu)$ in \eqref{canon}, we recover the original nonlinear system \eqref{Relax-EQ}, with $g$
comprising quadratic and higher order terms in $(w_\rmc,w_\rmh)$.
Expressing $w_\rmh=\Xi(w_\rmc)$, substituting into \eqref{canon}(ii), and applying \eqref{canon}(i),
we obtain the defining relation
\be\label{define}
\Gamma_0 \Xi'(w_\rmc)(Jw_\rmc + \tilde{g}_\rmc)= E_0 \Xi(w_\rmc) + \tilde{g}_\rmh.
\ee
Here $\tilde{g}_\rmc$ and $\tilde{g}_\rmh$ are obtained by solving for $\bu=(u_1,\tu,v_1,\tv)$ in terms of $(w_\rmc,w_\rmh)$ in \eqref{coord} and plugging in $f=B(\bu-\obu,\bu-\obu)$ in \eqref{g}.
Inverting $E_0$ we obtain that $ \Xi(w_\rmc)= E_0^{-1}\big(\Gamma_0 \Xi'(w_\rmc)(Jw_\rmc + g_\rmc)- g_\rmh\big).  $
This gives $\Xi(w_\rmc) - E_0^{-1}\Gamma_0 \Xi'(w_\rmc)Jw_\rmc= E_0^{-1} g_\rmc((w_\rmc,0))$ modulo higher order terms in $w_\rmc$, from which me may successively solve for the coefficients of the Taylor series of $\Xi$.
For example, in the simplest case $J=0$, this becomes just
$\Xi(w_\rmc) = E_0^{-1} g_\rmc((w_\rmc,0))$ plus higher order terms.
We omit the (standard; see, e.g. \cite{HI}) details.

\subsection{Center flow}\label{s:cflow}
The {\it center flow}, given by the reduced equation
\be\label{cflow}
w_\rmc'=Jw_\rmc + g_\rmc(w_c, \Xi(w_\rmc)),
\ee
may be approximated to any desired order $k$ in powers of $\|w_\rmc\|$ by Taylor expansion of $\Xi$ to order $k-1$.
In practice (as will be the case here), it is often sufficient to approximate the flow only to order $k=2$ in
order to perform a normal-form analysis well-describing the flow, in which case it is not necessary to compute $\Xi$ at all,
being that the $k=1$ order approximation of $\Xi$ is just $0$.

\subsection{Relaxation structure}\label{s:cons}
The discussion of approximation (and indeed all of our analysis)
up to now has been completely general, applying equally to any system of form \eqref{canon},
not necessarily originating from a system of form \eqref{Relax-EQ}.
We now make two substantial simplifications based on the special structure of \eqref{Relax-EQ}.  The first is to note that,
encoding the conservative structure of the system, coordinates $w_{\rmc,2}$ and $w_{\rmc,3}$ are constants of the flow,
hence may be considered as parameters. This reduces the center flow to an equation on an {\it $r$-dimensional fiber
indexed by $w_{\rmc,1}$}, $r=\dim\ker A_{11}$, a considerable simplification.  For related observations, see the treatment of existence
of small amplitude viscous and relaxation shock profiles in \cite{MaP,MasZ2}.

The second, using \eqref{newtan}, Remark \ref{fibers}, is that by shifting the base state of our expansion along the equilibrium manifold,
we are able to arrange without loss of generality $w_{\rmc,3}\equiv 0$, so that we obtain, ultimately,
a family of {\it $r$-dimensional fibers
\be\label{fibeq}
w_{\rmc,1}'= \zeta+ \phi(w_{\rmc,1},\zeta),
\quad
\phi(w_{\rmc,1},\zeta):=\tilde{g}_{\rmc,1}\big((w_{\rmc,1},\zeta, 0), \Xi(w_{\rmc,1},\zeta, 0)\big)=\mathcal{O}(\|w_{\rmc,1}\|^2,\|\zeta\|^2)
\ee
indexed by the $r$-dimensional parameter $\zeta:=w_{\rmc,2}$.}
In the simplest nontrivial case $r=1$ (treated in Section \ref{s:gnl}),
this amounts to a {\it one-parameter family of scalar equations.}
Taylor expanding \eqref{fibeq}, we have that $\tilde{g}_{\rmc,1}\big(w_{\rmc,1},0,0),0\big)=w_{\rmc,1}^T \chi w_{\rmc,1}+\mathcal{O}(\|w_{\rmc,1}\|^3)$, for some $\chi\in\RR$, which shows that
the normal form (in all cases) is given by
\be\label{fibeqT}
w_{\rmc,1}'= \zeta+ w_{\rmc,1}^T\chi w_{\rmc,1}+\mathcal{O}(\|w_{\rmc,1}\|^3+\|w_{\rmc,1}\|\,\|\zeta\|+\|\zeta\|^2).
\ee

\section{Bifurcation and existence of small-amplitude shock profiles}\label{s:bif}
Using the framework of Section \ref{s:approx}, it is straightforward to
describe bifurcation from equilibrium, or {\it near-equilibrium}
steady flow, in the cases \eqref{gnl} and \eqref{ldg} discussed in the introduction:
in particular, existence of small-amplitude standing kinetic shock and boundary layer solutions.

\subsection{Bifurcation from a simple, genuinely nonlinear characteristic}\label{s:gnl}
We first treat the case \eqref{gnl}, starting with normal form \eqref{fibeqT}, by relating the constants $\zeta$
and $\chi$ to quantities occurring in the equilibrium problem \eqref{ce}, using the principle that, since equilibria of
\eqref{nonlinear}, \eqref{ce2}, and \eqref{ce} all agree, the normal forms for their respective equilibrium problems
must agree (up to constant multiple) as well.
More elaborate versions of this argument may be found, e.g., in \cite{MaP,MasZ2}.

\begin{proof}[Proof of Theorem~\ref{t1.2}, case \eqref{gnl}]
	First, note that $T_{12}v_1$ in the original coordinates of \eqref{linear-sys-perturbed2} is exactly the first component $q_1$ of $q$ in \eqref{ce}, or $v_1=T_{12}^{-1} q_1$.
Substituting this into the second component of \eqref{coord}, we find after a brief calculation that
$\zeta=w_{\rmc,2}= -\varkappa^{-1} q_1$, where
\be\label{varkap}
\varkappa^{-1}= (T_{12}^*)^{-1}(P_{\bV_1}E(\Id -(E_{|\tilde {\bV}})^{-1}P_{\tilde{\bV}}E)) (T_{12})^{-1},
\ee
or, in the notation of \eqref{neg},
$\varkappa=  -T_{12} (E_{11}-E_{12}E_{22}^{-1}E_{21})^{-1} T_{12}^{*}>0 ,  $
where $E=\begin{bmatrix} E_{11}& E_{12}\\E_{21}&E_{22}\end{bmatrix}:\bV^\perp\oplus\bV\to\bV^\perp\oplus\bV$.
Using the block-matrix inversion formula
\be\label{blockinv}
\begin{bmatrix} E_{11}& E_{12}\\E_{21}&E_{22}\end{bmatrix}^{-1}=
\begin{bmatrix} (E_{11}-E_{12}E_{22}^{-1}E_{21})^{-1}&  -(E_{11}-E_{12}E_{22}^{-1}E_{21})^{-1}E_{12}E_{22}^{-1}\\
-(E_{22}-E_{21}E_{11}^{-1}E_{12})^{-1} E_{21}E_{11}^{-1}  & (E_{22}-E_{21}E_{11}^{-1}E_{12})^{-1} \end{bmatrix}
\ee
(verifiable by multiplication against $E$, or inversion of relation \eqref{neg}),
we find alternatively that
$$
\varkappa=  -e_1^T A_{12}E^{-1}A_{12}^* e_1,
$$
$e_1$ the first Euclidean basis element, or, using $\obr=e_1$,
$\varkappa= \obr^T D_* \obr$, with $D_*$ as in \eqref{ce2}.

Using the first component of \eqref{coord} to trade $w_{\rmc,1}$ for
$u_1$ by an invertible coordinate change preserving the order of error terms, we may thus rewrite \eqref{fibeqT} as
\be\label{tempfibeqT}
u_1'= \varkappa^{-1} (-q_1 +  \varkappa \chi u_{1}^2) +O(|u_{1}|^3 +|u_{1}||q_1|+ |q_1|^2),
\ee
where $\varkappa \chi$ is yet to be determined.
On the other hand, performing Lyapunov-Schmidt
reduction for the equilibrium problem \eqref{ce}, we obtain the normal form
$ 0= (-q_1 +  \frac12 \Lambda u_{1}^2) +O(|u_{1}|^3 +|u_{1}||q_1|+ |q_1|^2).  $
Using the fact that equilibria for \eqref{nonlinear} and \eqref{ce} agree, we find that $\varkappa \chi$ must be
equal to $\frac12 \Lambda$, yielding a final normal form consisting of the approximate Burgers flow \eqref{burgers}.
A similar computation yields the same normal form for fibers of the center manifold of the formal viscous problem \eqref{ce2};
see also the more detailed computations of \cite{MaP} yielding the same result.

For $q_1 \Lambda>0$, the scalar equation \eqref{burgers} evidently possesses equilibria $\sim \mp\sqrt{2 q_1/\Lambda}$,
connected (since the equation is scalar) by a heteroclinic profile.
Observing that $\sgn u_1'=-\sgn \Lambda$ for $u_1$ between the equilibria, so that
$\big(\lambda(u)\big)'\sim \Lambda u_1'$ has sign of $-\Lambda^2<0$,
we see further that the connection is
in the direction of decreasing characteristic $\lambda(u)$, hence a Lax-type connection
for \eqref{ce}; for further discussion, see \cite{MaP,MasZ2}.
\end{proof}

\subsubsection{Comparison to Chapman-Enskog profiles}\label{s:compce}
We perform the comparison of profiles of \eqref{nonlinear} and \eqref{ce2} in three steps, comparing their
primary, $u_1$, coordinates to a Burgers shock, then to each other, and finally comparing remaining
coordinates slaved to the fiber \eqref{burgers}.

\begin{lemma}[\cite{Li,PlZ}]\label{lburgers}
	Let $\eta\in \R^1$ be a heteroclinic connection of an approximate Burgers equation
	\be\label{approx}
	\varkappa \eta'= \frac12 \Lambda(-\eps^2 +  \eta^2) +  S(\eps,\eta),
	\quad
	S=O(|\eta|^3+|\eps|^3)\in C^{k+1}(\R^2), \quad k\geq 0,
	\ee
	and $\bar \eta:= -\eps\tanh(\Lambda \eps x/2\varkappa)$ a connection of the exact Burgers equation
	$\varkappa\bar \eta'= \frac12 \Lambda(-\eps^2 +  \bar \eta^2)$.
	Then,
	\ba\label{etacomp}
 |\eta_{\pm}-\bar{\eta}_{\pm}| &\leq C\epsilon^2,\\
	|\partial_x^k \big(\bar \eta - \bar \eta_\pm)(x)| &\sim  \eps^{k+1}e^{-\delta  \eps|x|},
	\quad x\gtrless 0, \quad \delta>0,\\
	\big|\partial_x^k \big((\eta- \eta_\pm)- (\bar \eta - \bar \eta_\pm)\big)(x)\big| &\le C \eps^{k+2}e^{-\delta  \eps|x|},
	\quad x\gtrless 0,
	\ea
	uniformly in $\eps>0$, where $\eta_\pm:=\eta(\pm \infty)$, $\bar \eta_\pm:=\bar \eta(\pm \infty)=\mp\eps$ denote
	endstates of the connections.
\end{lemma}
\begin{proof}(Following \cite{Li})
Rescaling ${\eta} \to {\eta}/\epsilon, {x} \to \Lambda \epsilon \tilde{x}/\beta$,
we obtain the ``blown-up'' equations
$$
\eta'=\frac{1}{2}(\eta^2-1)+\epsilon \tilde{S}(\eta,\epsilon)
\quad \tilde{S} \in C^{k+1}(\mathbb{R}^2)
$$
and $\bar{\eta}'=\frac{1}{2}(\bar{\eta}^2-1)$, for which estimates \eqref{approx} translate to
	\ba\label{scaleetacomp}
 |\eta_{\pm}-\bar{\eta}_{\pm}| &\leq C\epsilon,\\
	|\partial_x^k (\bar \eta - \bar \eta_\pm)(x)| &\sim  C\eps^{k}e^{-\theta  |x|},
	\quad x\gtrless 0, \quad \theta>0,\\
	|\partial_x^k \big((\eta- \eta_\pm)- (\bar \eta - \bar \eta_\pm)\big)(x)| &\le C \eps^{k+1}e^{-\theta  |x|},
	\quad x\gtrless 0.
	\ea
The estimates \eqref{scaleetacomp} follow readily from the implicit function theorem and stable manifold theorems together with
smooth dependence on parameters of solutions of ODE, giving the result.
\end{proof}

Setting $q_1=\Lambda \eps^2/2$, and either $\eta=u_{REL,1}$ or $\eta=u_{CE,1}$, we obtain approximate Burgers
equation \eqref{approx}, and thereby estimates \eqref{etacomp} relating $\eta=u_{REL,1}$, $u_{CE,1}$ to an exact
Burgers shock $\bar \eta$.

\begin{corollary}\label{ctri}
Let $\obu\in\ker Q$ be an equilibrium satisfying (H1)-(H2),
in the noncharacteristic case \eqref{gnl}, and $k$ and integer $\geq 2$.
Then, local to $\obu$ ($\bar u$), each pair of points $u_\pm$ corresponding to a standing Lax-type shock of \eqref{ce}
has a corresponding viscous shock solution $u_{CE}$ of \eqref{ce2} and
relaxation shock solution $\bu_{REL}=(u_{REL},v_{REL})$ of \eqref{nonlinear}, satisfying for all $j\leq k-2$:
	\ba\label{u1comp}
	|\partial_x^j ( u_{REL,1} - u_{REL,1}^\pm )(x)| &\sim C \eps^{j}e^{-\theta  |x|},
	\quad x\gtrless 0, \quad \theta>0,\\
	|\partial_x^j (u_{REL,1}-u_{CE,1})(x)| &\le C \eps^{j+1}e^{-\theta  |x|},
	\quad x\gtrless 0.
	\ea
\end{corollary}

\begin{proof}
	Immediate, by \eqref{scaleetacomp}, Lemma \ref{lburgers} and the triangle inequality, together with the observation that,
	since equilibria of \eqref{ce}, \eqref{ce2}, and \eqref{nonlinear} agree exactly,
	endstates $u_{REL,1}^\pm=u_{CE,1}^\pm$ agree.
\end{proof}

\begin{proof} [Proof of Corollary \ref{c1.3}]
	Noting that the $\im A_{11}$ and the $\bV$ components of $\bu_{REL}$ are the $C^1$ functions $\Psi(u_{REL,1})$, $\Phi(u_{REL,1})$ of $u_{REL,1}$ along
	the fiber \eqref{burgers}, we obtain \eqref{finalbds}(iii) immediately from \eqref{u1comp}(i).
	Denote by $\Psi_{CE}$ the map describing the dependence of $\im A_{11}$ component of $u_{CE}$ on $u_{CE,1}$
	on the corresponding fiber of \eqref{ce2}.
	Noting that $\Psi-\Psi_{CE}$ and $\Phi- v_*$ both vanish at the enstates $u_{REL,1}^\pm$, we have by smoothness
	of $\Psi$, $\Psi_{CE}$, $\Phi$, $v_*$ that
	$$
	|\Psi-\Psi_{CE}|, \, |\Psi-v_*|=\mathcal{O}(|u_{REL,1}- u_{REL,1}^+|,|u_{REL,1}- u_{REL,1}^-|),
	$$
	giving \eqref{finalbds}(i)-(ii) by \eqref{u1comp}(i)-(ii).
\end{proof}

\subsection{Bifurcation from a linearly degenerate characteristic}\label{s:ldg}
\begin{proof}[Proof of Theorem~\ref{t1.2}, case \eqref{ldg}]
In the case \eqref{ldg}, by an entirely similar argument, comparing to the normal form for the equilibrium problem \eqref{ce},
yields normal form \eqref{ldgflow}.
Here, the main point is to observe that, in the normal form for \eqref{ce}, all terms, including higher-order
error terms, include a factor
$q_1$, since in the fiber $q_1=0$ all points are equilibria.
Evidently, each $q_1$-fiber is either composed entirely of equilibria, or contains no equilibria, hence there exist no
nontrivial profiles connecting to equilibria either in forward or backward $x$.
\end{proof}

\section{Application to Boltzmann's equation}\label{s:localization}
We now specialize to the case Example~\ref{boltzeg} of Boltzmann's equation with hard sphere collision kernel.

\subsection{Existence and sharp localization in velocity of Center Manifolds}\label{s:sharploc}
Let $A$, $Q$ be as defined in Example~\ref{boltzeg} and $\bY^\sigma$ as defined in \eqref{Hs}, with $A_{ij}$, $\Gamma_0$
as in the rest of the paper.
We have the following result of \cite{MZ}.

\begin{lemma}[\cite{MZ}]\label{pcoerce}
Assume Hypotheses (H1) and (H2). For Boltzmann's equation with hard sphere kernel and any $1/2\leq \sigma<1$, the linear operator $E=Q'(M_{\bar u})_{|\bV}$, where $M_{\bar u}$ is the Maxwellian defined in \eqref{M}, and its inverse can be extended to bounded linear operators on $\bY^\sigma\cap\bV$.
\end{lemma}
\begin{lemma}\label{ccoerce}
Assume Hypotheses (H1) and (H2). For Boltzmann's equation with hard sphere kernel and $1/2\leq \sigma<1$, $\bH_\rmc\subset \bY^\sigma$.
\end{lemma}
\begin{proof}
The subspace $\bH_\rmc$ is the direct sum of the subspace of equilibria $\bV^\perp$,
equal to the tangent space to the manifold $\cE$ of Maxwellians, and the space $\bV_\rmc$ defined in \eqref{def-Vc}. The tangent subspace to $\cE$ at $\obu$ is given by polynomial multiples of $M_{\bar u}$, hence $\bV^\perp$ evidently lies in $\bY^\sigma$. Recalling that $A$ is a bounded multiplication operator in $\xi$, we have that $A$ can be extended to a bounded linear operator on $\bY^\sigma$. It follows that $\bY^\sigma$ is invariant under the orthogonal projectors associated to the orthogonal decomposition $\bH=\ker A_{11}\oplus\im A_{11}\oplus\bV_1\oplus\tbV$. Moreover, we have that $\bV_1=\im T_{12}^*=\im (A_{21})_{|\ker A_{11}}\subset\bY^\sigma$. Fix $v=(v_1,\tv)\in\bV_\rmc$.
Since $\bY^\sigma$ is invariant under $P_{\tbV}$, from \eqref{def-Vc} and Lemma~\ref{pcoerce} we infer that
$\tv=-E_0^{-1}P_{\tbV}Ev_1\in\bY^\sigma$, proving the lemma.
\end{proof}
\begin{lemma}\label{upgrade}
	Assume Hypotheses (H1) and (H2). The Fourier multiplier $\cK_0=\cF^{-1}M_{R_{\Gamma_0,E_0}}\cF$, associated to the
operator-valued function $R_{\Gamma_0,E_0}:\RR\to\mathcal{B}(\tbV)$ defined by $R_{\Gamma_0,E_0}(\omega)=(2\pi\rmi\omega\Gamma_0-E_0)^{-1}$,
	is bounded on $H^1_{-\alpha}(\R,\bY^\sigma\cap\tbV)$ for any $1/2\leq \sigma <1$ and $\alpha\in (0,\nu(\Gamma_0,E_0))$.
\end{lemma}
\begin{proof}
Fix $\alpha\in (0,\nu(\Gamma_0,E_0))$. The result for $\sigma=1/2$ has already been established, giving
\be\label{one}
\|\cK_0g \|_{H^1_{-\alpha}(\RR,\tbV)}\leq c\|g\|_{H^1_{-\alpha}(\RR,\tbV)}\quad\mbox{for any}\quad g\in H^1_{-\alpha}(\RR,\tbV).
\ee
We use a bootstrap argument like that of \cite[p. 677, Proposition 3.1]{MZ} to extend to $1/2<\sigma<1$.
Namely, we use the fact, observed in \cite[Section 2]{MZ}, that $Q'(\obu)=M_*+K$, where $M_*$ is the operator of multiplication by a real valued function
bounded above and below and
\be\label{bd}
	\|\langle \cdot \rangle^{1/2} K \by\|_{\bY^\sigma}\leq c \| \langle \cdot\rangle^{-1/2} \by\|_{\bY^\sigma}\quad\mbox{for any}\quad \by\in\langle \cdot\rangle^{1/2}\bY^\sigma.
\ee
From Lemma~\ref{l2.8}(i) we have that $\cK_0g$ is the unique mild solution of equation \eqref{tildev-eq1} for any $g\in H^1_{-\alpha}(\RR,\tbV)$. It follows that
$\bu=\big(\Gamma_1\cK_0g+E_1\cV\cK_0g,-\tA_{11}^{-1}\tA_{12}\cK_0g,0,\cK_0g\big)^{\mathrm{T}}\in H^1_{-\alpha}(\RR,\bH)$ is a solution of the system \eqref{linear-sys-perturbed2} , which is equivalent to the system \eqref{linear-sys-perturbed} for $f=g$. We infer that
\begin{equation}\label{a}
A\bu'=Q'(\obu)\bu+g.
\end{equation}
Recall that $A$ is the multiplication operator by the real valued function $a:\RR^3\to\RR$ defined by $a(\xi)=\xi_1/\langle \xi\rangle$. Observing that the diagonal operator
$(M_a\partial_x -M_*)^{-1}$ is bounded on $H^1_{-\alpha}(\R,\bY^\sigma)$, and since $Q'(\obu)=M_*+K$, from \eqref{a} we obtain that
\begin{equation}\label{b}
\bu=(M_a\partial_x -M_*)^{-1}(g+K\bu).
\end{equation}
From \eqref{bd} and \eqref{b} we obtain that
\ba\label{bd2}
\|\bu\|_{H^1_{-\alpha}(\RR,\bY^\sigma)}&\leq c
\|g\|_{H^1_{-\alpha}(\R,\bY^\sigma\cap\tbV)}+c\|K\bu\|_{H^1_{-\alpha}(\R,\bY^\sigma)}\leq c
\|g\|_{H^1_{-\alpha}(\R,\bY^\sigma\cap\tbV)}+c\|\langle \cdot\rangle^{-1/2} \bu \|_{H^1_{-\alpha}(\R,\bY^\sigma)}.
\ea
Noting that $\langle \xi\rangle^{-1/2}M_{\bar u}^{-\sigma}(\xi)<\frac{1}{2}M_{\bar u}^{-\sigma}(\xi)$ for $|\xi|\geq C$ and some $C\gg 1$,
and $M_{\bar u}^{-\sigma}(\xi)\lesssim M_{\bar u}^{-1/2}(\xi)$ for $|\xi|\leq C$, we may rearrange \eqref{bd2} to
conclude that
\ba\label{bdf}
	\|\bu\|_{H^1_{-\alpha}(\RR,\bY^\sigma)}&\leq c
\|g\|_{H^1_{-\alpha}(\R,\bY^\sigma\cap\tbV)}+c\|\bu\|_{H^1_{-\alpha}(\R,\bY^{1/2})}\leq c\|g\|_{H^1_{-\alpha}(\R,\bY^\sigma\cap\tbV)}+c\|\bu\|_{H^1_{-\alpha}(\R,\bY^{1/2})}\\
&\leq c\|g\|_{H^1_{-\alpha}(\R,\bY^\sigma\cap\tbV)}+c\|\cK_0g\|_{H^1_{-\alpha}(\R,\tbV)}+c\|\cV\cK_0g\|_{H^1_{-\alpha}(\R,\tbV)}\\&\leq c \|g\|_{H^1_{-\alpha}(\R,\bY^\sigma\cap\tbV)}+c\|g\|_{H^1_{-\alpha}(\R,\tbV)}.
\ea
Since $P_{\tbV}\bu=\cK_0g$, from \eqref{bdf} it follows that
\begin{equation}\label{bdf2}
\|\cK_0g\|_{H^1_{-\alpha}(\RR,\bY^\sigma\cap\tbV)}\leq c \|g\|_{H^1_{-\alpha}(\R,\bY^\sigma\cap\tbV)}+c\|g\|_{H^1_{-\alpha}(\R,\tbV)}
\end{equation}
Define now $\cY^\sigma \sim \bY^\sigma$ to be the space determined by mixed norm
$ \|\by\|_{ \cY^\sigma}:=\|\by\|_{\bY^\sigma}+n\|\by\|_{\bY^{1/2}}$, where $n>>1$. Summing $n$ times \eqref{one} and \eqref{bdf2}, we obtain
$\|\cK_0g \|_{H^1_{-\alpha}(\R,\cY^{\sigma}\cap\tbV)}\leq c\|g \|_{H^1_{-\alpha}(\R,\cY^{\sigma}\cap\tbV)}$ for $n$ sufficiently large,
yielding the result, finally, by equivalence of $\bY^\sigma$ and $\cY^\sigma$.
\end{proof}

\begin{proof}[Proof of Proposition \ref{p1.4}]
Similarly as in the standard case $\bH=\bY^{1/2}$, the Volterra integral determining the part of our fixed-point mapping is readily seen to be bounded on
$H^1_{-\alpha}(\R,\bY^\sigma)$,
whence we may repeat our previous argument in its entirety to obtain existence of a center manifold
valued in $\bY^\sigma$, as claimed.
\end{proof}

\br\label{sharprmk}
It is easily seen that the result $\cM_\rmc \subset \bY^\sigma$, $1/2\leq \sigma<1$ is sharp, even in the noncharacteristic case.
For, consider the difference $v(\xi)= e^{-\theta |\xi|^2}- e^{-|\xi|^2}$, $0<\theta<1$,
between a base point Maxwellian $e^{-|\xi|^2}$ and a nearby equilibrium consisting of a different
Maxwellian $e^{-\theta |\xi|^2}$ with slightly slower decay in $|\xi|$.
Evidently, $v(\xi)\sim e^{-\theta |\xi|^2}$ for large $\xi$, whence $v\in \bY^{\theta}$ but $v\not \in
\bY^{1}$.
\er

\subsection{Physical behavior}\label{s:phys}
Specialized to Boltzmann's equation, the observations on center manifold structure in Theorem \ref{t1.2}
have a number of interesting physical applications, for example to
{\it Milne's problem} and {\it condensation/evaporation} phenomena \cite{So1,So2,So3}.
See \cite{LiuYu} for further discussion.\footnote{
	Our analysis shows
	that assumption $f$ even in $\xi_2$, $\xi_3$ of \cite{LiuYu}, restricting $\mathrm{dim}\ker A_{11}=1$
	in case \eqref{ldg}, may be dropped.
}
\appendix

\section{Smoothness of substitution operators}\label{appendix}
In this section we study the smoothness properties in the mixed-norm spaces $\cZ_{\gamma,\beta}(\bH)$ topology of substitution operators $\cN: \overline{B}_{H^1_{-\alpha}(\RR,\bH)}(0,\delta)\to H^1_{-\alpha}(\RR,\bH)$
defined by
\begin{equation}\label{subst}
\big(\cN(f)\big)(x)=N(f(x)),\quad x\in\RR,\; f\in H^1_{-\alpha}(\RR,\bH).
\end{equation}
Here $N:\bH\to\bH$ is a bounded, nonlinear $C^\infty$ function on $\bH$, whose derivatives are bounded. Moreover, the weights $\alpha>0$, $\beta>0$ and $\gamma>0$ satisfy the conditions of Lemma~\ref{l3.2}, namely $0<\alpha<\gamma<\beta$ and $0<2\alpha<\beta-\gamma$.
\begin{lemma}\label{A1}
Assume  $N:\bH\to\bH$ is a $C^\infty$ function on $\bH$ such that $\sup_{h\in\bH}\|N^{(j)}(h)\|<\infty$ for any $j\geq 0$. Then, the substitution operator $\cN$ defined in \eqref{subst} is a $C^1$ function from $\cZ_{\gamma,\beta}(\bH)$ to $\cZ_{2\gamma,2\beta}(\bH)$. Moreover, $\cN'(f)$ can be extended to a bounded linear operator on $\cZ_{\gamma,\beta}(\bH)$ for any $f\in\overline{B}_{H^1_{-\alpha}(\RR,\bH)}(0,\delta)$.
\end{lemma}
\begin{proof} We fix $f_0\in\overline{B}_{H^1_{-\alpha}(\RR,\bH)}(0,\delta)$ and we define $T_0:\cZ_{\gamma,\beta}(\bH)\to\cZ_{\gamma,\beta}(\bH)$ by
\begin{equation}\label{A1-1}
(T_0z)(x)=N'(f_0(x))z(x),\quad x\in\RR,\;z\in\cZ_{\gamma,\beta}(\bH).
\end{equation}

\noindent{\bf Claim 1.} $T_0$ is well-defined and bounded on $\cZ_{\gamma,\beta}(\bH)$. Since $N'$ is a bounded function on $\bH$ one can readily check that $\|(T_0z)(x)\|\leq \|N'(f_0(x))\|\,\|z(x)\|\leq\sup_{h\in\bH}\|N'(h)\|\,\|z(x)\|$ for any $x\in\RR$ and $z\in\cZ_{\gamma,\beta}(\bH)$, which implies that
\begin{equation}\label{A1-2}
T_0z\in L^2_{-\gamma}(\RR,\bH)\;\mbox{and}\;\|T_0z\|_{L^2_{-\gamma}}\leq \sup_{h\in\bH}\|N'(h)\|\,\|z\|_{L^2_{-\gamma}}\;\mbox{for any}\;z\in\cZ_{\gamma,\beta}(\bH).
\end{equation}
Next, we note that $T_0z\in H^1_{\mathrm{loc}}(\RR,\bH)$ and $(T_0z)'(x)=N''(f_0(x))\Big(f_0'(x),z(x)\Big)+N'(f_0(x))z'(x)$ for any $x\in\RR$ and $z\in\cZ_{\gamma,\beta}(\bH)$.
Since the functions $N'$ and $N''$ are bounded functions on $\bH$ and $2\alpha<\beta-\gamma$ from \eqref{sob} we obtain that
\begin{align}\label{A1-3}
\int_\RR e^{-2\beta|x|}\|(T_0z)'(x)\|^2\rmd x&\leq 2\sup_{h\in\bH}\|N''(h)\|^2\int_\RR e^{-2\beta|x|}\|f_0'(x)\|^2\|z(x)\|^2\rmd x+2\sup_{h\in\bH}\|N'(h)\|^2\|z'\|_{L^2_{-\beta}}^2\nonumber\\
&\leq \Big(2\sup_{h\in\bH}\|N''(h)\|^2\int_\RR e^{-(\beta-\gamma)|x|}\|f_0'(x)\|^2\rmd x+2\sup_{h\in\bH}\|N'(h)\|^2\Big)\|z\|_{\cZ_{\gamma,\beta}}^2\nonumber\\
&\leq \Big(2\sup_{h\in\bH}\|N''(h)\|^2\|f_0\|_{H^1_{-\alpha}}^2+2\sup_{h\in\bH}\|N'(h)\|^2\Big)\|z\|_{\cZ_{\gamma,\beta}}^2\nonumber\\
&\leq\Big(2\delta^2\sup_{h\in\bH}\|N''(h)\|^2+2\sup_{h\in\bH}\|N'(h)\|^2\Big)\|z\|_{\cZ_{\gamma,\beta}}^2<\infty
\end{align}
for any $z\in\cZ_{\gamma,\beta}(\bH)$, proving Claim 1.

\noindent{\bf Claim 2.} $\cN$ is Frechet differentiable at $f_0$ in the $\big(\cZ_{\gamma,\beta}(\bH),\cZ_{2\gamma,2\beta}(\bH)\big)$ topology and $\cN'(f_0)=T_0$.
Since $N$ is a $C^\infty$ function on $\bH$ we have that
\begin{equation}\label{A1-4}
N(h_1)-N(h_2)-N'(h_2)(h_1-h_2)=\Big(\int_0^1sN''\big(sh_1+(1-s)h_2\big)\rmd s\Big)(h_1-h_2,h_1-h_2)
\end{equation}
for any $h_1,h_2\in\bH$. Since $N''$ is a bounded function on $\bH$ from \eqref{A1-4} we conclude that
\begin{equation}\label{A1-5}
\|N(h_1)-N(h_2)-N'(h_2)(h_1-h_2)\|\leq \sup_{h\in\bH}\|N''(h)\|\,\|h_1-h_2\|^2\;\mbox{for any}\; h_1,h_2\in\bH.
\end{equation}
Using Sobolev's inequality, from \eqref{A1-5} we obtain that
\begin{align}\label{A1-6}
\|\cN(f)&-\cN(f_0)-T_0(f-f_0)\|_{L^2_{-2\gamma}}^2\leq \sup_{h\in\bH}\|N''(h)\|^2\int_{\RR} e^{-4\gamma|x|}\|f(x)-f_0(x)\|^4\rmd x\nonumber\\
&\leq \sup_{h\in\bH}\|N''(h)\|^2 \|f-f_0\|_{L^2_{-\gamma}}^2\|f-f_0\|_{L^\infty_{-\gamma}}^2\leq \sup_{h\in\bH}\|N''(h)\|^2 \|f-f_0\|_{L^2_{-\gamma}}^3\|f'-f_0'\|_{L^2_{-\gamma}}\nonumber\\
&\leq \sup_{h\in\bH}\|N''(h)\|^2 \|f-f_0\|_{L^2_{-\gamma}}^3\|f-f_0\|_{H^1_{-\alpha}}\leq 2\delta\sup_{h\in\bH}\|N''(h)\|^2 \|f-f_0\|_{\cZ_{\gamma,\beta}}^3
\end{align}
for any $f\in\overline{B}_{H^1_{-\alpha}(\RR,\bH)}(0,\delta)$. Since $N$ is a $C^\infty$ function on $\bH$ and $T_0\in\cB(\cZ_{\gamma,\beta}(\bH))$ we infer that $\cN(f)-\cN(f_0)-T_0(f-f_0)\in H^1_{\mathrm{loc}}(\RR,\bH)$ and
\begin{align}\label{A1-7}
\big(\cN(f)-\cN(f_0)&-T_0(f-f_0)\big)'(x)=N'(f(x))f'(x)-N'(f_0(x))f_0'(x)-N'(f_0(x))\big(f'(x)-f_0'(x)\big)\nonumber\\
&\qquad\qquad\qquad\qquad-N''(f_0(x))\big(f_0'(x),f(x)-f_0(x)\big)\nonumber\\
&=\big(N'(f(x))-N'(f_0(x))\big)f'(x)-N''(f_0(x))\big(f'(x),f(x)-f_0(x)\big)\nonumber\\
&\qquad\qquad\qquad\qquad+N''(f_0(x))\big(f'(x)-f_0'(x),f(x)-f_0(x)\big)
\end{align}
for any $x\in\RR$ and $f\in\overline{B}_{H^1_{-\alpha}(\RR,\bH)}(0,\delta)$. Using again that $N$ is a $C^\infty$ function on $\bH$ we have that
\begin{equation}\label{A1-8}
N'(h_1)h_3-N'(h_2)h_3-N''(h_2)(h_3,h_1-h_2)=\Big(\int_0^1sN'''\big(sh_1+(1-s)h_2\big)\rmd s\Big)(h_3,h_1-h_2,h_1-h_2)
\end{equation}
Since $N'''$ is a bounded function on $\bH$ from \eqref{A1-8} it follows that
\begin{equation}\label{A1-9}
\|N'(h_1)h_3-N'(h_2)h_3-N''(h_2)(h_3,h_1-h_2)\|\leq \sup_{h\in\bH}\|N'''(h)\|\,\|h_3\|\,\|h_1-h_2\|^2\;\mbox{for any}\; h_1,h_2,h_3\in\bH.
\end{equation}
Since $2\alpha<\beta-\gamma$ from \eqref{sob}, \eqref{A1-7} and \eqref{A1-9} we obtain that
\begin{align}\label{A1-10}
\big\|\big(\cN(f)&-\cN(f_0)-T_0(f-f_0)\big)'\big\|_{L^2_{-2\beta}}^2\leq 2\sup_{h\in\bH}\|N'''(h)\|^2\int_{\RR} e^{-4\beta|x|}\|f'(x)\|^2\|f(x)-f_0(x)\|^4\rmd x\nonumber\\
&\qquad\qquad\qquad\qquad+2\sup_{h\in\bH}\|N''(h)\|^2\int_{\RR} e^{-4\beta|x|}\|f'(x)-f_0'(x)\|^2\|f(x)-f_0(x)\|^2\rmd x\nonumber\\
&\leq 2\sup_{h\in\bH}\|N'''(h)\|^2\Big(\int_{\RR} e^{-4\alpha|x|}\|f'(x)\|^2\rmd x\Big)\|f-f_0\|_{\cZ_{\gamma,\beta}}^4\nonumber\\
&\qquad\qquad\qquad\qquad+2\sup_{h\in\bH}\|N''(h)\|^2\|f'-f_0'\|_{L^2_{-\beta}}^2\|f-f_0\|_{L^\infty_{-\beta}}^2\nonumber\\
&\leq 2\sup_{h\in\bH}\|N'''(h)\|^2\|f\|_{H^1_{-\alpha}}\|f-f_0\|_{\cZ_{\gamma,\beta}}^4\nonumber\\
&\qquad\qquad\qquad\qquad+2\sup_{h\in\bH}\|N''(h)\|^2\|f-f_0\|_{\cZ_{\gamma,\beta}}^2\|f-f_0\|_{L^2_{-\beta}}\|f'-f_0'\|_{L^2_{-\beta}}\nonumber\\
&\leq \Big(2\delta\sup_{h\in\bH}\|N'''(h)\|^2+2\sup_{h\in\bH}\|N''(h)\|^2\Big)\|f-f_0\|_{\cZ_{\gamma,\beta}}^4
\end{align}
for any $f\in\overline{B}_{H^1_{-\alpha}(\RR,\bH)}(0,\delta)$. From \eqref{A1-6} and \eqref{A1-10} we conclude that there exists $c>0$ such that
\begin{equation}\label{A1-11}
\|\cN(f)-\cN(f_0)-T_0(f-f_0)\|_{\cZ_{2\gamma,2\beta}}\leq c\|f-f_0\|_{\cZ_{\gamma,\beta}}^{3/2}+c\|f-f_0\|_{\cZ_{\gamma,\beta}}^2
\end{equation}
for any $f\in\overline{B}_{H^1_{-\alpha}(\RR,\bH)}(0,\delta)$, proving Claim 2.

\noindent{\bf Claim 3.} $\cN'$ is continuous in the $\Big(\cZ_{\gamma,\beta}(\bH),\cB\big(\cZ_{\gamma,\beta}(\bH),\cZ_{2\gamma,2\beta}(\bH)\big)\Big)$ topology. First, we fix $f_0\in\overline{B}_{H^1_{-\alpha}(\RR,\bH)}(0,\delta)$. Since $N$ is a $C^\infty$ function on $\bH$ and its derivatives are bounded on $\bH$ we have that
\begin{equation}\label{Lagrange-est}
\|N^{(j}(h_1)-N^{(j)}(h_2)\|\leq \sup_{h\in\bH}\|N^{(k+1)}(h)\|\,\|h_1-h_2||\quad\mbox{for any}\quad h_1,h_2\in\bH, j\geq 1,
\end{equation}
which implies that
\begin{align}\label{A1-12}
\|\cN'(f)z-\cN'(f_0)z\|_{L^2_{-2\gamma}}^2&=\int_{\RR} e^{-4\gamma|x|}\|N'(f(x))z(x)-N'(f_0(x))z(x)\|^2\rmd x\nonumber\\
&\leq \sup_{h\in\bH}\|N''(h)\|^2\int_{\RR} e^{-4\gamma|x|}\|f(x)-f_0(x)\|^2\|z(x)\|^2\rmd x\nonumber\\
&\leq \sup_{h\in\bH}\|N''(h)\|^2\|z\|_{L^2_{-\gamma}}^2\|f-f_0\|_{L^\infty_{-\gamma}}^2\nonumber\\
&\leq \sup_{h\in\bH}\|N''(h)\|^2\|z\|_{\cZ_{\gamma,\beta}}^2\|f-f_0\|_{L^2_{-\gamma}}\|f-f_0\|_{H^1_{-\alpha}}\nonumber\\
&\leq 2\delta\sup_{h\in\bH}\|N''(h)\|^2\|z\|_{\cZ_{\gamma,\beta}}^2\|f-f_0\|_{\cZ_{\gamma,\beta}}
\end{align}
for any $f\in\overline{B}_{H^1_{-\alpha}(\RR,\bH)}(0,\delta)$ and $z\in\cZ_{\gamma,\beta}(\bH)$. Moreover, for any $f\in\overline{B}_{H^1_{-\alpha}(\RR,\bH)}(0,\delta)$ and $z\in\cZ_{\gamma,\beta}(\bH)$ one can readily check that $\cN'(f)z-\cN'(f_0)z\in H^1_{\mathrm{loc}}(\RR,\bH)$ and
\begin{align}\label{A1-13}
\big(\cN'(f)z-\cN'(f_0)z\big)'(x)&=N''(f(x))\big(f'(x),z(x)\big)-N''(f_0(x))\big(f_0'(x),z(x)\big)\nonumber\\
&+\big(N'(f(x)-N'(f_0(x)\big)z'(x)=F_1(x)+F_2(x)+F_3(x)
\end{align}
for any $x\in\RR$, where the functions $F_j:\RR\to\bH$, $j=1,2,3$, are defined by
\begin{align}\label{A1-14}
F_1(x)&=\big(N''(f(x))-N''(f_0(x))\big)\big(f'(x),z(x)\big),\nonumber\\
F_2(x)&=N''(f_0(x))\big(f'(x)-f_0'(x),z(x)\big),\nonumber\\
F_3(x)&=\big(N'(f(x)-N'(f_0(x)\big)z'(x).
\end{align}
Next, we estimate the $L^2_{-2\gamma}(\RR,\bH)$-norm of $F_j$, $j=1,2,3$, using \eqref{sob} and \eqref{Lagrange-est}.
\begin{align}\label{A1-15}
\|F_1\|_{L^2_{-2\beta}}^2&\leq\sup_{h\in\bH}\|N'''(h)\|^2\int_{\RR} e^{-4\beta|x|}\|f(x)-f_0(x)\|^2\|f'(x)\|^2\|z(x)\|^2\rmd x\nonumber\\
&\leq\sup_{h\in\bH}\|N'''(h)\|^2\Big(\int_{\RR} e^{-2(\beta-\gamma)|x|}\|f'(x)\|^2\rmd x\Big)\|f-f_0\|_{\cZ_{\gamma,\beta}}^2\|z\|_{\cZ_{\gamma,\beta}}^2\nonumber\\
&\leq\sup_{h\in\bH}\|N'''(h)\|^2\Big(\int_{\RR} e^{-4\alpha|x|}\|f'(x)\|^2\rmd x\Big)\|f-f_0\|_{\cZ_{\gamma,\beta}}^2\|z\|_{\cZ_{\gamma,\beta}}^2\nonumber\\
&\leq\delta^2\sup_{h\in\bH}\|N'''(h)\|^2\|f-f_0\|_{\cZ_{\gamma,\beta}}^2\|z\|_{\cZ_{\gamma,\beta}}^2;
\end{align}
\begin{align}\label{A1-16}
\|F_2\|_{L^2_{-2\beta}}^2&\leq\sup_{h\in\bH}\|N''(h)\|^2\int_{\RR} e^{-4\beta|x|}\|f'(x)-f_0'(x)\|^2\|z(x)\|^2\rmd x\nonumber\\
&\leq\sup_{h\in\bH}\|N''(h)\|^2\Big(\int_{\RR} e^{-2(\alpha+\beta)|x|}\|f'(x)-f_0'(x)\|^2\rmd x\Big)\|z\|_{\cZ_{\gamma,\beta}}^2\nonumber\\
&\leq\sup_{h\in\bH}\|N''(h)\|^2\|f'-f_0'\|_{L^2_{-\beta}}^2\|z\|_{\cZ_{\gamma,\beta}}^2\leq\sup_{h\in\bH}\|N''(h)\|^2\|f-f_0\|_{\cZ_{\gamma,\beta}}^2\|z\|_{\cZ_{\gamma,\beta}}^2;
\end{align}
\begin{align}\label{A1-17}
\|F_3\|_{L^2_{-2\beta}}^2&\leq\sup_{h\in\bH}\|N''(h)\|^2\int_{\RR} e^{-4\beta|x|}\|f(x)-f_0(x)\|^2\|z'(x)\|^2\rmd x\leq\sup_{h\in\bH}\|N''(h)\|^2\|f-f_0\|_{L^\infty_{-\beta}}^2\|z'\|_{L^2_{-\beta}}^2\nonumber\\
&\leq\sup_{h\in\bH}\|N''(h)\|^2\|f-f_0\|_{L^2_{-\beta}}\|f-f_0\|_{H^1_{-\beta}}\|z\|_{\cZ_{\gamma,\beta}}^2
\leq2\delta\sup_{h\in\bH}\|N''(h)\|^2\|f-f_0\|_{\cZ_{\gamma,\beta}}\|z\|_{\cZ_{\gamma,\beta}}^2.
\end{align}
Summarizing the estimates \eqref{A1-15}--\eqref{A1-17}, from \eqref{A1-13} we conclude that there exists $c>0$ such that
\begin{equation}\label{A1-18}
\big\|\big(\cN'(f)z-\cN'(f_0)z\big)'\big\|_{L^2_{-2\beta}}\leq c\big(\|f-f_0\|_{\cZ_{\gamma,\beta}}^{1/2}+\|f-f_0\|_{\cZ_{\gamma,\beta}}\big)\|z\|_{\cZ_{\gamma,\beta}}
\end{equation}
for any $f\in\overline{B}_{H^1_{-\alpha}(\RR,\bH)}(0,\delta)$ and $z\in\cZ_{\gamma,\beta}(\bH)$. From \eqref{A1-12} and \eqref{A1-18} we obtain that
\begin{equation}\label{A1-19}
\|\cN'(f)-\cN'(f_0)\|_{\cB\big(\cZ_{\gamma,\beta}(\bH),\cZ_{2\gamma,2\beta}(\bH)\big)}\leq c\big(\|f-f_0\|_{\cZ_{\gamma,\beta}}^{1/2}+\|f-f_0\|_{\cZ_{\gamma,\beta}}\big)
\end{equation}
for any $f\in\overline{B}_{H^1_{-\alpha}(\RR,\bH)}(0,\delta)$, proving Claim 3 and the lemma.
\end{proof}
To prove higher order differentiability of the nonlinear map $\cN$ defined in \eqref{subst}, we need to study the smoothness properties of operator-valued substitution operators. We recall that for any three Hilbert space $\bX$, $\bY$ and $\bZ$ we can identify the set $\cB\big(\bX,\cB(\bY,\bZ)\big)$ with the set of bilinear maps from $\bX\times\bY$ to $\bZ$, denoted by $\cB_2(\bX\times\bY,\bZ)$.
\begin{lemma}\label{A2}
Let $\bH$ and $\tbH$ be two Hilbert spaces,  $L:\tbH\to\cB(\bH)$ a $C^\infty$, $p\geq 0$ such that
\begin{equation}\label{1-2-deriv-bound}
 \|L'(\tbh)\|\leq c\|\tbh\|^p\;\mbox{and}\; \|L''(\tbh)\|\leq c\|\tbh\|^p\;\mbox{for any}\;\tbh\in\tbH
\end{equation}
and $0<\gamma<\beta$. Denoting by $\tgamma=(p+2)\beta+\gamma$ and $\tbeta=(p+4)\beta$, the nonlinear map $\cW:\cZ_{\gamma,\beta}(\tbH)\to\cB\big(\cZ_{\gamma,\beta}(\bH),\cZ_{\tgamma,\tbeta}(\bH)\big)$ defined by
\begin{equation}\label{def-cW}
\big(\cW(f)z\big)(x)=L(f(x))z(x),\;\mbox{for}\; f\in\cZ_{\gamma,\beta}(\tbH),\,z\in\cZ_{\gamma,\beta}(\bH),\,x\in\RR,
\end{equation}
is of class $C^1$ and
\begin{equation}\label{deriv-cW}
\big(\cW'(f)(z_1,z_2)\big)(x)=L'(f(x))\big(z_1(x),z_2(x)\big),\;\mbox{for}\; f,z_1\in\cZ_{\gamma,\beta}(\tbH),\, z_2\in\cZ_{\gamma,\beta}(\bH),\,x\in\RR.
\end{equation}
\end{lemma}
\begin{proof} We fix $f_0\in\cZ_{\gamma,\beta}(\tbH)$ and define $\cD_0:\cZ_{\gamma,\beta}(\tbH)\times\cZ_{\gamma,\beta}(\bH)\to\cZ_{\tgamma,\tbeta}(\bH)$ by
\begin{equation}\label{A2-1}
\big(\cD_0(z_1,z_2)\big)(x)=L'(f_0(x))\big(z_1(x),z_2(x)\big),\;\mbox{for}\; z_1\in\cZ_{\gamma,\beta}(\tbH),\, z_2\in\cZ_{\gamma,\beta}(\bH),\,x\in\RR.
\end{equation}
\noindent{\bf Claim 1.} The bilinear map $\cD_0$ is well-defined and bounded.
Indeed, from \eqref{1-2-deriv-bound} it follows that
\begin{equation}\label{A2-2}
\big\|\big(\cD_0(z_1,z_2)\big)(x)\big\|\leq c\|f_0(x)\|^p\|z_1(x)\|\,\|z_2(x)\| \;\mbox{for any}\; z_1\in\cZ_{\gamma,\beta}(\tbH),\, z_2\in\cZ_{\gamma,\beta}(\bH),\,x\in\RR,
\end{equation}
which implies that
\begin{align}\label{A2-3}
\|\cD_0(z_1,z_2)\|_{L^2_{-\tgamma}}&\leq c\|f_0\|_{L^\infty_{-\beta}}^p\|z_1\|_{L^\infty_{-\beta}}\|z_2\|_{L^2_{-\gamma}}\leq c\|f_0\|_{H^1_{-\beta}}^p\|z_1\|_{H^1_{-\beta}}\|z_2\|_{\cZ_{\gamma,\beta}}\nonumber\\
&\leq c\|f_0\|_{\cZ_{\gamma,\beta}}^p\|z_1\|_{\cZ_{\gamma,\beta}}\|z_2\|_{\cZ_{\gamma,\beta}}\;\mbox{for any}\; z_1\in\cZ_{\gamma,\beta}(\tbH),\, z_2\in\cZ_{\gamma,\beta}(\bH).
\end{align}
Since $L$ is a $C^\infty$ function from $\tbH$ to $\cB(\bH)$, we have that $\cD_0(z_1,z_2)\in H^1_{\mathrm{loc}}(\RR,\bH)$ and
\begin{equation*}
\big(\cD_0(z_1,z_2)\big)'(x)=L'(f_0(x))\big(z_1'(x),z_2(x)\big)+L'(f_0(x))\big(z_1(x),z_2'(x)\big)
+L''(f_0(x))\big(f_0'(x),z_1(x),z_2(x)\big)
\end{equation*}
for any $z_1\in\cZ_{\gamma,\beta}(\tbH)$, $z_2\in\cZ_{\gamma,\beta}(\bH)$ and $x\in\RR$, which implies that
\begin{align}\label{A2-4}
\big\|\big(\cD_0(z_1,z_2)\big)'(x)\big\|\leq c\|f_0(x)\|^p\Big[\|f_0'(x)\|\,\|z_1(x)\|\,\|z_2(x)\|+\|z_1'(x)\|\,\|z_2(x)\|+\|z_1(x)\|\,\|z_2'(x)\|\Big]
\end{align}
for any $z_1\in\cZ_{\gamma,\beta}(\tbH)$, $z_2\in\cZ_{\gamma,\beta}(\bH)$ and $x\in\RR$. From \eqref{1-2-deriv-bound} and \eqref{A2-4} we obtain that
\begin{align}\label{A2-5}
\|(\cD_0&(z_1,z_2))'\|_{L^2_{-\tbeta}}\leq c\|f_0\|_{L^\infty_{-\beta}}^p\Big[\|f_0'\|_{L^2_{-\beta}}\|z_1\|_{L^\infty_{-\beta}}\|z_2\|_{L^\infty_{-\beta}}+
\|z_1'\|_{L^2_{-\beta}}\|z_2\|_{L^\infty_{-\beta}}+\|z_1\|_{L^\infty_{-\beta}}\|z_2'\|_{L^2_{-\beta}}\Big]\nonumber\\
&\leq c\|f_0\|_{H^1_{-\beta}}^p\big(\|f_0\|_{\cZ_{\gamma,\beta}}+1)\|z_1\|_{H^1_{-\beta}}\|z_2\|_{H^1_{-\beta}}\leq c\|f_0\|_{\cZ_{\gamma,\beta}}^p\big(\|f_0\|_{\cZ_{\gamma,\beta}}+1\big)\|z_1\|_{\cZ_{\gamma,\beta}}\|z_2\|_{\cZ_{\gamma,\beta}}
\end{align}
for any $z_1\in\cZ_{\gamma,\beta}(\tbH)$, $z_2\in\cZ_{\gamma,\beta}(\bH)$. From \eqref{A2-3} and \eqref{A2-5} we conclude that $\cD_0(z_1,z_2)\in\cZ_{\tgamma,\tbeta}(\bH)$ and
\begin{equation}\label{A2-6}
\|\cD_0(z_1,z_2)\|_{\cZ_{\tgamma,\tbeta}}\leq c\|f_0\|_{\cZ_{\gamma,\beta}}^p\big(\|f_0\|_{\cZ_{\gamma,\beta}}^2+1\big)^{1/2}\|z_1\|_{\cZ_{\gamma,\beta}}\|z_2\|_{\cZ_{\gamma,\beta}}
\end{equation}
for any $z_1\in\cZ_{\gamma,\beta}(\tbH)$, $z_2\in\cZ_{\gamma,\beta}(\bH)$, proving Claim 1.

\noindent{\bf Claim 2.} $\cW$ is Frechet differentiable at $f_0$ and $\cW'(f_0)=\cD_0$. Since $L$ is $C^\infty$ function from $\tbH$ to $\cB(\bH)$ we have that
\begin{equation}\label{A2-7}
L(\tbh_1)g-L(\tbh_2)g-L'(\tbh_2)(\tbh_1-\tbh_2,g)=\Big(\int_0^1sL''\big(s\tbh_1+(1-s)\tbh_2\big)\rmd s\Big)(\tbh_1-\tbh_2,\tbh_1-\tbh_2,g)
\end{equation}
for any $\tbh_1,\tbh_2\in\tbH$ and $g\in\bH$. From \eqref{1-2-deriv-bound} we infer that
\begin{align}\label{A2-8}
\|L(\tbh_1)g-L(\tbh_2)g-L'(\tbh_2)(\tbh_1-\tbh_2,g)\|&\leq c\Big(\int_0^1s\|s\tbh_1+(1-s)\tbh_2\|^p\rmd s\Big)\|\tbh_1-\tbh_2\|^2\|g\|\nonumber\\
&\leq c\big(\|\tbh_1\|^p+\|\tbh_2\|^p\big)\|\tbh_1-\tbh_2\|^2\|g\|
\end{align}
for any $\tbh_1,\tbh_2\in\tbH$ and $g\in\bH$. From \eqref{def-cW}, \eqref{A2-1} and \eqref{A2-8} it follows that
\begin{equation}\label{A2-9}
\big\|\big(\cW(f)z-\cW(f_0)z-\cD_0(f-f_0,z)\big)(x)\big\|\leq c\big(\|f(x)\|^p+\|f_0(x)\|^p\big)\|f(x)-f_0(x)\|^2\|z(x)\|
\end{equation}
for any $f\in\cZ_{\gamma,\beta}(\tbH)$, $z\in\cZ_{\gamma,\beta}(\bH)$ and $x\in\RR$, which implies that
\begin{align}\label{A2-10}
\|\cW(f)z&-\cW(f_0)z-\cD_0(f-f_0,z)\|_{L^2_{-\tgamma}}\leq c(\|f\|_{L^\infty_{-\beta}}^p+\|f_0\|_{L^\infty_{-\beta}}^p)\|f-f_0\|_{L^\infty_{-\beta}}^2\|z\|_{L^2_{-\gamma}}\nonumber\\
&\leq c(\|f\|_{H^1_{-\beta}}^p+\|f_0\|_{H^1_{-\beta}}^p)\|f-f_0\|_{H^1_{-\beta}}^2\|z\|_{L^2_{-\gamma}}\leq c(\|f\|_{\cZ_{\gamma,\beta}}^p+\|f_0\|_{\cZ_{\gamma,\beta}}^p)\|f-f_0\|_{\cZ_{\gamma,\beta}}^2\|z\|_{\cZ_{\gamma,\beta}}
\end{align}
for any $f\in\cZ_{\gamma,\beta}(\tbH)$, $z\in\cZ_{\gamma,\beta}(\bH)$. Next, we need to establish a couple of estimates satisfied by the function $L$ and its derivatives that are needed in the sequel. We note that
\begin{align}\label{A2-11}
\big(L'(\tbh_1)&-L'(\tbh_2)\big)(\tbh_3,g)-L''(\tbh_2)(\tbh_3,\tbh_1-\tbh_2,g)\nonumber\\&=\Big(\int_0^1sL'''\big(s\tbh_1+(1-s)\tbh_2\big)\rmd s\Big)(\tbh_3,\tbh_1-\tbh_2,\tbh_1-\tbh_2,g)\nonumber\\
\big(L'(\tbh_1)&-L'(\tbh_2)\big)(\tbh_3,g)=\Big(\int_0^1L''\big(s\tbh_1+(1-s)\tbh_2\big)\rmd s\Big)(\tbh_3,\tbh_1-\tbh_2,g)
\end{align}
for any $\tbh_1,\tbh_2,\tbh_3\in\tbH$ and $g\in\bH$, which implies that
\begin{align}\label{A2-12}
\big\|\big(L'(\tbh_1)-L'(\tbh_2)\big)(\tbh_3,g)-L''(\tbh_2)(\tbh_3,\tbh_1-\tbh_2,g)\big\|&\leq c\big(\|\tbh_1\|^p+\|\tbh_2\|^p\big)\|\tbh_1-\tbh_2\|^2\|\tbh_3\|\,\|g\|\nonumber\\
\big\|\big(L'(\tbh_1)-L'(\tbh_2)\big)(\tbh_3,g)\big\|&\leq c\big(\|\tbh_1\|^p+\|\tbh_2\|^p\big)\|\tbh_1-\tbh_2\|\,\|\tbh_3\|\,\|g\|
\end{align}
for any $\tbh_1,\tbh_2,\tbh_3\in\tbH$ and $g\in\bH$.

Since $L$ is $C^\infty$ function from $\tbH$ to $\cB(\bH)$ and $\cD_0\in\cB_2(\cZ_{\gamma,\beta}(\tbH)\times\cZ_{\gamma,\beta}(\bH),\cZ_{\tgamma,\tbeta}(\bH))$ from \eqref{def-cW}, \eqref{A2-1} we infer that for any $f\in\cZ_{\gamma,\beta}(\tbH)$, $z\in\cZ_{\gamma,\beta}(\bH)$ the function $\cW(f)z-\cW(f_0)z-\cD_0(f-f_0,z)\in H^1_{\mathrm{loc}}(\RR,\bH)$ and
\begin{align}\label{A2-13}
\big(\cW(f)z&-\cW(f_0)z-\cD_0(f-f_0,z)\big)'(x)=L'(f(x))\big(f'(x),z(x)\big)+L(f(x))z'(x)\nonumber\\
&\quad-L'(f_0(x))\big(f_0'(x),z(x)\big)-L(f_0(x))z'(x)-L''(f_0(x))\big(f_0'(x),f(x)-f_0(x),z(x)\big)\nonumber\\
&\quad-L'(f_0(x))\big(f'(x)-f'_0(x),z(x)\big)-L'(f_0(x))\big(f(x)-f_0(x),z'(x)\big)\nonumber\\
&=\big(L'(f(x))-L'(f_0(x))\big)\big(f'(x),z(x)\big)-L''(f_0(x))\big(f_0'(x),f(x)-f_0(x),z(x)\big)\nonumber\\
&\quad+\big(L(f(x))-L(f_0(x))\big)z'(x)-L'(f_0(x))\big(f(x)-f_0(x),z'(x)\big)=G_1(x)+G_2(x)+G_3(x)
\end{align}
for any $x\in\RR$, where the functions $F_j:\RR\to\bH$, $j=1,2,3$, are defined by
\begin{align}\label{A2-14}
G_1(x)&=\big(L'(f(x))-L'(f_0(x))\big)\big(f'(x)-f_0'(x),z(x)\big),\nonumber\\
G_2(x)&=\big(L'(f(x))-L'(f_0(x))\big)\big(f_0'(x),z(x)\big)-L''(f_0(x))\big(f_0'(x),f(x)-f_0(x),z(x)\big),\nonumber\\
G_3(x)&=\big(L(f(x))-L(f_0(x))\big)z'(x)-L'(f_0(x))\big(f(x)-f_0(x),z'(x)\big).
\end{align}
We use \eqref{A2-8} and \eqref{A2-12} to estimate the $L^2_{-\tbeta}(\RR,\bH)$ of the functions $G_j$, $j=1,2,3$, defined in \eqref{A2-14}. Indeed, we have that
\begin{align}\label{A2-15}
\|G_1\|_{L^2_{-\tbeta}}&\leq  c(\|f\|_{L^\infty_{-\beta}}^p+\|f_0\|_{L^\infty_{-\beta}}^p)\|f-f_0\|_{L^\infty_{-\beta}}\|f'-f_0'\|_{L^2_{-\beta}}\|z\|_{L^\infty_{-\beta}}\nonumber\\
&\leq c(\|f\|_{H^1_{-\beta}}^p+\|f_0\|_{H^1_{-\beta}}^p)\|f-f_0\|_{H^1_{-\beta}}\|f-f_0\|_{\cZ_{\gamma,\beta}}\|z\|_{H^1_{-\beta}}\nonumber\\
&\leq c(\|f\|_{\cZ_{\gamma,\beta}}^p+\|f_0\|_{\cZ_{\gamma,\beta}}^p)\|f-f_0\|_{\cZ_{\gamma,\beta}}^2\|z\|_{\cZ_{\gamma,\beta}};
\end{align}
\begin{align}\label{A2-16}
\|G_2\|_{L^2_{-\tbeta}}&\leq c(\|f\|_{L^\infty_{-\beta}}^p+\|f_0\|_{L^\infty_{-\beta}}^p)\|f_0'\|_{L^2_{-\beta}}\|f-f_0\|_{L^\infty_{-\beta}}^2\|z\|_{L^\infty_{-\beta}}\nonumber\\
&\leq c(\|f\|_{H^1_{-\beta}}^p+\|f_0\|_{H^1_{-\beta}}^p)\|f_0\|_{\cZ_{\gamma,\beta}}\|f-f_0\|_{H^1_{-\beta}}^2\|z\|_{H^1_{-\beta}}\nonumber\\
&\leq c(\|f\|_{\cZ_{\gamma,\beta}}^p+\|f_0\|_{\cZ_{\gamma,\beta}}^p)\|f_0\|_{\cZ_{\gamma,\beta}}\|f-f_0\|_{\cZ_{\gamma,\beta}}^2\|z\|_{\cZ_{\gamma,\beta}};
\end{align}
\begin{align}\label{A2-17}
\|G_3\|_{L^2_{-\tbeta}}&\leq c(\|f\|_{L^\infty_{-\beta}}^p+\|f_0\|_{L^\infty_{-\beta}}^p)\|f-f_0\|_{L^\infty_{-\beta}}^2\|z'\|_{L^2_{-\beta}}\nonumber\\
&\leq c(\|f\|_{H^1_{-\beta}}^p+\|f_0\|_{H^1_{-\beta}}^p)\|f-f_0\|_{H^1_{-\beta}}^2\|z\|_{\cZ_{\gamma,\beta}}\nonumber\\
&\leq c(\|f\|_{\cZ_{\gamma,\beta}}^p+\|f_0\|_{\cZ_{\gamma,\beta}}^p)\|f-f_0\|_{\cZ_{\gamma,\beta}}^2\|z\|_{\cZ_{\gamma,\beta}}.
\end{align}
From \eqref{A2-10}, \eqref{A2-13}, \eqref{A2-15}, \eqref{A2-16} and \eqref{A2-17} we infer that
\begin{equation}\label{A2-18}
\big\|\cW(f)z-\cW(f_0)z-\cD_0(f-f_0,z)\big\|_{\cZ_{\tgamma,\tbeta}}\leq c (\|f\|_{\cZ_{\gamma,\beta}}^p+\|f_0\|_{\cZ_{\gamma,\beta}}^p)\big(\|f_0\|_{\cZ_{\gamma,\beta}}+1\big)\|f-f_0\|_{\cZ_{\gamma,\beta}}^2\|z\|_{\cZ_{\gamma,\beta}}
\end{equation}
for any $f\in\cZ_{\gamma,\beta}(\tbH)$, $z\in\cZ_{\gamma,\beta}(\bH)$, proving Claim 2.

\noindent{\bf Claim 3.} $\cW'$ is continuous from $\cZ_{\gamma,\beta}(\tbH)$ to $\cB_2\big(\cZ_{\gamma,\beta}(\tbH)\times\cZ_{\gamma,\beta}(\bH),\cZ_{\tgamma,\tbeta}(\bH) \big)$. Indeed, from
\eqref{A2-1} and \eqref{A2-12} we infer that
\begin{align}\label{A2-19}
\big\|\big(\cW'(f)-\cW'(f_0)\big)\big(z_1,z_2)\big\|_{L^2_{-\tgamma}}&\leq c(\|f\|_{L^\infty_{-\beta}}^p+\|f_0\|_{L^\infty_{-\beta}}^p)\|f-f_0\|_{L^\infty_{-\beta}}\|z_1\|_{L^\infty_{-\beta}}\|z_2\|_{L^2_{-\gamma}}\nonumber\\
&\leq c(\|f\|_{H^1_{-\beta}}^p+\|f_0\|_{H^1_{-\beta}}^p)\|f-f_0\|_{H^1_{-\beta}}\|z_1\|_{H^1_{-\beta}}\|z_2\|_{\cZ_{\gamma,\beta}}\nonumber\\
&\leq c(\|f\|_{\cZ_{\gamma,\beta}}^p+\|f_0\|_{\cZ_{\gamma,\beta}}^p)\|f-f_0\|_{\cZ_{\gamma,\beta}}\|z_1\|_{\cZ_{\gamma,\beta}}\|z_2\|_{\cZ_{\gamma,\beta}}
\end{align}
for any $f,z_1\in\cZ_{\gamma,\beta}(\tbH)$ and $z_2\in\cZ_{\gamma,\beta}(\bH)$. Using again that $L$ is a $C^\infty$ function from $\tbH$ to $\cB(\bH)$ we obtain that for any $f,z_1\in\cZ_{\gamma,\beta}(\tbH)$ and $z_2\in\cZ_{\gamma,\beta}(\bH)$ the function $\big(\cW'(f)-\cW'(f_0)\big)\big(z_1,z_2)\in H^1_{\mathrm{loc}}(\RR,\bH)$ and
\begin{align}\label{A2-20}
\Big(\big(\cW'(f)&-\cW'(f_0)\big)\big(z_1,z_2)\Big)'(x)=L''(f(x))\big(f'(x),z_1(x),z_2(x)\big)-L''(f_0(x))\big(f_0'(x),z_1(x),z_2(x)\big)\nonumber\\
&+\big(L'(f(x))-L'(f_0(x))\big)\big(z_1'(x),z_2(x)\big)+\big(L'(f(x))-L'(f_0(x))\big)\big(z_1(x),z_2'(x)\big)\nonumber\\
&=H_1(x)+H_2(x)+H_3(x)+H_4(x)
\end{align}
for any $x\in\RR$, where the functions $H_j:\RR\to\bH$, $j=1,2,3,4$, are defined by
\begin{align}\label{A2-21}
H_1(x)&=\big(L''(f(x))-L''(f_0(x))\big)\big(f'(x),z_1(x),z_2(x)\big), \nonumber\\
H_2(x)&=L''(f_0(x))\big(f'(x)-f_0'(x),z_1(x),z_2(x)\big),\nonumber\\
H_3(x)&=\big(L'(f(x))-L'(f_0(x))\big)\big(z_1'(x),z_2(x)\big),\nonumber\\
H_4(x)&=\big(L'(f(x))-L'(f_0(x))\big)\big(z_1(x),z_2'(x)\big).
\end{align}
Since $L$ is a $C^\infty$ function from $\tbH$ to $\cB(\bH)$ we have that
\begin{equation}\label{A2-22}
\big(L''(\tbh_1)-L''(\tbh_2)\big)(\tbh_3,\tbh_4,g)=\Big(\int_0^1L'''\big(s\tbh_1+(1-s)\tbh_2\big)\rmd s\Big)(\tbh_3,\tbh_4,\tbh_1-\tbh_2,g)
\end{equation}
for any $\tbh_1,\tbh_2,\tbh_3,\tbh_4\in\tbH$ and $g\in\bH$. From \eqref{1-2-deriv-bound} it follows that
\begin{equation}\label{A2-23}
\big\|\big(L''(\tbh_1)-L''(\tbh_2)\big)(\tbh_3,\tbh_4,g)\big\|\leq c\big(\|\tbh_1\|^p+\|\tbh_2\|^p\big)\|\tbh_1-\tbh_2\|\,\|\tbh_3\|\,\|\tbh_4\|\,\|g\|
\end{equation}
for any $\tbh_1,\tbh_2,\tbh_3,\tbh_4\in\tbH$ and $g\in\bH$. Below we estimate the $L^2_{-\tbeta}(\RR,\bH)$-norm of the functions $H_j$, $j=1,2,3,4$, using the estimates  \eqref{A2-12} and \eqref{A2-23}.
\begin{align}\label{A2-24}
\|H_1\|_{L^2_{-\tbeta}}&\leq  c(\|f\|_{L^\infty_{-\beta}}^p+\|f_0\|_{L^\infty_{-\beta}}^p)\|f-f_0\|_{L^\infty_{-\beta}}\|f'\|_{L^2_{-\beta}}\|z_1\|_{L^\infty_{-\beta}}\|z_2\|_{L^\infty_{-\beta}}\nonumber\\
&\leq c(\|f\|_{H^1_{-\beta}}^p+\|f_0\|_{H^1_{-\beta}}^p)\|f-f_0\|_{H^1_{-\beta}}\|f\|_{\cZ_{\gamma,\beta}}\|z_1\|_{H^1_{-\beta}}\|z_2\|_{H^1_{-\beta}}\nonumber\\
&\leq c(\|f\|_{\cZ_{\gamma,\beta}}^p+\|f_0\|_{\cZ_{\gamma,\beta}}^p)\|f\|_{\cZ_{\gamma,\beta}}\|f-f_0\|_{\cZ_{\gamma,\beta}}\|z_1\|_{\cZ_{\gamma,\beta}}\|z_2\|_{\cZ_{\gamma,\beta}};
\end{align}
\begin{align}\label{A2-25}
\|H_2\|_{L^2_{-\tbeta}}&\leq  c\|f_0\|_{L^\infty_{-\beta}}^p\|f'-f_0'\|_{L^2_{-\beta}}\|z_1\|_{L^\infty_{-\beta}}\|z_2\|_{L^\infty_{-\beta}}\nonumber\\
&\leq c\|f_0\|_{H^1_{-\beta}}^p\|f-f_0\|_{\cZ_{\gamma,\beta}}\|z_1\|_{H^1_{-\beta}}\|z_2\|_{H^1_{-\beta}}\nonumber\\
&\leq c\|f_0\|_{\cZ_{\gamma,\beta}}^p\|f-f_0\|_{\cZ_{\gamma,\beta}}\|z_1\|_{\cZ_{\gamma,\beta}}\|z_2\|_{\cZ_{\gamma,\beta}};
\end{align}
\begin{align}\label{A2-26}
\|H_3\|_{L^2_{-\tbeta}}&\leq  c(\|f\|_{L^\infty_{-\beta}}^p+\|f_0\|_{L^\infty_{-\beta}}^p)\|f-f_0\|_{L^\infty_{-\beta}}\|z_1'\|_{L^2_{-\beta}}\|z_2\|_{L^\infty_{-\beta}}\nonumber\\
&\leq c(\|f\|_{H^1_{-\beta}}^p+\|f_0\|_{H^1_{-\beta}}^p)\|f-f_0\|_{H^1_{-\beta}}\|z_1\|_{\cZ_{\gamma,\beta}}\|z_2\|_{H^1_{-\beta}}\nonumber\\
&\leq c(\|f\|_{\cZ_{\gamma,\beta}}^p+\|f_0\|_{\cZ_{\gamma,\beta}}^p)\|f-f_0\|_{\cZ_{\gamma,\beta}}\|z_1\|_{\cZ_{\gamma,\beta}}\|z_2\|_{\cZ_{\gamma,\beta}};
\end{align}
\begin{align}\label{A2-27}
\|H_4\|_{L^2_{-\tbeta}}&\leq  c(\|f\|_{L^\infty_{-\beta}}^p+\|f_0\|_{L^\infty_{-\beta}}^p)\|f-f_0\|_{L^\infty_{-\beta}}\|z_1\|_{L^\infty_{-\beta}}\|z_2'\|_{L^1_{-\beta}}\nonumber\\
&\leq c(\|f\|_{H^1_{-\beta}}^p+\|f_0\|_{H^1_{-\beta}}^p)\|f-f_0\|_{H^1_{-\beta}}\|z_1\|_{H^1_{-\beta}}\|z_2\|_{\cZ_{\gamma,\beta}}\nonumber\\
&\leq c(\|f\|_{\cZ_{\gamma,\beta}}^p+\|f_0\|_{\cZ_{\gamma,\beta}}^p)\|f-f_0\|_{\cZ_{\gamma,\beta}}\|z_1\|_{\cZ_{\gamma,\beta}}\|z_2\|_{\cZ_{\gamma,\beta}}.
\end{align}
From \eqref{A2-19}, \eqref{A2-20}, \eqref{A2-24}, \eqref{A2-25}, \eqref{A2-26} and \eqref{A2-27} we conclude that
\begin{equation}\label{A2-28}
\big\|\big(\cW'(f)-\cW'(f_0)\big)\big(z_1,z_2)\big\|_{\cZ_{\tgamma,\tbeta}}\leq c(\|f\|_{\cZ_{\gamma,\beta}}^p+\|f_0\|_{\cZ_{\gamma,\beta}}^p)(1+\|f\|_{\cZ_{\gamma,\beta}})\|f-f_0\|_{\cZ_{\gamma,\beta}}\|z_1\|_{\cZ_{\gamma,\beta}}\|z_2\|_{\cZ_{\gamma,\beta}}
\end{equation}
for any $f,z_1\in\cZ_{\gamma,\beta}(\tbH)$ and $z_2\in\cZ_{\gamma,\beta}(\bH)$, proving Claim 3 and the lemma.
\end{proof}


\end{document}